\documentclass[12pt, reqno]{amsart}
\allowdisplaybreaks[3]
\usepackage[left=3cm,right=3cm,top=2cm,bottom=2cm]{geometry}

\usepackage{amsthm}
\usepackage{amscd}
\usepackage{amsmath}
\usepackage{amsfonts}
\usepackage{amssymb}
\usepackage{graphicx}
\usepackage{amsopn}
\usepackage[usenames]{color}
\usepackage{tikz}
\usetikzlibrary{arrows,shapes,snakes,automata,backgrounds,petri}
\usepackage{textcomp}
\usepackage[all]{xy}

\usepackage[utf8]{inputenc}
\usepackage[T2A]{fontenc}
\usepackage{amsmath,amssymb,amsthm}
\usepackage[english]{babel}

\usepackage{hyperref}

\theoremstyle{theorem}
\newtheorem{theorem}{Theorem}
\newtheorem{corollary}[theorem]{Corollary}
\newtheorem{prop}[theorem]{Proposition}
\newtheorem{lemma}[theorem]{Lemma}

\theoremstyle{definition}
\newtheorem{definition}[theorem]{Definition}

\newtheorem{fact}[theorem]{Fact}
\newtheorem{remark}[theorem]{Remark}

\def\R{\mathcal{R}}

\def\ZZ{\mathcal{Z}}

\def\O{\mathcal{O}}
\def\M{\mathcal{M}}
\def\X{\mathcal{X}}
\def\C{\mathcal{C}}
\def\D{\mathcal{D}}
\def\A{\mathcal{A}}
\def\I{\mathcal{I}}

\def\W{\mathcal{W}}
\def\N{\mathcal{N}}
\def\T{\mathcal{T}}
\def\U{\mathcal{U}}

\def\B{\mathcal{B}}
\def\S{\Sigma}
\def\Mod{{\rm Mod}}
\def\cd{{\rm cd}}
\def\Ind{{\rm Ind}}
\def\Im{{\rm Im}}
\def\PMod{{\rm PMod}}
\def\Stab{{\rm Stab}}

\def\id{{\rm id}}
\def\Z{\mathbb{Z}}
\def\Q{\mathbb{Q}}
\def\E{\widehat{E}}

\def\G{G'}

\def\BB{\mathcal{B}'}

\newcommand{\rk}{\mathop{\mathrm{rk}}}

\def\H{{\rm H}}
\def\Sp{{\rm Sp}}
\def\SL{{\rm SL}}
\def\Homeo{{\rm Homeo}}
\def\pt{{\rm pt}}

\numberwithin{theorem}{section}
\allowdisplaybreaks[3]

\begin{document}

\title{The Top Homology Group of the Genus 3 Torelli Group}

\title{The Top Homology Group of the Genus 3 Torelli Group}
\thanks{The author is partially supported by the HSE University Basic Research Program, the Simons Foundation and the Theoretical Physics and Mathematics Advancement Foundation ``BASIS''. \smallskip }

\address{National Research University Higher School of Economics, Russian Federation}
\address{Skolkovo Institute of Science and Technology, Skolkovo, Russia}
\email{spiridonovia@ya.ru}
\author{Igor A. Spiridonov}

\maketitle

\begin{abstract}
	The Torelli group of a genus $g$ oriented surface $\Sigma_g$ is the subgroup $\mathcal{I}_g$ of the mapping class group ${\rm Mod}(\Sigma_g)$ consisting of all mapping classes that act trivially on ${\rm H}_1(\Sigma_g, \mathbb{Z})$. The quotient group ${\rm Mod}(\Sigma_g) / \mathcal{I}_g$ is isomorphic to the symplectic group ${\rm Sp}(2g, \mathbb{Z})$. 
	The cohomological dimension of the group $\mathcal{I}_g$ equals to $3g-5$. The main goal of the present paper is to compute the top homology group of the Torelli group in the case $g = 3$ as ${\rm Sp}(6, \mathbb{Z})$-module. We prove an isomorphism
	$${\rm H}_4(\mathcal{I}_3, \mathbb{Z}) \cong {\rm Ind}^{{\rm Sp}(6, \mathbb{Z})}_{S_3 \ltimes {\rm SL}(2, \mathbb{Z})^{\times 3}} \mathcal{Z},$$
	where $\mathcal{Z}$ is the quotient of $\mathbb{Z}^3$ by its diagonal subgroup $\mathbb{Z}$ with the natural action of the permutation group $S_3$ (the action of ${\rm SL}(2, \mathbb{Z})^{\times 3}$ is trivial). We also construct an explicit set of generators and relations for the group ${\rm H}_4(\mathcal{I}_3, \mathbb{Z})$.
\end{abstract}

\tableofcontents

\section{Introduction}

Let $\S_{g}$ be a compact oriented genus $g$ surface. Let $\Mod(\S_{g})$ be the \textit{mapping class group} of $\S_{g}$, defined by $\Mod(\S_{g}) = \pi_{0}(\Homeo^{+}(\S_{g}))$, where $\Homeo^{+}(\S_{g})$ is the group of orientation-preserving homeomorphisms of $\S_{g}$. The group $\Mod(\S_{g})$ acts on $\H_{1}(\S_{g}, \Z)$ and preserves the algebraic intersection form, so we have a representation $\Mod(\S_{g}) \rightarrow \Sp(2g, \Z)$, which is well-known to be surjective. The kernel $\I_g$ of this representation is known as the \textit{Torelli group}. This can be written as the short exact sequence 
\begin{equation*}
1 \rightarrow \I_g \rightarrow \Mod(\S_{g}) \rightarrow \Sp(2g, \Z) \rightarrow 1.
\end{equation*}
The \textit{extended Torelli group} $\widehat{\I}_g$ is the preimage of the center $\{\pm 1\}$ of $\Sp(2g, \Z)$. So we have the extension
\begin{equation*}
1 \rightarrow \I_g \rightarrow \widehat{\I}_g \rightarrow \{\pm 1\} \rightarrow 1.
\end{equation*}

If $g=1$ we have $\Mod(\S_1) \cong \SL(2, \Z)$, so the group $\I_1$ is trivial. McCullough and Miller \cite{McCullough} proved, that the group $\I_2$ is not finitely generated. Mess \cite{Mess} showed, that in fact $\I_2$ is an infinitely generated free group. Johnson \cite{Johnson1} proved, that $\I_g$ is finitely generated if $g \geq 3$. However, it is an open problem if the group $\I_g$ is finitely presented for some $g \geq 3$.

The Torelli groups have one more interpretation. Consider the  \textit{Torelli space} $\T_g$, i.e. the module space of smooth complex curves with fixed symplectic basis in the first homology. It is well known that $\T_g$ is an Eilenberg-MacLane space of type $K(\I_g, 1)$ provided that $g \geq 2$. Hence we obtain $\H^{*}(\I_g, \Z) \cong \H^*(\T_g, \Z)$. As the consequence, we have that $\H^{*}(\I_g, \Z)$ corresponds to the set of characteristic classes of homologically trivial surface bundles i.e. surface bundles, such that the fundamental group of the base acts trivially on the homology of the fiber.

A natural problem is to study (co)homology of the groups $\I_g$ in the case $g \geq 3$. The first homology group $\H_1(\I_g, \Z)$ was described explicitly by Johnson \cite{Johnson3}. No other nonzero homology group $\H_k(\I_g, \Z)$ for $k \geq 2$ has been explicitly computed yet. 

Recall that the \textit{cohomological dimension} $\cd(G)$ of a group $G$ is the supremum over all $n$ so that there exists a $G$-module $M$ with $\H^n(G, M) \neq 0$.
Bestvina, Bux and Margalit \cite{Bestvina} in 2007 constructed the contractible complex of cycles $\B_{g}$, on which the Torelli group $\I_g$ acts cellularly. Using the spectral sequence associated with this action, they showed that the group $\I_g$ has cohomological dimension $3g-5$ and that the top homology group $\H_{3g-5}(\I_g, \Z)$ is not finitely generated. Gaifullin \cite{Gaifullin_T} in 2019 proved that for $2g-3 \leq k \leq 3g-5$ the homology group $\H_{k}(\I_g, \Z)$ contains a free abelian subgroup of an infinite rank. In \cite{Gaifullin_T3}, Gaifullin obtained a partial result towards the conjecture that $\H_2(\I_3, \Z)$ is not finitely generated, see Remark \ref{remarkG}. 

Let us fix a complex structure of $\S_g$ so that it becomes a hyperelliptic smooth complex curve. Denote by $\iota \in \Mod(\S_{g})$ the corresponding \textit{hyperelliptic involution}.
We have $\iota^2 = \id$ and $\iota$ acts on $\H_1(\S_g, \Z)$ as $(-1)$. By definition $\iota$ is an element of $\widehat{\I}_g$, not belonging to $\I_g$. Since $\widehat{\I}_g / \I_g \cong \Z/ 2\Z$, there is the natural action of the group $\Z / 2\Z$ on $\H_*(\I_g, \Z)$. This action coincides with the action of the hyperelliptic involution and does not depend on the choice of (hyperelliptic) complex structure on $\S_g$. If 2 is invertible in the coefficient ring $R$, then we obtain the splitting
\begin{equation*}
\H_*(\I_g, R) = \H^+_*(\I_g, R) \oplus \H^-_*(\I_g, R),
\end{equation*}
where a hyperelliptic involution acts trivially on $\H^+_*(\I_g, R)$ and acts as $(-1)$ on $\H^-_*(\I_g, R)$.

In the case $g=3$ there are some special results on the structure of $\H_*(\I_3, \Z)$. 
Hain \cite{Hain_ab} computed explicitly the groups $\H^+_*(\I_3, \Z[1/2])$ as $\Sp(6, \Z)$-modules.
In particular, he proved that 
\begin{equation} \label{Hain}
\H^+_4(\I_3, \Z[1/2]) \cong \Ind^{\Sp(6, \Z)}_{S_3 \ltimes \SL(2, \Z)^{\times 3}} \ZZ \otimes \Z[1/2],
\end{equation}
where $\ZZ$ is the quotient of $\Z^3$ by its diagonal subgroup $\Z$ with the natural action of the permutation group $S_3$ (the action of $\SL(2, \Z)^{\times 3}$ is trivial).
Hain's approach was to use stratified Morse theory for the image of the period map $\T_3 \to \mathfrak{h}_3$, where $\mathfrak{h}_3$ is the upper Siegel half-space.

In the present paper we study the structure of the whole group $\H_4(\I_3, \Z)$ as $\Sp(6, \Z)$-module.
The main result is as follows.
\begin{theorem} \label{mainth}
	There is an isomorphism of $\Sp(6, \Z)$-modules
	\begin{equation*}
	\H_4(\I_3, \Z) = \H_4(\T_3, \Z) \cong \Ind^{\Sp(6, \Z)}_{S_3 \ltimes \SL(2, \Z)^{\times 3}} \ZZ,
	\end{equation*}
	where $\ZZ$ is the quotient of $\Z^3$ by its diagonal subgroup $\Z$ with the natural action of the permutation group $S_3$ (the action of $\SL(2, \Z)^{\times 3}$ is trivial).
\end{theorem}

\begin{corollary}
	The hyperelliptic involution acts trivially on $\H_4(\I_3, \Z)$. In particular, we have
	$\H^-_4(\I_3, \Z[1/2]) = 0.$
\end{corollary}

In Section \ref{S3} we construct an explicit set of generators and relations for the group $\H_4(\I_3, \Z)$. The outline of the present paper is also provided in Section \ref{S3}, after some preliminaries given in Section \ref{S2}.

Let us remark that Theorem \ref{mainth} agrees with Hain's results.
Our main approach is based on the spectral sequence for the action of the Torelli group on the complex of cycles. The proof does not use the Hain's results.

\textbf{Acknowledgements.} The author would like to thank his scientific advisor A. A. Gaifullin for useful discussions and constant attention to this work. The author is a winner of the mathematical August M\"oebius contest of graduate and undergraduate student papers and thanks the jury for the high praise of his work.

\section{Preliminaries} \label{S2}

\subsection{Mapping class group of a surface with punctures and boundary components} Let $\S$ be an oriented surface, possibly with punctures and boundary components. We do not assume that $\S$ is connected. However, we require $\H_*(\S, \Q)$ be a finite dimensional vector space. The mapping class group of $\S$ is defined as $\Mod(\S) = \pi_{0}(\Homeo^{+}(\S, \partial\S))$, where $\Homeo^{+}(\S, \partial \S)$ is the group of orientation-preserving homeomorphisms of $\S$ that act as the identity on $\partial \S$ and preserve the connected components of $\S$. 
By $\PMod(\S) \subseteq \Mod(\S)$ we denote the \textit{pure mapping class group} of $\S$, i.e. the subgroup that acts trivially on punctures. We have the exact sequence
\begin{equation} \label{Pmod}
1 \rightarrow \PMod(\S_{g, n}^b) \rightarrow \Mod(\S_{g, n}^b) \rightarrow S_n \rightarrow 1,
\end{equation}
where by $\S_{g, n}^b$ we denote the connected genus $g$ surface with $n$ punctures and $b$ boundary components.
In the case $\S = \S_{g, 1}$ we can also define the Torelli group $\I(\S_{g, 1}) = \I_{g, 1}$ as the kernel of the action of $\Mod(\S_{g, 1})$ on $\H_1(\S_{g, 1}, \Z)$. 

\subsection{Hochschild-Serre spectral sequence} 
Given a short exact sequence of groups
\begin{equation*}
1 \to Q \to G \to P \to 1
\end{equation*}
there is an associated Hochschild-Serre spectral sequence. The second page is given by
\begin{equation} \label{HS}
E_{p, q}^{2} \cong \H_{p}(P, H_{q}(Q, \Z)) \Rightarrow \H_{p+q}(G, \Z),
\end{equation}
where the coefficients are local: $P$ acts on $Q$ by conjugation. The group
$
\bigoplus_{p+q=n}E_{p, q}^{\infty}
$
is the adjoint graded group for certain filtration in the homology group $\H_{p+q}(G, \Z)$.

The following facts immediately follow from the existence of the Hochschild-Serre spectral sequence.

\begin{fact}\label{HS1}
	Consider a short exact sequence of groups
	\begin{equation*}
	1 \to Q \to G \to P \to 1
	\end{equation*}
	with $\cd(P) = p < \infty$ and $\cd(Q) = q < \infty$. Then all the differentials from and to the group $E^2_{p, q}$ are trivial, so
	\begin{equation*}
	\H_{p+q}(G, \Z) \cong E_{p, q}^\infty = E_{p, q}^2 =  \H_{p}(P, H_{q}(Q, \Z)),
	\end{equation*}
	where the coefficients are local: $P$ acts on $Q$ by conjugation.
\end{fact}

Recall that for $n$ pairwise commuting elements $h_{1}, \dots, h_{n}$ of the group $G$ one can construct an \textit{abelian cycle} $\A(h_{1}, \dots, h_{n}) \in \H_{n}(G, \Z)$ defined in the following way. Consider the homomorphism $\phi: \Z^{n} \rightarrow G$ that maps the generator of the $i$-th factor to the $h_{i}$. Then $\A(h_{1}, \dots, h_{n}) = \phi_{*}(\mu_{n})$, where $\mu_{n}$ is the standard generator of $\H_{n}(\Z^{n}, \Z)$. Fact \ref{HS1} implies the following results.

\begin{fact}\label{HS2}
	Consider a central extension
	\begin{equation} \label{Cext}
	1 \to \Z \to G \to P \to 1
	\end{equation}
	with $\cd(P) = p < \infty$ and $\Z = \left\langle a \right\rangle $. Then
	\begin{equation} \label{HS3}
	\H_{p}(P, \Z) \cong \H_{p+1}(G, \Z). 
	\end{equation}
	Moreover, for $p$ pairwise commuting elements $h_{1}, \dots, h_{p}$ of the group $P$, the isomorphism (\ref{HS3}) maps the abelian cycle $\A(h_1, \dots, h_p)$ to the abelian cycle $\A(\widetilde{h}_1, \dots, \widetilde{h}_p, a)$, where $\widetilde{h}_i \in G$ denotes any preimage of $h_i$.
\end{fact}

The last statement in Fact \ref{HS2} comes from the functoriality of Hochschild-Serre spectral sequence applied to the central extension 
\begin{equation} \label{Zext}
1 \to \Z \to \Z^{n+1} \to \Z^n \to 1
\end{equation}
and the morphism from (\ref{Zext}) to (\ref{Cext}) which is identical on $\Z$ and maps the $i^{{\rm th}}$ generator of $\Z^n$ to $h_i$, because in the case (\ref{Zext}) Fact \ref{HS2} is obvious.

Let $G$ be a group. For $h \in G$ we denote by $[h] \in \H_1(G, \Z)$ the corresponding homology class.

\begin{fact}\label{HS4}
	Consider a short exact sequence
	\begin{equation*}
	\begin{CD}
	1 @>>>  &   R_1  @>{\iota}>> &   G @>{p}>> &   R_2 @>>> &  1
	\end{CD}
	\end{equation*}
	where $R_1$ and $R_2$ are free groups. Assume that the action of $R_2$ on $\H_*(R_1, \Z)$ is trivial. Then
	\begin{equation} \label{HS5}	
	\H_{2}(G, \Z) \cong \H_{1}(R_2, \H_1(R_1, \Z)) \cong \H_1(R_2, \Z) \otimes \H_1(R_1, \Z).
	\end{equation}
	Let $f \in R_1$ and $g \in G$ such that $\iota(f)$ commutes with $g$. Then the isomorphism (\ref{HS5}) maps the abelian cycle $\A(g, \iota(f))$ to $[p(g)] \otimes [f] \in \H_1(R_2, \Z) \otimes \H_1(R_1, \Z)$.
\end{fact}

\subsection{Complex of cycles} 

A \textit{multicurve} $M$ on $\S_g$ is a finite union of pairwise disjoint and pairwise nonhomotopic simple closed curves $M = \gamma_1 \cup \dots \cup \gamma_n$, such that no $\gamma_i$ is homotopic to a point. A multicurve $M$ is said to be \textit{oriented} if each its component is equipped with an orientation. Throughout the paper we consider simple closed curves and multicurves up to isotopy.

Bestvina, Bux, and Margalit \cite{Bestvina} constructed a contractible $CW$-complex $\B_g$ called \textit{complex of cycles} on which the Torelli group acts without rotations. ``Without rotations'' means that if an element  $h \in \I_g$ stabilizes a cell $\sigma$, then $h$ stabilizes $\sigma$ pointwise. Let us recall the construction of $\B_g$. More details can be found in \cite{Bestvina, Hatcher, Gaifullin_T, Gaifullin_T3}.

Let us denote by $\C$ the set of all isotopy classes of oriented non-separating simple closed curves on $\S_{g}$. Fix any element $0 \neq x \in \H_1(\S_{g}, \Z)$. The construction of $\B_g = \B_g(x)$ depends on the choice of the homology class $x$, however the $CW$-complexes $\B_g(x)$ are pairwise homeomorphic for different $x$. 

A \textit{Basic 1-cycle} for the homology class $x$ is a formal linear combination $\gamma = \sum_{i}^n k_i  \gamma_i$ where $\gamma_i \in \C$ and $0 < k_i \in \mathbb{N}$ satisfying the following properties:

(1) the homology classes $[\gamma_1], \dots, [\gamma_n]$ are linearly independent,

(2) $\sum_i^n k_i [\gamma_i] = x$,

(3) we can choose pairwise disjoint representatives of the isotopy classes $\gamma_1, \dots, \gamma_n$.

The multicurve $\gamma_{1} \cup \dots \cup \gamma_{n}$ is called the \textit{support} of $\gamma$. 

Let us denote by $\M(x)$ the set of oriented multicurves $M = \gamma_{1} \cup \dots \cup \gamma_s$ (for arbitrary $s$) satisfying the following properties:

(i) no nontrivial linear combination of the homology classes $[\gamma_1], \dots, [\gamma_s]$ with nonnegative coefficients equals zero,

(ii) for each $1 \leq i \leq s$ there exists a basic 1-cycle for $x$ supported in $M$ and containing $\gamma_i$.

For each $M \in \M(x)$ let us denote by $P_M \subset \mathbb{R}_{+}^{\C}$ the convex hull of the basic 1-cycles supported in $M$. By construction $P_M$ is a convex polytope. By definition the \textit{complex of cycles} is the regular $CW$-complex given by $\B_g(x) = \cup_{M \in \M(x)} P_M$. 

A \textit{1-cycle} for the homology class $x \in \H_1(\S_g, \Z)$ is a formal linear combination $\gamma = \sum_{i}^n k_i  \gamma_i$ where $\gamma_i \in \C$ and $k_i \in \mathbb{R}_+$ satisfying the properties (2) and (3). The multicurve $\gamma_{1} \cup \dots \cup \gamma_{n}$ is called the \textit{support} of $\gamma$. By definition the set of $1$-cycles for $x$ is precisely the set points of $\B_{g}(x)$.
Therefore, an oriented multicurve $M$ belongs to $\M(x)$ if and only if it is the support of some $1$-cycle $\gamma$ for $x$. Moreover, for each $M \in \M(x)$, the set of vertices of $P_M$ is precisely the set of basic $1$-cycles for $x$ supported in $M$. Bestvina, Bux, and Margalit \cite[Lemma 2.1]{Bestvina} showed that
\begin{equation} \label{dim}
\dim P_M = |M| - \rk  M = |\S_g \setminus M| - 1,
\end{equation}
where $|M|$ is the number of components of $M$, $|\S_g \setminus M|$ is the number of connected components of $|\S_g \setminus M|$, and $\rk M$ is the rank of the subgroup of $\H_{1}(\S_g, \Z)$ spanned by the homology classes of the
components of $M$. Consequently, we have $\dim \B_{g} = 2g-3$.
By $\M_p(x) \subseteq \M(x)$ we denote the set of multicurves corresponding to the cells of dimension $p$.

\begin{theorem} \cite[Theorem E]{Bestvina}\label{contrcyc}
	Let $g \geq 2$ and $0 \neq x \in \H_{1}(\S_g, \Z)$. Then $\B_g(x)$ is contractible.
\end{theorem}

\subsection{The spectral sequence for a group action on a CW-complex}

Suppose that a group $G$ acts cellularly without rotations on a contractible $CW$-complex $X$. 
Let $C_*(X, \Z)$ be the cellular chain complex of $X$ and $\R_*$ be a projective resolution for $\Z$ over $\Z G$. Consider the double complex  $B_{p, q} = C_p(X, \Z) \otimes_G \R_q$ with the filtration by columns. The corresponding spectral sequence (see (7.7) in \cite[Section VII.7]{Brown}) has the form
\begin{equation} \label{spec_sec}
E_{p, q}^1 \cong \bigoplus_{\sigma \in \X_p}\H_q(\Stab_G (\sigma), \Z) \Rightarrow \H_{p+q}(G, \Z),
\end{equation}
where $\X_p$ is a set containing one representative from each $G$-orbit of $p$-cells of $X$. The group
$
\bigoplus_{p+q=n}E_{p, q}^{\infty}
$
is the adjoint graded group for a filtration in the homology group $H_{p+q}(G, \Z)$. Note that for an arbitrary $CW$-complex $X$ the spectral sequence (\ref{spec_sec}) converges to the equivariant homology $\H^G_{p+q}(X, \Z)$. So for a contractible $CW$-complex $X$ we have $\H^G_{p+q}(X, \Z) \cong \H_{p+q}(G, \Z)$.

Now let $(E_{*,*}^*, d_{*, *}^*)$ be the spectral sequence (\ref{spec_sec}) for the action on $\I_g$ on $\B_g(x)$ for some primitive element $0 \neq x \in \H_{1}(\S_g, \Z)$.
The fact that $\I_g$ acts on $\B_g(x)$ without rotations follows from the result of Ivanov \cite[Theorem 1.2]{Ivanov}: if an element $h \in \I_g$ stabilises some multicurve $M$ the $h$ stabilises each component of $M$.
We have the spectral sequence 
\begin{equation} \label{spec_sec1}
E_{p, q}^1 \cong \bigoplus_{M \in \M_p(x) / \I_g}\H_q(\Stab_{\I_g} (M), \Z) \Rightarrow \H_{p+q}(\I_g, \Z).
\end{equation}
(For a group $G$ acting on a set $X$ we denote by $X / G$ any set containing one representative from each $G$-orbit in the set $X$.) Let us introduce some notation. Let $P \subseteq \B_g(x)$ be a $p$-cell and $h \in \H_q(\I_g, \Z)$ be a homology class. By $P \otimes h \in E_{p, q}^1$ we denote the element that maps to $h \in \H_q(\Stab_{\I_g} (P), \Z)$ under the isomorphism (\ref{spec_sec1}). This notation is convenient because the term $E_{p, q}^0$ is defined as the tensor product. The differential $d_{p, q}^1$ has the form $$d_{p, q}^1 (P \otimes h) = \partial P \otimes h \in E_{p-1, q}^1,$$
where $h$ in the right hand side denotes the images of $h$ under the mappings induced by the inclusions $\Stab_{\I_g} (P) \hookrightarrow \Stab_{\I_g} (Q)$ for all cells $Q \subseteq \partial P$.

Bestvina, Bux and Margalit proved \cite[Proposition 6.2]{Bestvina} that for each cell $\sigma \in \B_g(x)$ we have
\begin{equation} \label{ineq}
\dim(\sigma) + \cd(\Stab_{\I_g}(\sigma)) \leq 3g-5.
\end{equation}
Formulas (\ref{spec_sec1}) and (\ref{ineq}) immediately imply the following fact.
\begin{corollary} \label{cor}
	Let $E_{*,*}^*$ be the spectral sequence (\ref{spec_sec1}).
	Then $E^1_{p, q} = 0$ for $p+q > 3g-5$.
\end{corollary}

\subsection{Stabilisers of multicurves}

We denote by $T_{\gamma}$ the left Dehn twist about a simple closed curve $\gamma$.
Let $M$ be a multicurve on $\S_g$. We denote by $\Stab_{\Mod(\S_{g})}(\overrightarrow{M}) \subseteq \Mod(\S_g)$ the subgroup consisting of all mapping classes that stabilize every component of $M$ and preserve the orientation of every component
of $M$. Then there is the following Birman-Lubotzky–McCarthy exact sequence (see \cite[Lemma 2.1]{Lubotzky})
\begin{equation} \label{LB}
1 \rightarrow G(M) \rightarrow \Stab_{\Mod(\S_{g})}(\overrightarrow{M}) \rightarrow \PMod(\S_{g} \setminus M) \rightarrow 1,
\end{equation}
where $G(M)$ is the group generated by Dehn twists about the components of $M$.

There is the a version of the sequence (\ref{LB}) for the Torelli group, see \cite[Section 6.2]{Bestvina}. Let $M$ be a multicurve on $\S_g$ without separating components. Then we have the exact sequence
\begin{equation} \label{LB2}
1 \rightarrow \Z^{BP(M)} \rightarrow \Stab_{\I_g}(M) \rightarrow \PMod(\S_{g} \setminus M),
\end{equation}
where $BP(M)$ is the number of curves of $M$ minus the
number of distinct homology classes represented by the curves of $M$. The group $\Z^{BP(M)}$ is generated by the twists about bounding pairs contained in $M$. Recall that a twist about bounding pair is a map $T_{\delta'}^{-1}T_{\delta}$, where $\delta$ and $\delta'$ are disjoint curves representing the same homology class. Moreover, we have the following inequality (\cite[Lemma 6.13]{Bestvina}).
\begin{equation} \label{cd}
\cd(\Stab_{\I_g} M) \leq 3g- 3 - P(M) - |M| + BP(M),
\end{equation}
where $P(M)$ is the number of positive genus components of $\S_g \setminus M$.

Also, we will need the following proposition.
\begin{prop} \cite[Theorem 4.6]{Primer} \label{Birman} 
	\textbf{(Birman exact sequence).} Let $\S$ be a surface with $\chi(\S) < 0$. Denote by $\S'$ the surface obtained from $\S$ by removing a point $x$ in the interior of $\S$. Then the following sequence is exact
	\begin{equation}\label{Birman1}
	1 \to \pi_1(\S, x) \to \PMod(\S') \to \PMod(\S) \to 1.
	\end{equation}
\end{prop}
Here the map $\PMod(\S') \to \PMod(\S)$ is obtained by “forgetting” the puncture $x$, and the image of an element of $\pi_1(\S, x)$ in $\PMod(\S')$ is realized by “pushing” the puncture $x$ around that element of $\pi_1(\S, x)$. There is a version of the Birman exact sequence for the Torelli group \cite[Proposition 6.13]{Primer}:
\begin{equation}\label{BirmanT}
1 \to \pi_1(\S_g, \pt) \to \I_{g, 1} \to \I_g \to 1.
\end{equation}

\section{Strategy of the proof} \label{S3}

From now throughout the paper we fix $g=3$ and put $\S = \S_3$, $\I = \I_3$, $\B(x) = \B_3(x)$. For a group $G$ acting on a set $Y$ we denote $G_y = \Stab_G(y)$ for $y \in Y$. In particular, we denote by $\I_M$ the stabiliser of a (multi)curve $M$ in $\I$. Also we set $\H = \H_1(\S, \Z)$. We say that a subgroup $U \subset \H$ is \textit{symplectic} if the restriction of the intersection form on $U$ has determinant $1$. We say that $(V_1, V_2, V_3)$ is a (ordered) \textit{splitting} of $\H$, if $\H = V_1 \oplus V_2 \oplus V_3$ where $V_1$, $V_2$ and $V_3$ are symplectic subgroups of rank $2$. For $U \subseteq \H$ we denote by $U^{\perp} \subseteq \H$ the orthogonal subgroup with respect to the intersection form.

For a separating curve $\theta$ we denote by $X_\theta$ the once-punctured torus bounded by $\theta$, and denote by $\overline{X}_\theta$ its closure. Also we denote $\H_\theta = \H_1(X_\theta, \Z) \subset \H$.
For a multicurve $M$ we say that a symplectic subgroup $U \subset \H$ is \textit{admissible}, if $\rk U = 2$ and $U = \H_\theta$ for some separating curve $\theta$ disjoint from $M$. Similarly, we say that a splitting $V = (V_1, V_2, V_3)$ of $\H$ is \textit{admissible}, if $V$ can be obtained via two separating curves disjoint from $M$. 

\subsection{The construction of s-classes}

Now let us describe explicitly the set of generators of the group $\H_4(\I_3, \Z)$, which we will call \textit{s-classes}. For a splitting $(V_1, V_2, V_3)$ let us now define the correspondent s-class $s(V_1, V_2, V_3) \in \H_4(\I_3, \Z)$ in the following way.

Consider any two disjoint separating curves $\gamma$ and $\delta$ on $\S_3$ bounding punctured tori $X_\gamma$ and $X_\delta$ respectively (see Fig. \ref{S}), such that $\H_\gamma = \H_1(X_\gamma, \Z) = V_1$ and $\H_\delta = \H_1(X_\delta, \Z) = V_3$ (this implies that $\H_1(Y, \Z) = V_2$, where $Y$ is the third connected component of $\S_3 \setminus \{\gamma, \delta\}$). 

\begin{figure}[h]
	\begin{center}
		\scalebox{0.5}{
			\begin{tikzpicture}
			\draw[red, very thick, dashed] (2,3) to[out = -100, in = 100] (2, -3);
			\draw[red, very thick, dashed] (-2,3) to[out = -100, in = 100] (-2, -3);
			
			\draw[very thick] (-4, 3) to (4, 3);
			\draw[very thick] (-4, -3) to (4, -3);
			\draw[very thick] (-4, 3) arc (90:270:3);
			\draw[very thick] (4, 3) arc (90:-90:3);
			\draw[very thick] (0, 0) circle (1);
			\draw[very thick] (-4, 0) circle (1);
			\draw[very thick] (4, 0) circle (1);

			\draw[red, very thick] (2,3) to[out = -80, in = 80] (2, -3);
			\draw[red, very thick] (-2,3) to[out = -80, in = 80] (-2, -3);
			
			\node[red, scale = 2] at (2.7, 2) {$\delta$};
			\node[red, scale = 2] at (2.7-4, 2) {$\gamma$};
			
			\node[scale = 3] at (3.5, -2) {$X_\delta$};
			\node[scale = 3] at (-3.5, -2) {$X_\gamma$};
			\node[scale = 3] at (0, -2) {$Y$};

			\end{tikzpicture}} \end{center}
	\caption{}
	\label{S}
\end{figure}
 
Consider the group $\I_\gamma =\Stab_{\I} (\gamma)$. The exact sequence (\ref{LB}) implies that we have
\begin{equation}\label{eq3}
1 \to \left\langle T_\gamma \right\rangle  \to \I_\gamma \to \PMod(\S_{1, 1}) \times \PMod(\S_{2, 1}).
\end{equation}
Since the group $\I_{1, 1}$ is trivial, one can easily compute the image of $\I_\gamma$ is precisely $\I_{2, 1}$.
\begin{equation}\label{eq4}
1 \to \left\langle T_\gamma \right\rangle  \to \I_\gamma \to \I_{2, 1} \to 1.
\end{equation}
Also let us consider the exact sequence (\ref{BirmanT}) in the case $g=2$:
\begin{equation}\label{BirmanT2}
1 \to \pi_1(\S_2, \pt) \to \I_{2, 1} \to \I_2 \to 1.
\end{equation}
The groups $\pi_1(\S_g, \pt)$ and $\I_2$ are generated by Dehn twists and bounding pair maps disjoint from $\gamma$, therefore $\left\langle T_\gamma \right\rangle $ belongs to the center of $\I_\gamma$. Consequently, $\I_{2, 1}$ acts trivially on $\H_1(\left\langle T_\gamma \right\rangle , \Z)$. Hence Fact \ref{HS1} applied to the exact sequence (\ref{eq4}) yields an isomorphism 
$$\alpha_{\gamma}: \H_{3}(\I_{2, 1}, \Z) \cong \H_4(\I_\gamma, \Z)$$
which is given by the Gysin map in this case.

The group $\I_{2}$ acts trivially on $\H_2(\pi_1(\S_2, \pt), \Z) = \H_2(\S_2, \Z) \cong \Z$.
Hence Fact \ref{HS1} applied to the exact sequence (\ref{BirmanT2}) yields an isomorphism $$\beta_{\gamma}: \H_{1}(\I_{2}, \Z) \cong \H_3(\I_{2, 1}, \Z).$$ Therefore we have
\begin{equation}\label{iso1}
\alpha_{\gamma} \circ \beta_{\gamma}: \H_{1}(\I_{2}, \Z)  \cong \H_4(\I_\gamma, \Z).
\end{equation}
Denote $$s(V_2, V_3) = \alpha_{\gamma} \circ \beta_{\gamma} ([T_\delta]) \in \H_4(\I_\gamma, \Z),$$
we call these homology classes \textit{initial s-classes}.

Mess \cite{Mess} proved that the group $\I_2$ is freely generated by an infinite number of Dehn twists about separating curves. Moreover, $\I_2$ has one generator for each splitting of the group $\H_1(\S_2, \Z)$ into direct of two symplectic subgroups of rank $2$. Hence the isomorphism (\ref{iso1}) yields $s(V_2, V_3) = s(V_3, V_2)$. Moreover, we obtain the following result.
\begin{prop} \label{pre}
	Let $\gamma$ be a separating curve on $\S$ bounding once-puncture torus with the first homology group $V_1 \subset \H$.
	Then $\H_4(\I_\gamma, \Z)$ is a free abelian group with the basis consisting of the elements $s(V_2, V_3) = s(V_3, V_2)$ for all splittings of $\H$ of the form $(V_1, V_2, V_3)$, considered up to interchanging $V_2$ and $V_3$.
\end{prop}

Now let as define the \textit{s-class} by
$$s(V_1, V_2, V_3) = (\iota_{\gamma})_* (s(V_2, V_3)) \in \H_4(\I_3, \Z),$$
where by $\iota_{\gamma}$ we denote the inclusion $\iota_{\gamma}: \I_\gamma \hookrightarrow \I$.

Note that homology class $s(V_1, V_2, V_3) \in \H_4(\I, \Z)$ does not depend on the choice of $\gamma$ and $\delta$. This follows from the fact that all such pairs of curves are $\I$-equivalent. Proposition \ref{pre} implies the following result.

\begin{prop} \label{relation1}
	For a splitting $(V_1, V_2, V_3)$ of $\H$ we have 
	\begin{equation}\label{rel1}
	s(V_1, V_2, V_3) = s(V_1, V_3, V_2).
	\end{equation}
\end{prop}
However, we will prove that the classes $s(V_1, V_2, V_3)$ does depend on which of  subgroups is listed first.

\subsection{The spectral sequences for the complex of cycles}

Now let $(E_{*,*}^*, d_{*, *}^*)$ be the spectral sequence (\ref{spec_sec}) for the action on $\I$ on $\B(x)$ for some primitive element $0 \neq x \in \H$. By Corollary \ref{cor} we have $E^1_{p, q} = 0$ for $p+q > 4$. Since $\dim \B(x) = 3$ it follows that $E^1_{4, 0} = 0$. Hence all nonzero terms of the page $E^1$ are shown on the left in Fig. \ref{E1}.

The following two propositions, which will be proved in Sections \ref{S4} and \ref{S5}, imply that all nonzero terms of the page $E^2$ are shown on the right in Fig. \ref{E1}.
\begin{prop} \label{diff31}
	The differential $d^1_{3, 1}: E^1_{3, 1} \to E^1_{2, 1}$ is injective.
\end{prop}
\begin{prop} \label{diff22}
	The differential $d^1_{2, 2}: E^1_{2, 2} \to E^1_{1, 2}$ is injective.
\end{prop}
Proposition \ref{diff31} is not hard. However, the proof of Proposition \ref{diff22} is quite complicated, and is one of the central results of the whole work.

\begin{figure}[h]
	\begin{minipage}[h]{0.49\linewidth}
		\begin{center}
			\scalebox{1}{
				\begin{tikzpicture}
				\draw[->] (0, 0) to (0, 6);
				\draw[] (1, 0) to (1, 5);
				\draw[] (2, 0) to (2, 4);
				\draw[] (3, 0) to (3, 3);
				\draw[] (4, 0) to (4, 2);
				
				\draw[->] (0, 0) to (5, 0);
				\draw[] (0, 1) to (4, 1);
				\draw[] (0, 2) to (4, 2);
				\draw[] (0, 3) to (3, 3);
				\draw[] (0, 4) to (2, 4);
				\draw[] (0, 5) to (1, 5);
				
				\draw[->] (3.5, 1.5) to (2.5, 1.5);
				\draw[] (3.5, 1.5) arc (-90:90:0.1);
				
				\draw[->] (2.5, 2.5) to (1.5, 2.5);
				\draw[] (2.5, 2.5) arc (-90:90:0.1);
				
				\node[scale = 1][left] at (0, 0.5) {$0$};
				\node[scale = 1][left] at (0, 1.5) {$1$};
				\node[scale = 1][left] at (0, 2.5) {$2$};
				\node[scale = 1][left] at (0, 3.5) {$3$};
				\node[scale = 1][left] at (0, 4.5) {$4$};
				
				\node[scale = 1][below] at (0.5, 0) {$0$};
				\node[scale = 1][below] at (1.5, 0) {$1$};
				\node[scale = 1][below] at (2.5, 0) {$2$};
				\node[scale = 1][below] at (3.5, 0) {$3$};
				
				\node[scale = 1][left] at (0, 5.5) {$q$};
				\node[scale = 1][below] at (4.5, 0) {$p$};
				
				\end{tikzpicture}} \\  $E^1$ \end{center}
	\end{minipage}
	\hfill
	\begin{minipage}[h]{0.49\linewidth}
		\begin{center}
			\scalebox{1}{
				\begin{tikzpicture}
				\draw[->] (0, 0) to (0, 6);
				\draw[] (1, 0) to (1, 5);
				\draw[] (2, 0) to (2, 4);
				\draw[] (3, 0) to (3, 2);
				\draw[] (4, 0) to (4, 1);
				
				\draw[->] (0, 0) to (5, 0);
				\draw[] (0, 1) to (4, 1);
				\draw[] (0, 2) to (3, 2);
				\draw[] (0, 3) to (2, 3);
				\draw[] (0, 4) to (2, 4);
				\draw[] (0, 5) to (1, 5);
				
				\node[scale = 1][left] at (0, 0.5) {$0$};
				\node[scale = 1][left] at (0, 1.5) {$1$};
				\node[scale = 1][left] at (0, 2.5) {$2$};
				\node[scale = 1][left] at (0, 3.5) {$3$};
				\node[scale = 1][left] at (0, 4.5) {$4$};
				
				\node[scale = 1][below] at (0.5, 0) {$0$};
				\node[scale = 1][below] at (1.5, 0) {$1$};
				\node[scale = 1][below] at (2.5, 0) {$2$};
				\node[scale = 1][below] at (3.5, 0) {$3$};

				\node[scale = 1][left] at (0, 5.5) {$q$};
				\node[scale = 1][below] at (4.5, 0) {$p$};
				
				\end{tikzpicture}} \\  $E^2$ \end{center}
	\end{minipage}
	\caption{The pages $E^1$ and $E^2$ and the differentials $d^1_{3, 1}$ and $d^1_{2, 2}$.}
	\label{E1}
\end{figure}
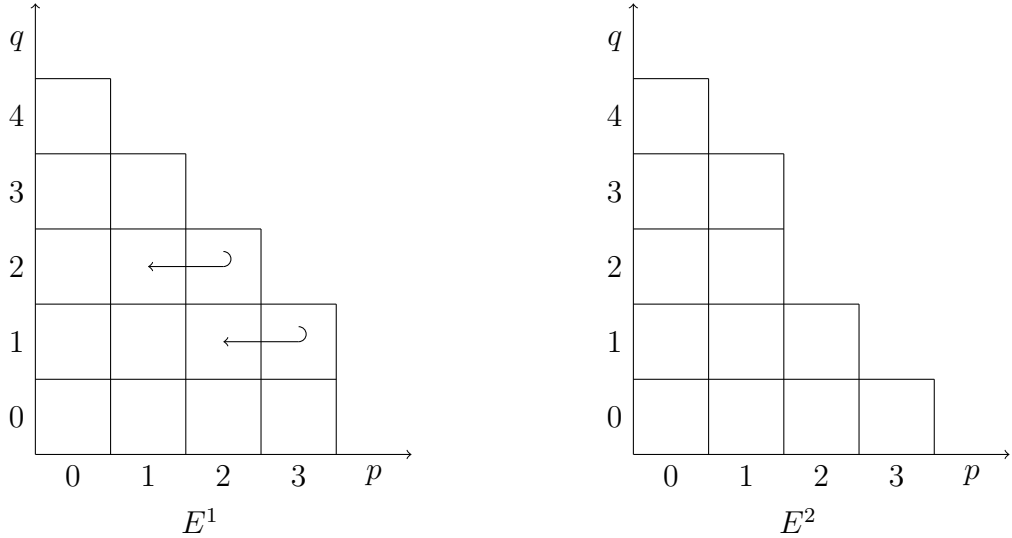

\begin{corollary} \label{corEinf}
	We have the exact sequence
	$$1 \to E^1_{0, 4} \to \H_4(\I, \Z) \to E^2_{1, 3} \to 1.$$
\end{corollary}
\begin{proof}
	All differentials $d^2, d^3, \dots$ from and to the group $E^2_{1, 3}$ are trival, so $E^\infty_{1, 3} = E^2_{1, 3}$. Also we obviously have $E^\infty_{0, 4} = E^1_{0, 4}$. Propositions \ref{diff31} and \ref{diff22} imply that the page $E^2$ has the form shown in Fig. \ref{E1}. Therefore the page $E^\infty$ has the same form. The exact sequence 
	$$1 \to E^\infty_{0, 4} \to \H_4(\I, \Z) \to E^\infty_{1, 3} \to 1$$
	implies the result.
\end{proof}

\begin{remark} \label{remarkG}
	Gaifullin \cite{Gaifullin_T3} studied the same spectral sequence (also in the case $g = 3$) in order to obtain a partial result towards the conjecture that $\H_2(\I_3, \Z)$ is not finitely generated. He proved the the group $E_{0, 2}^3$ is not finitely generated, that is, the group $E_{0, 2}^1$ remains infinitely generated after taking quotients by images of the differentials $d_{1, 2}^1$ and $d_{2, 1}^2$. If one give a proof that it also remains infinitely generated after taking quotient by the image of $d_{0, 3}^3$, this would imply that $\I_3$ is not finitely presented.
\end{remark}

\subsection{The morphism of the spectral sequences}

For each separating curve $\gamma$ on $\S$ denote by $(E^{(\gamma) *}_{*, *}, d_{*, *}^{(\gamma) *})$ the spectral sequence (\ref{spec_sec}) for the action of $\I_{\gamma}$ on $\B(x)$.
Denote by $j^{(\gamma) *}_{*, *}: E^{(\gamma) *}_{*, *} \to E^*_{*, *}$ the morphism of the spectral sequences induced by the inclusion $\iota_{\gamma}: \I_{\gamma} \hookrightarrow \I$. 

Consider the morphism
\begin{equation} \label{morspec}
J_{*, *}^*: \bigoplus_{\gamma} E^{(\gamma) *}_{*, *} \to E^*_{*, *},
\end{equation}
where $J_{*, *}^*$ is induces by $j^{(\gamma) *}_{*, *}$, and denote $\E^*_{*, *} = \bigoplus_{\gamma} E^{(\gamma) *}_{*, *}$, where the sums are over all separating curves $\gamma$ on $\S$. The following two propositions, which will be proved in Sections \ref{S6} and \ref{S7}, describe some properties of the morphism $J_{*, *}^*: \E^*_{*, *} \to E^*_{*, *}$.

\begin{prop} \label{surj1}
	The map $J^2_{1, 3}: \E^2_{1, 3} \to E^2_{1, 3}$ is surjective.
\end{prop}

\begin{prop} \label{surj0}
	The map $J^1_{0, 4}: \E^1_{0, 4} \to E^1_{0, 4}$ is surjective.
\end{prop}

\begin{corollary} \label{gener}
	The set of all s-classes generates the group $\H_4(\I_3, \Z)$.
\end{corollary}

\begin{proof}
	Let $\gamma$ be a separating curve on $\S_g$.
	Proposition \ref{pre} implies that $\H_4(\I_{\gamma}, \Z)$ is generated by initial s-classes. Therefore the group $\E^1_{0, 4}$ is generated by linear combinations of initial s-classes and the group $\E^2_{1, 3} \subseteq \H_4(\I_{\gamma}, \Z) / \E^1_{0, 4}$ is generated by cosets containing initial s-classes. Hence Propositions \ref{surj0} implies that $E^1_{0, 4}$ is generated by linear combinations of s-classes. Proposition \ref{surj1} implies that $E^2_{1, 3} = \H_4(\I_3, \Z) / E^1_{0, 4}$ is generated by cosets containing s-classes. Corollary \ref{corEinf} concludes the proof.
\end{proof}

Using the morphism (\ref{morspec}) we will also prove the following result in Section \ref{S8}.

\begin{theorem} \label{th2}
	For any splitting $(V_1, V_2, V_3)$ we have
	\begin{equation}\label{rel2}
	s(V_1, V_2, V_3) + s(V_2, V_3, V_1) + s(V_3, V_1, V_2) = 0.
	\end{equation}
	Any linear relation between s-classes follows from (\ref{rel1}) and (\ref{rel2}).
\end{theorem}

Now let us deduce the main result from Corollary \ref{gener} and Theorem \ref{th2}.

\begin{proof}[Proof of Theorem \ref{mainth}.] Obviously the action of $\Sp(6, \Z)$ on the s-classes has the form $h \cdot s(V_1, V_2, V_3) = s(h V_1, h V_2, h V_3)$. The stabiliser in $\Sp(6, \Z)$ of an unordered splitting $\H = V_1 \oplus V_2 \oplus V_3$ is isomorphic to $S_3 \ltimes \SL(2, \Z)^{\times 3}$. The $S_3 \ltimes \SL(2, \Z)^{\times 3}$-module, formally generated by the six s-classes corresponding to six permutations of $V_1, V_2, V_3$ in the splitting $(V_1, V_2, V_3)$, satisfying the relations (\ref{rel1}) and (\ref{rel2}), is isomorphic to $\ZZ$. Therefore Corollary \ref{gener} together with the last statement of Theorem \ref{th2} imply that
\begin{equation*}
\H_4(\I_3, \Z) \cong \ZZ \otimes_{S_3 \ltimes \SL(2, \Z)^{\times 3}} \Z[\Sp(6, \Z)] = \Ind^{\Sp(6, \Z)}_{S_3 \ltimes \SL(2, \Z)^{\times 3}} \ZZ.
\end{equation*}
This concludes the proof.
\end{proof}

This paper is organized as follows. In Section \ref{S4} we compute explicitly the group $E^1_{3, 1}$ and give a straightforward proof of Proposition \ref{diff31}. The proof of Proposition \ref{diff22} is more complicated and require more computations, because we need some additional information about the structure of the group $E^1_{1, 2}$. This proof is given in Section \ref{S5}. Section \ref{S6} contains a proof of Proposition \ref{surj1}. The key idea is to use the \textit{complex of relative cycles}, introduced by the author in \cite{Spiridonov}. Then, in Section \ref{S7}, we provide a proof of Proposition \ref{surj0}. It requires us to consider some additional spectral sequences associated with the action of curve stabilisers on some subcomplexes of $\B(x)$. Finally, in Section \ref{S8} we give a proof of Theorem \ref{th2} based on these spectral sequences.

\section{Proof of Proposition \ref{diff31}} \label{S4}

Let us introduce some notation. It will be usually convenient to forget about the orientation of the multicurves from the set $\M(x)$ and consider them up to the action of the (setwise) stabiliser of $\Z[x] \subset \H$ in the whole group $\Mod(\S)$. The corresponding equivalence classes will be called \textit{combinatorial types} of multicurves.

In order to prove Proposition \ref{diff31} we need to compute the group $E^1_{3, 1}$ explicitly. We claim that there are only two combinatorial types of multicurves in $\M_3(x)$, see Fig. \ref{M3}; we denote these classes by $\M'_3$ and $\M''_3$ the corresponding subsets of $\M_3(x)$. Indeed, straightforward computation shows that for $M \in \M_3(x)$ each connected component of the surface $\S \setminus M$ has Euler characteristic equal to $-1$, i.e. it is a sphere with three punctures. Since the components of $M$ are nonseparating, we have two possible cases: either there exist two thrice-punctured spheres with two common boundary components, or not. In the first case $M$ belongs to $\M'_3$, while in the second case we have $M \in \M''_3$.

A multicurve $M''_3 \in \M''_3$ does not contain bounding pairs. Since the mapping class group of a three punctured sphere is trivial, formula (\ref{LB2}) implies that the group 
$\I_{M''_3}$ is also trivial, so we have $\H_1(\I_{M''_3}, \Z) = 0$. Consider a multicurve $M'_3 \in \M'_3$. The exact sequence (\ref{LB2}) implies that
$\I_{M'_3} = \left\langle T_{\gamma_1} T_{\gamma_2}^{-1} \right\rangle \cong \Z$. 
We obtain the following result.

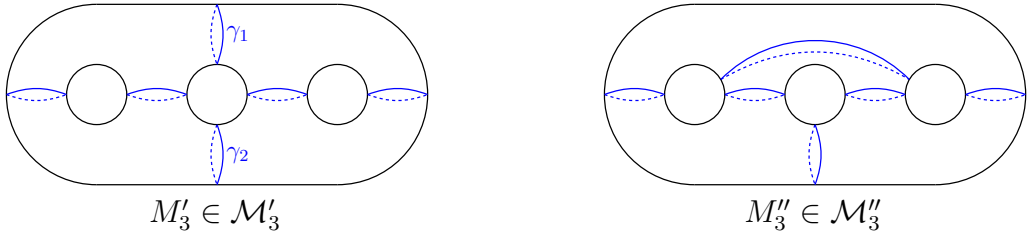
\begin{figure}[h]
	\begin{minipage}[h]{0.49\linewidth}
		\begin{center}
			\scalebox{0.4}{
				\begin{tikzpicture}
				\draw[blue, very thick, dashed] (0,3) to[out = -110, in = 110] (0, 1);
				\draw[blue, very thick, dashed] (0,-1) to[out = -110, in = 110] (0, -3);
				\draw[blue, very thick, dashed] (1,0) to[out = -20, in = 200] (3, 0);
				\draw[blue, very thick, dashed] (5,0) to[out = -20, in = 200] (7, 0);
				\draw[blue, very thick, dashed] (-3,0) to[out = -20, in = 200] (-1, 0);
				\draw[blue, very thick, dashed] (-7,0) to[out = -20, in = 200] (-5, 0);

				\draw[very thick] (-4, 3) to (4, 3);
				\draw[very thick] (-4, -3) to (4, -3);
				\draw[very thick] (-4, 3) arc (90:270:3);
				\draw[very thick] (4, 3) arc (90:-90:3);
				\draw[very thick] (0, 0) circle (1);
				\draw[very thick] (-4, 0) circle (1);
				\draw[very thick] (4, 0) circle (1);
				
				\draw[blue, very thick] (0,3) to[out = -70, in = 70] (0, 1);
				\draw[blue, very thick] (0,-1) to[out = -70, in = 70] (0, -3);
				\draw[blue, very thick] (1,0) to[out = 20, in = 160] (3, 0);
				\draw[blue, very thick] (5,0) to[out = 20, in = 160] (7, 0);
				\draw[blue, very thick] (-3,0) to[out = 20, in = 160] (-1, 0);
				\draw[blue, very thick] (-7,0) to[out = 20, in = 160] (-5, 0);

				\node[blue, scale = 2] at (0.7, 2) {$\gamma_1$};
				\node[blue, scale = 2] at (0.7, -2) {$\gamma_2$};
				
				
				\end{tikzpicture}} \\  $M'_3 \in \M'_3$ \end{center}
	\end{minipage}
	\hfill
	\begin{minipage}[h]{0.49\linewidth}
		\begin{center}
			\scalebox{0.4}{
				\begin{tikzpicture}
				\draw[blue, very thick, dashed] (0,-1) to[out = -110, in = 110] (0, -3);
				\draw[blue, very thick, dashed] (1,0) to[out = -20, in = 200] (3, 0);
				\draw[blue, very thick, dashed] (5,0) to[out = -20, in = 200] (7, 0);
				\draw[blue, very thick, dashed] (-3,0) to[out = -20, in = 200] (-1, 0);
				\draw[blue, very thick, dashed] (-7,0) to[out = -20, in = 200] (-5, 0);
				
				\draw[blue, very thick, dashed] (-4 + 0.866, 0.5) to[out = 30, in = 150] (4 - 0.866, 0.5);

				\draw[very thick] (-4, 3) to (4, 3);
				\draw[very thick] (-4, -3) to (4, -3);
				\draw[very thick] (-4, 3) arc (90:270:3);
				\draw[very thick] (4, 3) arc (90:-90:3);
				\draw[very thick] (0, 0) circle (1);
				\draw[very thick] (-4, 0) circle (1);
				\draw[very thick] (4, 0) circle (1);
				
				\draw[blue, very thick] (0,-1) to[out = -70, in = 70] (0, -3);
				\draw[blue, very thick] (1,0) to[out = 20, in = 160] (3, 0);
				\draw[blue, very thick] (5,0) to[out = 20, in = 160] (7, 0);
				\draw[blue, very thick] (-3,0) to[out = 20, in = 160] (-1, 0);
				\draw[blue, very thick] (-7,0) to[out = 20, in = 160] (-5, 0);
				
				\draw[blue, very thick] (-4 + 0.866, 0.5) to[out = 45, in = 135] (4 - 0.866, 0.5);

				
				
				
				\end{tikzpicture}} \\  $M''_3 \in \M''_3$ \end{center}
	\end{minipage}
	\caption{The combinatorial types $\M'_3$ and $\M''_3$.}
	\label{M3}
\end{figure}

\begin{prop}\label{basis31}
	The elements $P_{M'_3} \otimes [T_{\gamma_1} T_{\gamma_2}^{-1}]$, form a basis of the free abelian group $E^1_{3, 1}$. Here $M'_3$ runs over the set $\M'_3(x) / \I$, and the curves $\gamma_1$, $\gamma_2$ are shown in Fig. \ref{M3}.
\end{prop}

We also need some information about the group $E^1_{2, 1}$. Consider the combinatorial type $\M'_2(x) \subset \M_2(x)$, the corresponding representative is shown in blue in Fig. \ref{M22}.

\begin{prop}\cite[Proposition 6.11]{Gaifullin_T3} \label{propG}
	Let $M'_2 \in  \M'_2$ be a multicurve, see Fig. \ref{M22}. Then the stabiliser $\I_{M'_2}$ is a free group with an infinite number of generators
	\begin{equation*}
	T_{\delta}^k T_{\gamma_1} T_{\gamma_2}^{-1} T_{\delta}^{-k}, \; \; k \in \Z.
	\end{equation*}
\end{prop}

\begin{corollary} \label{freeE21}
	The elements $P_{M'_2} \otimes [T_{\delta}^k T_{\gamma_1} T_{\gamma_2}^{-1} T_{\delta}^{-k}]$ generate a free abelian subgroup in $E^1_{2, 1}$. Here $M'_2$ runs over the set $\M'_2(x) / \I$, $k \in \Z$, and the curves $\gamma_1$, $\gamma_2$, $\delta$ are shown in Fig. \ref{M3}.
\end{corollary}

Denote the subgroup constructed in Corollary \ref{freeE21} by $Q \subset E^1_{2, 1}$. Let $p_Q: E^1_{2, 1} \to Q$ be the projection corresponding to the inclusion $\M'_2(x) / \I \hookrightarrow \M_2(x) / \I$.

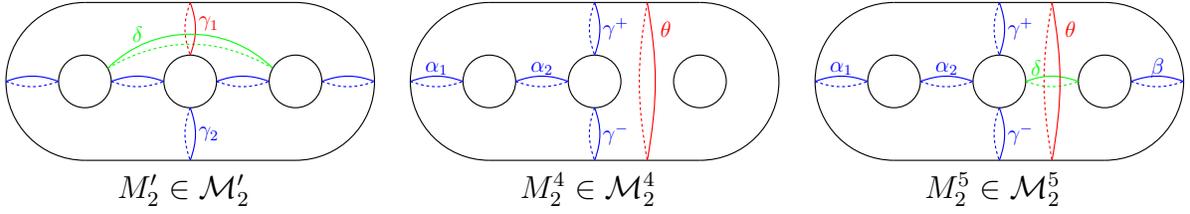
\begin{figure}[h]
	\begin{minipage}[h]{0.31\linewidth}
		\begin{center}
			\scalebox{0.35}{
				\begin{tikzpicture}
				\draw[red, very thick, dashed] (0,3) to[out = -110, in = 110] (0, 1);
				\draw[blue, very thick, dashed] (0,-1) to[out = -110, in = 110] (0, -3);
				\draw[blue, very thick, dashed] (1,0) to[out = -20, in = 200] (3, 0);
				\draw[blue, very thick, dashed] (5,0) to[out = -20, in = 200] (7, 0);
				\draw[blue, very thick, dashed] (-3,0) to[out = -20, in = 200] (-1, 0);
				\draw[blue, very thick, dashed] (-7,0) to[out = -20, in = 200] (-5, 0);
				
				\draw[green, very thick, dashed] (-4 + 0.866, 0.5) to[out = 30, in = 150] (4 - 0.866, 0.5);


				\draw[very thick] (-4, 3) to (4, 3);
				\draw[very thick] (-4, -3) to (4, -3);
				\draw[very thick] (-4, 3) arc (90:270:3);
				\draw[very thick] (4, 3) arc (90:-90:3);
				\draw[very thick] (0, 0) circle (1);
				\draw[very thick] (-4, 0) circle (1);
				\draw[very thick] (4, 0) circle (1);
				
				\draw[red, very thick] (0,3) to[out = -70, in = 70] (0, 1);
				\draw[blue, very thick] (0,-1) to[out = -70, in = 70] (0, -3);
				\draw[blue, very thick] (1,0) to[out = 20, in = 160] (3, 0);
				\draw[blue, very thick] (5,0) to[out = 20, in = 160] (7, 0);
				\draw[blue, very thick] (-3,0) to[out = 20, in = 160] (-1, 0);
				\draw[blue, very thick] (-7,0) to[out = 20, in = 160] (-5, 0);
				
				\draw[green, very thick] (-4 + 0.866, 0.5) to[out = 45, in = 135] (4 - 0.866, 0.5);

				\node[red, scale = 2] at (0.7, 2.2) {$\gamma_1$};
				\node[blue, scale = 2] at (0.7, -2) {$\gamma_2$};
				\node[green, scale = 2] at (-2, 1.8) {$\delta$};
				\end{tikzpicture}} \\  $M'_2 \in \M'_2$ \end{center}
	\end{minipage}
	\hfill
	\begin{minipage}[h]{0.31\linewidth}
		\begin{center}
			\scalebox{0.35}{
				\begin{tikzpicture}
				\draw[blue, very thick, dashed] (0,3) to[out = -110, in = 110] (0, 1);
				\draw[blue, very thick, dashed] (0,-1) to[out = -110, in = 110] (0, -3);
				\draw[blue, very thick, dashed] (-3,0) to[out = -20, in = 200] (-1, 0);
				\draw[blue, very thick, dashed] (-7,0) to[out = -20, in = 200] (-5, 0);
				\draw[red, very thick, dashed] (2,3) to[out = -100, in = 100] (2, -3);

				\draw[very thick] (-4, 3) to (4, 3);
				\draw[very thick] (-4, -3) to (4, -3);
				\draw[very thick] (-4, 3) arc (90:270:3);
				\draw[very thick] (4, 3) arc (90:-90:3);
				\draw[very thick] (0, 0) circle (1);
				\draw[very thick] (-4, 0) circle (1);
				\draw[very thick] (4, 0) circle (1);
				
				\draw[blue, very thick] (0,3) to[out = -70, in = 70] (0, 1);
				\draw[blue, very thick] (0,-1) to[out = -70, in = 70] (0, -3);
				\draw[blue, very thick] (-3,0) to[out = 20, in = 160] (-1, 0);
				\draw[blue, very thick] (-7,0) to[out = 20, in = 160] (-5, 0);
				\draw[red, very thick] (2,3) to[out = -80, in = 80] (2, -3);

				\node[blue, scale = 2] at (0.8, 2) {$\gamma^+$};
				\node[blue, scale = 2] at (0.8, -2) {$\gamma^-$};
				
				\node[red, scale = 2] at (2.7, 2) {$\theta$};
				
				\node[blue, scale = 2] at (-2, 0.5) {$\alpha_2$};
				\node[blue, scale = 2] at (-6, 0.5) {$\alpha_1$};

				\end{tikzpicture}} \\  $M_2^4 \in \M^4_2$ \end{center}
	\end{minipage}
	\hfill
	\begin{minipage}[h]{0.31\linewidth}
		\begin{center}
			\scalebox{0.35}{
				\begin{tikzpicture}
				\draw[blue, very thick, dashed] (0,3) to[out = -110, in = 110] (0, 1);
				\draw[blue, very thick, dashed] (0,-1) to[out = -110, in = 110] (0, -3);
				\draw[green, very thick, dashed] (1,0) to[out = -20, in = 200] (3, 0);
				\draw[blue, very thick, dashed] (5,0) to[out = -20, in = 200] (7, 0);
				\draw[blue, very thick, dashed] (-3,0) to[out = -20, in = 200] (-1, 0);
				\draw[blue, very thick, dashed] (-7,0) to[out = -20, in = 200] (-5, 0);
				\draw[red, very thick, dashed] (2,3) to[out = -100, in = 100] (2, -3);

				\draw[very thick] (-4, 3) to (4, 3);
				\draw[very thick] (-4, -3) to (4, -3);
				\draw[very thick] (-4, 3) arc (90:270:3);
				\draw[very thick] (4, 3) arc (90:-90:3);
				\draw[very thick] (0, 0) circle (1);
				\draw[very thick] (-4, 0) circle (1);
				\draw[very thick] (4, 0) circle (1);
				
				\draw[blue, very thick] (0,3) to[out = -70, in = 70] (0, 1);
				\draw[blue, very thick] (0,-1) to[out = -70, in = 70] (0, -3);
				\draw[green, very thick] (1,0) to[out = 20, in = 160] (3, 0);
				\draw[blue, very thick] (5,0) to[out = 20, in = 160] (7, 0);
				\draw[blue, very thick] (-3,0) to[out = 20, in = 160] (-1, 0);
				\draw[blue, very thick] (-7,0) to[out = 20, in = 160] (-5, 0);
				\draw[red, very thick] (2,3) to[out = -80, in = 80] (2, -3);

				\node[blue, scale = 2] at (0.8, 2) {$\gamma^+$};
				\node[blue, scale = 2] at (0.8, -2) {$\gamma^-$};
				
				\node[blue, scale = 2] at (6, 0.5) {$\beta$};
				\node[green, scale = 2] at (1.4, 0.5) {$\delta$};
				
				\node[red, scale = 2] at (2.7, 2) {$\theta$};
				
				\node[blue, scale = 2] at (-2, 0.5) {$\alpha_2$};
				\node[blue, scale = 2] at (-6, 0.5) {$\alpha_1$};
				
				\end{tikzpicture}} \\  $M_2^5 \in \M^{5}_2$ \end{center}
	\end{minipage}
	\caption{The combinatorial types $\M'_2(x)$, $\M^4_2$ and $\M^5_2$.}
	\label{M22}
\end{figure}

\begin{proof}[Proof of Proposition \ref{diff31}.] It suffices to prove that the map $p_Q \circ d^1_{3, 1}: E^1_{3, 1} \to Q$ is injective. Consider a basis element $P_{M'_3} \otimes [T_{\gamma_1} T_{\gamma_2}^{-1}]$ from Proposition \ref{basis31}. We have
\begin{equation}\label{d31}
P_Q \circ d^1_{3, 1} (P_{M'_3} \otimes [T_{\gamma_1} T_{\gamma_2}^{-1}]) = P_{M'_3 \setminus \gamma_2} \otimes [T_{\gamma_1} T_{\gamma_2}^{-1}] - P_{M'_3 \setminus \gamma_1} \otimes [T_{\gamma_1} T_{\gamma_2}^{-1}],
\end{equation}
the sign here depends on the orientation on $P_{M'_3}$. Therefore the image under the mapping $P_Q \circ d^1_{3, 1}$ of each basis element of $E^1_{3, 1}$ is the difference of some two basis elements of $Q$ (see Corollary \ref{freeE21}). Let us show that these pairs of basis elements of $Q$ do not intersect with each other. Indeed, let us fix $M'_2 \in \M'_2$. Then any $M \in \M'_3$ satisfying $M'_2 \subseteq \partial M$ has the form $M'_2 \cup T_{\delta}^k(\gamma_1)$ for some $k \in \Z$. Moreover, all curves $T_{\delta}^k(\gamma_1)$ (see Fig. \ref{M22}) belong to pairwise disjoint $\I_{M'_2}$-orbits. Since $T_{T_{\delta}^k(\gamma_1)} = T^{k}_\delta T_{\gamma_1} T_\delta^{-k}$, this concludes the proof.
\end{proof}

\section{Proof of Proposition \ref{diff22}} \label{S5}

The main goal of this section is to prove Proposition \ref{diff22}. We will first find a basis of the group $E^1_{2, 2}$. The next step is to construct a large family of linearly independent elements in the group $E^1_{1, 2}$. Finally, we will give an explicit formula for the differential $d^1_{2, 2}: E^1_{2, 2} \to E^1_{1, 2}$ and prove that it is injective using the partial description of the group $E^1_{1, 2}$.

\subsection{The term $E^1_{2, 2}$} \label{sub51}

Let us compute the group $E^1_{2, 2}$ explicitly. We claim that there are three combinatorial types of multicurves in this case; namely, these are $\M'_2(x), \M^4_2, \M^5_2 \subset \M_2(x)$, see Fig. \ref{M22}. Indeed, straightforward computation shows that for $M \in \M_2(x)$ the surface $\S \setminus M$ has two connected components of Euler characteristic equal to $-1$ and one connected component of Euler characteristic equal to $-2$.  The components of the first type are spheres with three punctures, while the last one is either a sphere with four punctures or a torus with two punctures. If two thrice-punctured spheres have one common boundary component, we have $M \in \M'_2$. Otherwise, we obtain $M \in \M_2^4$ or $M \in \M_2^5$ if the third connected of $\S \setminus M$ is a thrice-punctured torus or a sphere with four punctures, respectively.

Let $M'_2 \in \M'_2$. Then by Proposition \ref{propG} the group $\I_{M'_2}$ is free, so $\H_2(\I_{M'_2}, \Z) = 0$. 
In order study the combinatorial types $\M_2^4$ and $\M_2^5$ we need to recall the following results from \cite{Bestvina}. Let $\gamma$ and $\beta$ be disjoint nonhomologous nonseparating simple closed curve on a genus $2$ surface $\S_2$. For a separating curve $\theta$ disjoint from $\gamma$, let us denote by $\H_\theta \subset \H_1(\S_2, \Z)$ the homology group of the torus bounded by $\theta$ and not containing $\gamma$. Let $U \subset \H_1(\S_2, \Z)$ be a symplectic subgroup of rank $2$. We say that $U$ is admissible for $\gamma$, if $U = \H_\theta$ for some separating curve $\theta$ disjoint from $\gamma$. Similarly, we say that $U$ is admissible for $\gamma \cup \beta$, if $U = \H_\theta$ for some separating curve $\theta$ disjoint from $\gamma \cup \beta$.

\begin{lemma}\cite[Lemma 7.2]{Bestvina} \label{gen2}
	Let $\gamma$ be a nonseparating simple closed curve on $\S_2$. Then the group $\Stab_{\I_2}(\gamma)$ is freely generated by an infinite number of Dehn twists $T_\theta$ about separating curves $\theta$. Here $\H_\theta$ runs over the set of admissible for $\gamma$ symplectic subgroups in $\H_1(\S_2, \Z)$.
\end{lemma}

\begin{lemma}\cite[Lemma 7.1]{Bestvina} \label{gen22}
	Let $\gamma$ and $\beta$ be disjoint nonhomologous nonseparating simple closed curves on $\S_2$. Then the group $\Stab_{\I_2}(\gamma \cup \beta)$ is freely generated by an infinite number of Dehn twists $T_\theta$ about separating curves $\theta$. Here $\H_\theta$ runs over the set of admissible for $\gamma \cup \beta$ symplectic subgroups in $\H_1(\S_2, \Z)$.
\end{lemma}

\begin{prop} \label{H124}
	Let $M^4_2 \in \M^4_2$.
	There is an isomorphism $\I_{M^4_2} \cong \Z \times F_{\infty}$. Here $\Z = \left\langle T_{\gamma^+} T_{\gamma^-}^{-1} \right\rangle$ and $F_{\infty}$ is a free group generated by infinite number of Dehn twists $T_{\theta}$. Here $\H_\theta$ runs over the set of admissible symplectic subgroups for $M^4_2$.
\end{prop}

\begin{remark}
	We do not claim that one can take \textit{any} collection of Dehn twists $T_{\theta}$ satisfying the above conditions. We do not need an explicit set of generators here. The same can be said about Proposition \ref{H125}.
\end{remark}

\begin{proof}[Proof of Proposition \ref{H124}]
	Since the group $\PMod(\S_{0, 3})$ is trivial, the exact sequence (\ref{LB2}) has the form
	\begin{equation}\label{exacts}
	1 \to \left\langle T_{\gamma^+} T_{\gamma^-}^{-1} \right\rangle  \to \I_{M^4_2} \to \PMod(\S_{1, 2}).
	\end{equation}
	Let $\gamma$ be a nonseparating simple closed curve on a genus $2$ surface $\S_2$. The exact sequence (\ref{LB2}) yields the inclusion $$\Stab_{\I_2}(\gamma) \hookrightarrow \PMod(\S_{1, 2}).$$
	We claim that the image of $\I_{M^4_2}$ in $\PMod(\S_{1, 2})$ in (\ref{exacts}) coincides with the image of $\Stab_{\I_2}(\gamma)$. Indeed, both of these groups can be characterized as the subgroups in $\PMod(\S_{1, 2})$, consisting of all elements  acting trivially on $\H_1(\S_{1, 2}, \Z)$. Therefore Proposition \ref{H124} follows from Lemma \ref{gen2}.
\end{proof}

\begin{prop} \label{H125}
	Let $M^5_2 \in \M^5_2$.
	There is an isomorphism $\I_{M^5_2} \cong \Z \times F_{\infty}$. Here $\Z = \left\langle T_{\gamma^+} T_{\gamma^-}^{-1} \right\rangle$ and $F_{\infty}$ is a free group generated by the infinite number of Dehn twists $T_{\theta}$. Here $\H_\theta$ runs over the set of admissible symplectic subgroups for $M^5_2$.
\end{prop}

\begin{proof}
Since the group $\PMod(\S_{0, 3})$ is trivial, the exact sequence (\ref{LB2}) has the form
\begin{equation}\label{exacts2}
1 \to \left\langle T_{\gamma^+} T_{\gamma^-}^{-1} \right\rangle  \to \I_{M^5_2} \to \PMod(\S_{0, 4}).
\end{equation}
Let $\gamma$ and $\beta$ be disjoint nonhomologous nonseparating simple closed curves on $\S_2$.  The exact sequence (\ref{LB2}) yields the inclusion $$\Stab_{\I_2}(\gamma \cup \beta) \hookrightarrow \PMod(\S_{0, 2}).$$
We claim that the image of $\I_{M^5_2}$ in $\PMod(\S_{0, 4})$ in (\ref{exacts2}) coincides with the image of $\Stab_{\I_2}(\gamma \cup \beta)$. Indeed, both subgroups are contained in $\Stab_{\PMod_{1, 2}}(\beta)$ and consist of all elements  acting trivially on $\H_1(\S_{1, 2}, \Z)$.
Therefore the results follows from Lemma \ref{gen22}.
\end{proof}

Consider a multicurve $M_2^i \in \M_2^i$ for $i = 4, 5$. Let $U \subset \H$ be an admissible symplectic subgroup for $\M_2^i$. Define the homology class
$$\A_{U} = \A(T_{\gamma^+} T_{\gamma^-}^{-1}, T_{\theta}) \in \H_2(\I_{\M_2^i}, \Z),$$
where $\theta$ is a separating curve disjoint from $M_2^i$ such that $\H_{\theta} = U$. Note that the homology class $\A_U$ does not depend on the choice of $\theta$.

\begin{corollary} \label{H24}
	Consider a multicurve $M_2^i \in \M_2^i$ for $i = 4, 5$.
	The homology classes $\A_{U}$ form a basis of the free abelian group $\H_2(\I_{\M_2^i}, \Z)$. Here $U \subset \H$ runs over all admissible symplectic subgroups for $M^4_2$.
\end{corollary}

Corollary \ref{H24} immediately implies the following result. 

\begin{corollary} \label{freeE22}
	The elements $P_{M^4_2} \otimes \A_{U^4}$ and $P_{M^5_2} \otimes \A_{U^5}$ form a basis of the free abelian group $E^1_{2, 2}$. Here $M^4_2 \in \M^4_2 / \I$, $M^5_2 \in \M^5_2 / \I$, $U^4$ and $U^5$ run over the sets of admissible symplectic subgroups for $P_{M^4_2}$ and $P_{M^5_2}$ respectively.
\end{corollary}

\subsection{The term $E^1_{1, 2}$}

In order to prove Proposition \ref{diff22} we also need some information about the group $E^1_{1, 2}$. In the next two subsections we will construct a large family of linearly independent elements in $E^1_{1, 2}$.
Let $\M_1^2$, $\M_1^3$ and $\M_1^4$ be the combinatorial types of multicurves in $\M_1(x)$, provided in blue in Fig. \ref{M122} (note that these are not the only types).

\begin{prop} \label{H13}
	Let $i = 2, 3, 4$ and let $M_1^i \in \M_1^i$. Then we have an isomorphism $\I_{M_1^i} \cong F_\infty \times \Z \times F_\infty$. Here $\Z = \left\langle T_{\gamma^+} T_{\gamma^-}^{-1} \right\rangle$. The first copy of $F_{\infty}$ is a free group generated by an infinite number of Dehn twists $T_{\theta_1}$; the second copy of $F_{\infty}$ is a free group generated by an infinite number of Dehn twists $T_{\theta_2}$. Here $\H_{\theta_1}$ ($\H_{\theta_2}$) runs over all admissible symplectic subgroups of for $M^i_1$ belongs to the left (right) hand side of $\S$ (see Fig \ref{M122}).
\end{prop}

\begin{proof}
	The proof is similar to the proofs of Propositions \ref{H124} and \ref{H125}.
\end{proof}

\begin{figure}[h]
	\begin{minipage}[h]{0.31\linewidth}
		\begin{center}
			\scalebox{0.35}{
				\begin{tikzpicture}
				\draw[blue, very thick, dashed] (0,3) to[out = -110, in = 110] (0, 1);
				\draw[blue, very thick, dashed] (0,-1) to[out = -110, in = 110] (0, -3);
				\draw[red, very thick, dashed] (2,3) to[out = -100, in = 100] (2, -3);
				\draw[red, very thick, dashed] (-2,3) to[out = -100, in = 100] (-2, -3);

				\draw[very thick] (-4, 3) to (4, 3);
				\draw[very thick] (-4, -3) to (4, -3);
				\draw[very thick] (-4, 3) arc (90:270:3);
				\draw[very thick] (4, 3) arc (90:-90:3);
				\draw[very thick] (0, 0) circle (1);
				\draw[very thick] (-4, 0) circle (1);
				\draw[very thick] (4, 0) circle (1);
				
				\draw[blue, very thick] (0,3) to[out = -70, in = 70] (0, 1);
				\draw[blue, very thick] (0,-1) to[out = -70, in = 70] (0, -3);
				\draw[red, very thick] (2,3) to[out = -80, in = 80] (2, -3);
				\draw[red, very thick] (-2,3) to[out = -80, in = 80] (-2, -3);

				\node[blue, scale = 2] at (0.8, 2) {$\gamma^+$};
				\node[blue, scale = 2] at (0.8, -2) {$\gamma^-$};
				
				\node[red, scale = 2] at (-1.3, 2) {$\theta_1$};
				\node[red, scale = 2] at (2.7, 2) {$\theta_2$};
				

				\end{tikzpicture}} \\  $M_1^2 \in \M_1^2$ \end{center}
	\end{minipage}
	\hfill
	\begin{minipage}[h]{0.31\linewidth}
		\begin{center}
			\scalebox{0.35}{
				\begin{tikzpicture}
				\draw[blue, very thick, dashed] (0,3) to[out = -110, in = 110] (0, 1);
				\draw[blue, very thick, dashed] (0,-1) to[out = -110, in = 110] (0, -3);
				\draw[blue, very thick, dashed] (-7,0) to[out = -20, in = 200] (-5, 0);
				\draw[red, very thick, dashed] (2,3) to[out = -100, in = 100] (2, -3);
				\draw[red, very thick, dashed] (-2,3) to[out = -100, in = 100] (-2, -3);

				\draw[very thick] (-4, 3) to (4, 3);
				\draw[very thick] (-4, -3) to (4, -3);
				\draw[very thick] (-4, 3) arc (90:270:3);
				\draw[very thick] (4, 3) arc (90:-90:3);
				\draw[very thick] (0, 0) circle (1);
				\draw[very thick] (-4, 0) circle (1);
				\draw[very thick] (4, 0) circle (1);
				
				\draw[blue, very thick] (0,3) to[out = -70, in = 70] (0, 1);
				\draw[blue, very thick] (0,-1) to[out = -70, in = 70] (0, -3);
				\draw[blue, very thick] (-7,0) to[out = 20, in = 160] (-5, 0);
				\draw[red, very thick] (2,3) to[out = -80, in = 80] (2, -3);
				\draw[red, very thick] (-2,3) to[out = -80, in = 80] (-2, -3);

				\node[blue, scale = 2] at (0.8, 2) {$\gamma^+$};
				\node[blue, scale = 2] at (0.8, -2) {$\gamma^-$};
				
				
				\node[red, scale = 2] at (-1.3, 2) {$\theta_1$};
				\node[red, scale = 2] at (2.7, 2) {$\theta_2$};
				
				\node[blue, scale = 2] at (-6, 0.5) {$\alpha$};
				
				\end{tikzpicture}} \\  $M_1^3 \in \M_1^3$ \end{center}
	\end{minipage}
	\hfill
	\begin{minipage}[h]{0.31\linewidth}
		\begin{center}
			\scalebox{0.35}{
				\begin{tikzpicture}
				\draw[blue, very thick, dashed] (0,3) to[out = -110, in = 110] (0, 1);
				\draw[blue, very thick, dashed] (0,-1) to[out = -110, in = 110] (0, -3);
				\draw[blue, very thick, dashed] (5,0) to[out = -20, in = 200] (7, 0);
				\draw[blue, very thick, dashed] (-7,0) to[out = -20, in = 200] (-5, 0);
				\draw[red, very thick, dashed] (2,3) to[out = -100, in = 100] (2, -3);
				\draw[red, very thick, dashed] (-2,3) to[out = -100, in = 100] (-2, -3);

				\draw[very thick] (-4, 3) to (4, 3);
				\draw[very thick] (-4, -3) to (4, -3);
				\draw[very thick] (-4, 3) arc (90:270:3);
				\draw[very thick] (4, 3) arc (90:-90:3);
				\draw[very thick] (0, 0) circle (1);
				\draw[very thick] (-4, 0) circle (1);
				\draw[very thick] (4, 0) circle (1);
				
				\draw[blue, very thick] (0,3) to[out = -70, in = 70] (0, 1);
				\draw[blue, very thick] (0,-1) to[out = -70, in = 70] (0, -3);
				\draw[blue, very thick] (5,0) to[out = 20, in = 160] (7, 0);
				\draw[blue, very thick] (-7,0) to[out = 20, in = 160] (-5, 0);
				\draw[red, very thick] (2,3) to[out = -80, in = 80] (2, -3);
				\draw[red, very thick] (-2,3) to[out = -80, in = 80] (-2, -3);

				\node[blue, scale = 2] at (0.8, 2) {$\gamma^+$};
				\node[blue, scale = 2] at (0.8, -2) {$\gamma^-$};
				
				\node[blue, scale = 2] at (6, 0.5) {$\beta$};
				
				\node[red, scale = 2] at (-1.3, 2) {$\theta_1$};
				\node[red, scale = 2] at (2.7, 2) {$\theta_2$};
				
				\node[blue, scale = 2] at (-6, 0.5) {$\alpha$};

				\end{tikzpicture}} \\  $M_1^4 \in \M_1^4$ \end{center}
	\end{minipage}
	\caption{The combinatorial types $\M_1^2$, $\M_1^3$ and $\M_1^4$.}
	\label{M122}
\end{figure}
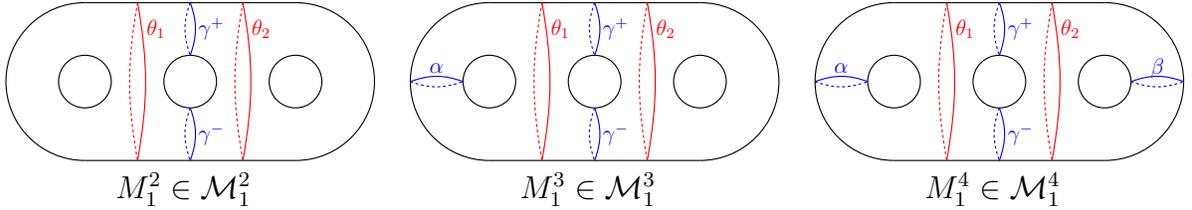

Let $i = 2, 3, 4$ and let $M_1^i \in \M_1^i$. Let $U \subset \H$ be an admissible symplectic subgroup for $M_1^i$. Define the homology class
$$\A_{U}  = \A(T_{\gamma^+} T_{\gamma^-}^{-1}, T_{\theta}) \in \H_2(\I_{M_1^i}, \Z),$$
where $\theta$ is a separating curve disjoint from $M_1^i$ such that $\H_{\theta} = U$.

\begin{corollary} \label{H113}
	 Let $i = 2, 3, 4$ and let $M_1^i \in \M_1^i$. Then the homology classes $\A_{U}$ generate a free abelian subgroup in $\H_2(\I_{M^i_1}, \Z)$. Here $U \subset \H$ runs over the set of all admissible symplectic subgroups for $M^i_1$. 
\end{corollary}

\begin{corollary} \label{freeE12}
	The elements $P_{M^i_1} \otimes \A_{U^i}$ form a basis of a free abelian subgroup in $E^1_{1, 2}$. Here $i = 2, 3, 4$, $M^i_{1} \in \M^i_{1} / \I$, $U^i \subset \H$ runs over the set of all admissible symplectic subgroups for $M^i_1$. 
\end{corollary}

\subsection{Homomorphisms $\nu_{\gamma, \W}$}

Let $\gamma$ be a nonseparating curve on $\S$ and let $\W = (W_1, W_2)$ be an orthogonal splitting of the group $\H_\gamma = [\gamma]^{\perp} / [\gamma]$. Using the results of Mess \cite{Mess}, Gaifullin \cite{Gaifullin_T3} constructed a homomorphism $\nu_{\gamma, \W}: \I_\gamma \to \Z$ with the following properties.

\begin{prop}\cite[Proposition 6.4]{Gaifullin_T3} \label{hom}
	
	(a) Suppose that $\delta$ is a separating curve on $\S$ disjoint from $\gamma$. Then $\nu_{\gamma, \W}(T_\delta) = 1$ if $\delta$ yields the splitting $\W$ for $\H_\gamma$, and $\nu_{\gamma, \W}(T_\delta) = 0$ otherwise.
	
	(b) Suppose the $\{\gamma, \gamma'\}$ is a bounding pair. Then $\nu_{\gamma, \W}(T_\gamma T_{\gamma'}^{-1}) = -1$ if $\{\gamma, \gamma'\}$ yields the splitting $\W$ for $\H_\gamma$, and $\nu_{\gamma, \W}(T_\gamma T_{\gamma'}^{-1}) = 0$ otherwise.
	
	(c) Suppose that $\{\delta, \delta'\}$ is a bounding pair such that $\delta$ and $\delta'$ are disjoint from $\gamma$ and neither $\delta$ nor $\delta'$ is homotopic to $\gamma$. Then $\nu_{\gamma, \W}(T_\delta T_{\delta'}^{-1}) = 0$.
\end{prop}

Let $\N_1^3$ and $\N_1^4$ be the combinatorial types of multicurves in $\M_1(x)$ shown in blue in Fig. \ref{M12}.

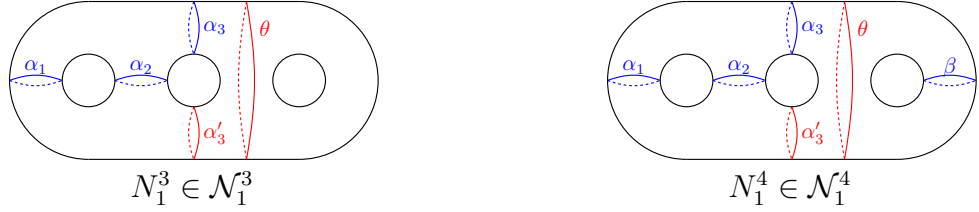
\begin{figure}[h]
	\begin{minipage}[h]{0.49\linewidth}
		\begin{center}
			\scalebox{0.35}{
				\begin{tikzpicture}
				\draw[blue, very thick, dashed] (0,3) to[out = -110, in = 110] (0, 1);
				\draw[red, very thick, dashed] (0,-1) to[out = -110, in = 110] (0, -3);
				\draw[blue, very thick, dashed] (-3,0) to[out = -20, in = 200] (-1, 0);
				\draw[blue, very thick, dashed] (-7,0) to[out = -20, in = 200] (-5, 0);
				\draw[red, very thick, dashed] (2,3) to[out = -100, in = 100] (2, -3);

				\draw[very thick] (-4, 3) to (4, 3);
				\draw[very thick] (-4, -3) to (4, -3);
				\draw[very thick] (-4, 3) arc (90:270:3);
				\draw[very thick] (4, 3) arc (90:-90:3);
				\draw[very thick] (0, 0) circle (1);
				\draw[very thick] (-4, 0) circle (1);
				\draw[very thick] (4, 0) circle (1);
				
				\draw[blue, very thick] (0,3) to[out = -70, in = 70] (0, 1);
				\draw[red, very thick] (0,-1) to[out = -70, in = 70] (0, -3);
				\draw[blue, very thick] (-3,0) to[out = 20, in = 160] (-1, 0);
				\draw[blue, very thick] (-7,0) to[out = 20, in = 160] (-5, 0);
				\draw[red, very thick] (2,3) to[out = -80, in = 80] (2, -3);

				\node[blue, scale = 2] at (0.8, 2) {$\alpha_3$};
				\node[red, scale = 2] at (0.8, -2) {$\alpha'_3$};
				
				\node[red, scale = 2] at (2.7, 2) {$\theta$};
				
				\node[blue, scale = 2] at (-2, 0.5) {$\alpha_2$};
				\node[blue, scale = 2] at (-6, 0.5) {$\alpha_1$};

				\end{tikzpicture}} \\  $N_1^3 \in \N_1^3$ \end{center}
	\end{minipage}
	\hfill
	\begin{minipage}[h]{0.49\linewidth}
		\begin{center}
			\scalebox{0.35}{
				\begin{tikzpicture}
				\draw[blue, very thick, dashed] (0,3) to[out = -110, in = 110] (0, 1);
				\draw[red, very thick, dashed] (0,-1) to[out = -110, in = 110] (0, -3);
				\draw[blue, very thick, dashed] (5,0) to[out = -20, in = 200] (7, 0);
				\draw[blue, very thick, dashed] (-3,0) to[out = -20, in = 200] (-1, 0);
				\draw[blue, very thick, dashed] (-7,0) to[out = -20, in = 200] (-5, 0);
				\draw[red, very thick, dashed] (2,3) to[out = -100, in = 100] (2, -3);

				\draw[very thick] (-4, 3) to (4, 3);
				\draw[very thick] (-4, -3) to (4, -3);
				\draw[very thick] (-4, 3) arc (90:270:3);
				\draw[very thick] (4, 3) arc (90:-90:3);
				\draw[very thick] (0, 0) circle (1);
				\draw[very thick] (-4, 0) circle (1);
				\draw[very thick] (4, 0) circle (1);
				
				\draw[blue, very thick] (0,3) to[out = -70, in = 70] (0, 1);
				\draw[red, very thick] (0,-1) to[out = -70, in = 70] (0, -3);
				\draw[blue, very thick] (5,0) to[out = 20, in = 160] (7, 0);
				\draw[blue, very thick] (-3,0) to[out = 20, in = 160] (-1, 0);
				\draw[blue, very thick] (-7,0) to[out = 20, in = 160] (-5, 0);
				\draw[red, very thick] (2,3) to[out = -80, in = 80] (2, -3);

				\node[blue, scale = 2] at (0.8, 2) {$\alpha_3$};
				\node[red, scale = 2] at (0.8, -2) {$\alpha'_3$};
				
				\node[blue, scale = 2] at (6, 0.5) {$\beta$};
				
				\node[red, scale = 2] at (2.7, 2) {$\theta$};
				
				\node[blue, scale = 2] at (-2, 0.5) {$\alpha_2$};
				\node[blue, scale = 2] at (-6, 0.5) {$\alpha_1$};
				
				\end{tikzpicture}} \\  $N_1^4 \in \N_1^4$ \end{center}
	\end{minipage}
	\caption{The combinatorial types $\N_1^3$ and $\N_1^4$.}
	\label{M12}
\end{figure}

Let us recall about the following useful relation in the mapping class group.

\begin{prop}[\textbf{Lantern relation}] \cite[Proposition 5.1]{Primer} \label{Lant}
	Let the curves $b_1, b_2, b_3, b_4$ bound a sphere with $4$ punctures of the surface $\S_g$ and let the curves $x, y, z$ are as shown in Fig. \ref{Lan}. Then we have the relation
	$$T_x T_y T_z = T_{b_1}T_{b_2}T_{b_3}T_{b_4}.$$
\end{prop}

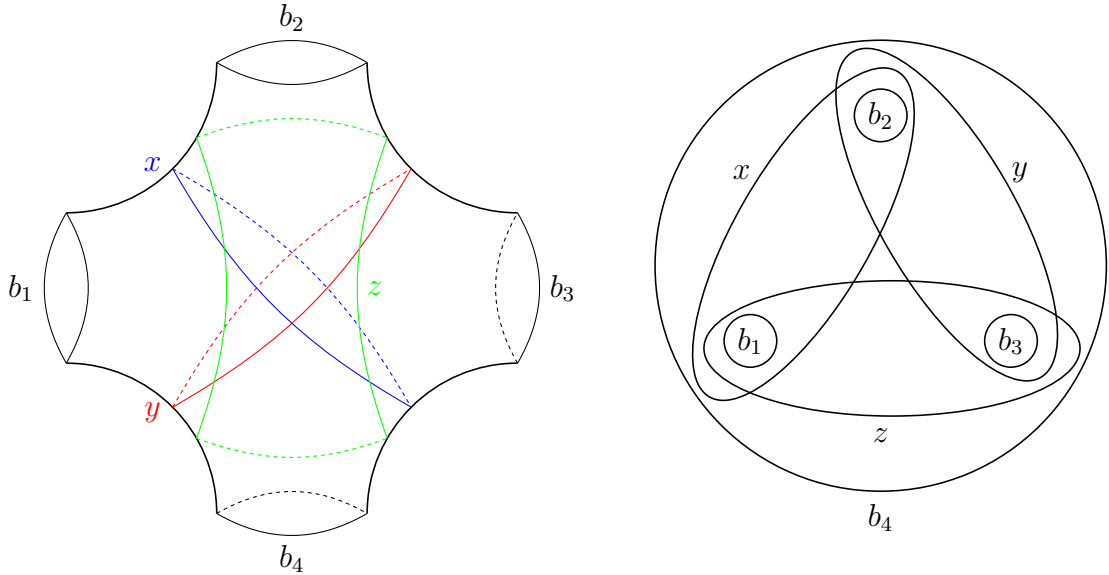
\begin{figure}[h]
	\begin{minipage}[h]{0.49\linewidth}
	\begin{center}
		\scalebox{0.5}{
			\begin{tikzpicture}
			
			\draw[very thick] (6, 2) arc (-90:-180:4);
			\draw[very thick] (-6, 2) arc (-90:0:4);
			\draw[very thick] (6, -2) arc (90:180:4);
			\draw[very thick] (-6, -2) arc (90:0:4);
			
			\draw[thick] (-2,6) to[out = -30, in = -150] (2, 6);
			\draw[thick] (-2,6) to[out = 30, in = 150] (2, 6);
			
			\draw[thick] (-2,-6) to[out = -30, in = -150] (2, -6);
			\draw[thick][dashed] (-2,-6) to[out = 30, in = 150] (2, -6);
			
			\draw[thick] (-6,-2) to[out = 60, in = -60] (-6, 2);
			\draw[thick] (-6,-2) to[out = 120, in = -120] (-6, 2);
			
			\draw[thick] (6,-2) to[out = 60, in = -60] (6, 2);
			\draw[thick][dashed] (6,-2) to[out = 120, in = -120] (6, 2);
			
			\draw[thick][red] (-6 + 2.8284,-6+2.8284) to[out = 30, in = -120] (6 - 2.8284,6-2.8284);
			\draw[thick][red][dashed] (-6 + 2.8284,-6+2.8284) to[out = 60, in = -150] (6 - 2.8284,6-2.8284);
			
			\draw[thick][blue] (-6 + 2.8284,6-2.8284) to[out = -60, in = 150] (6 - 2.8284,-6+2.8284);
			\draw[thick][blue][dashed] (-6 + 2.8284,6-2.8284) to[out = -30, in = 120] (6 - 2.8284,-6+2.8284);
			
			\draw[thick][green] (6 - 3.4641, 4) to[out = -110, in = 110] (6 - 3.4641, -4);
			\draw[thick][green] (-6 + 3.4641, 4) to[out = -70, in = 70] (-6 + 3.4641, -4);
			
			\draw[thick][green][dashed] (-6 + 3.4641, 4) to[out = 20, in = 160] (6 - 3.4641,4);
			
			\draw[thick][green][dashed] (-6 + 3.4641, -4) to[out = -20, in = -160] (6 - 3.4641,-4);

			\node[scale = 2] at (0, 7.2) {$b_2$};	
			\node[scale = 2] at (0, -7.2) {$b_4$};
			\node[scale = 2] at (7.2, 0) {$b_3$};
			\node[scale = 2] at (-7.2, 0) {$b_1$};
			
			\node[scale = 2][green] at (2.2, 0) {$z$};
			
			\node[scale = 2][blue] at (-3.7, 3.3) {$x$};
			\node[scale = 2][red] at (-3.7, -3.3) {$y$};

			\end{tikzpicture}} \end{center}
		\end{minipage}
	\hfill
		\begin{minipage}[h]{0.49\linewidth}
		\begin{center}
			\scalebox{0.5}{
				\begin{tikzpicture}
				
				\draw [very thick, rotate=0] (0.3, -2.2) ellipse (5 and 1.8);
				\draw [very thick, rotate=60] (-0.3, 2.2) ellipse (5 and 1.8);
				\draw [very thick, rotate=-60] (-0.3, 2.2) ellipse (5 and 1.8);
				
				\draw[very thick] (0, 0) circle (6);
				\draw[very thick] (0, 4) circle (0.7);
				\draw[very thick] (3.464, -2) circle (0.7);
				\draw[very thick] (-3.464, -2) circle (0.7);

				\node[scale = 2] at (0, 4) {$b_2$};	
				\node[scale = 2] at (0, -6.7) {$b_4$};
				\node[scale = 2] at (3.464, -2) {$b_3$};
				\node[scale = 2] at (-3.464, -2) {$b_1$};
				
				\node[scale = 2][] at (0, -4.5) {$z$};
				
				\node[scale = 2][] at (-3.7, 2.5) {$x$};
				\node[scale = 2][] at (3.7, 2.5) {$y$};

				\end{tikzpicture}} \end{center}
	\end{minipage}
	\caption{Two views of the Lantern relation (the pictures are also taken from \cite{Primer}).}
	\label{Lan}
\end{figure}

Let $i = 3, 4$ and $N_1^i \in \N_1^i$. Let $U \subset \H$ be an admissible subgroup for $N_1^i$ given by a separating curve $\theta$. Consider the subgroup $\O^i_U \subseteq \H_2(\I_{N_1^i}, \Z)$ generated by the abelian cycles
$\{\A_{U}^j = \A(T_{\alpha_j} T_{\alpha'_j}^{-1}, T_\theta)\}$,
where $j = 1, 2, 3$ and $\alpha'_j$ is a curve disjoint from $N_1^i \cup \theta$ such that $\{\alpha_j, \alpha'_j\}$ is a bounding pair. Note that the choice of $\alpha'_j$ is not unique. We claim that all such possible curves are $\I_{N_1^i}$-equivalent. Indeed, for different choices of $\alpha'_j$ the multicurves $N_1^i \cup \alpha'_j$ are $\PMod(\S_{2, 1})$-equivalent, where $\S_{2, 1}$ is the genus $2$ subsurface on the left of $\theta$. Let $f \in \Stab_{\PMod(\S_{2, 1})}(N_1^i)$ be an element carrying this equivalence. By composing $f$ with a power of $T_{\alpha_j}T_{\alpha'_j}^{-1}$ we can assume that $f$ act trivially on the $\H_1(\S_{1, 2})$, since by construction $f$ is already acts trivially on the classes $[\alpha_1]$, $[\alpha_2]$ and $[\alpha_3]$. Hence we obtain $f \in \I(\S_{2, 1}) \subset \I$. Consequently, $\A_U^j$ does not depend on the choice of $\alpha'_j$.

\begin{lemma} \label{pres}
	Let $i = 3, 4$ and $N_1^i \in \N_1^i$. The abelian group $\O^i_U \cong \Z^2$ has a presentation with three generators 
	\begin{equation}\label{generat}
	\A(T_{\alpha_1} T_{\alpha'_1}^{-1}, T_\theta), \;  \A(T_{\alpha_2} T_{\alpha'_2}^{-1}, T_\theta), \; \A(T_{\alpha_3} T_{\alpha'_3}^{-1}, T_\theta),
	\end{equation}
	and one relation
	\begin{equation} \label{relat}
	\A(T_{\alpha_1} T_{\alpha'_1}^{-1}, T_\theta) + \A(T_{\alpha_2} T_{\alpha'_2}^{-1}, T_\theta) + \A(T_{\alpha_3} T_{\alpha'_3}^{-1}, T_\theta) = 0.
	\end{equation}
\end{lemma}
\begin{proof}
	First, let us prove the relation (\ref{relat}). We can assume that the curves $\theta, \alpha_1, \alpha'_1, \alpha_2,$ $\alpha'_2,$ $\alpha_3, \alpha'_3$ are shown in Fig. \ref{Lante}. The Lantern relation implies
	$$T_{\alpha'_3} T_{\alpha'_1} T_{\alpha'_2} = T_{\alpha_1} T_{\alpha_2} T_{\alpha_3} T_{\theta}.$$
	We can rewrite this equation as follows.
	$$(T_{\alpha_2} T_{\alpha_2}^{-1})(T_{\alpha_1} T_{\alpha_1}^{-1})(T_{\alpha_3} T_{\alpha_3}^{-1}) = T_{\theta}^{-1}.$$
	Hence
	$$\A(T_{\alpha_1} T_{\alpha'_1}^{-1}, T_\theta) + \A(T_{\alpha_2} T_{\alpha'_2}^{-1}, T_\theta) + \A(T_{\alpha_3} T_{\alpha'_3}^{-1}, T_\theta) =$$ 
	$$= \A((T_{\alpha_2} T_{\alpha_2}^{-1})(T_{\alpha_1} T_{\alpha_1}^{-1})(T_{\alpha_3} T_{\alpha_3}^{-1}), T_{\theta}) = \A(T_{\theta}^{-1}, T_{\theta}) = 0.$$
	
	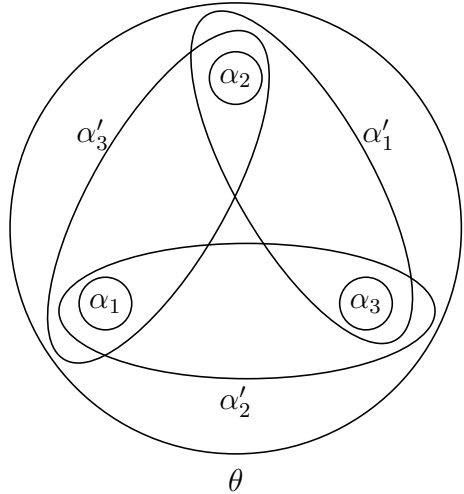
\begin{figure}[h]
			\begin{center}
				\scalebox{0.5}{
					\begin{tikzpicture}

					\draw [very thick, rotate=0] (0.3, -2.2) ellipse (5 and 1.8);
					\draw [very thick, rotate=60] (-0.3, 2.2) ellipse (5 and 1.8);
					\draw [very thick, rotate=-60] (-0.3, 2.2) ellipse (5 and 1.8);
					
					\draw[very thick] (0, 0) circle (6);
					\draw[very thick] (0, 4) circle (0.7);
					\draw[very thick] (3.464, -2) circle (0.7);
					\draw[very thick] (-3.464, -2) circle (0.7);

					\node[scale = 2] at (0, 4) {$\alpha_2$};	
					\node[scale = 2] at (0, -6.7) {$\theta$};
					\node[scale = 2] at (3.464, -2) {$\alpha_3$};
					\node[scale = 2] at (-3.464, -2) {$\alpha_1$};
					
					\node[scale = 2][] at (0, -4.6) {$\alpha'_2$};
					
					\node[scale = 2][] at (-3.8, 2.5) {$\alpha'_3$};
					\node[scale = 2][] at (3.8, 2.5) {$\alpha'_1$};

					\end{tikzpicture}} \end{center}
		\caption{The curves $\theta, \alpha_1, \alpha'_1, \alpha_2, \alpha'_2, \alpha_3, \alpha'_3.$}
		\label{Lante}
	\end{figure}

	Now let show that all relation among the three elements (\ref{generat}) follow from (\ref{relat}). Consider the homomorphisms $\nu_{\alpha_j, \W}$, where $\W = (U, U^{\perp_{H_{\alpha_j}}})$. Their restrictions on $\I_{N_1^i}$ represent cohomology classes $[\nu_{\alpha_j, \W}] \in \H^1(\I_{N_1^i}, \Z)$. Let us compute the value of their products on the abelian cycles $\A(T_{\alpha_j} T_{\alpha'_j}^{-1}, T_\theta)$.
	By Proposition \ref{hom} we have
	$$\left\langle [\nu_{\alpha_1, \W}] \smallsmile [\nu_{\alpha_2, \W}] ,  \A(T_{\alpha_1} T_{\alpha'_1}^{-1}, T_\theta) \right\rangle = -\det \begin{pmatrix}
	\nu_{\alpha_1, \W}(T_{\alpha_1} T_{\alpha'_1}^{-1}) & \nu_{\alpha_1, \W}(T_{\theta}) \\
	\nu_{\alpha_2, \W}(T_{\alpha_1} T_{\alpha'_1}^{-1}) & \nu_{\alpha_2, \W}(T_{\theta}) 
	\end{pmatrix} = 
	$$
	$$
	= -\det
	\begin{pmatrix}
	-1 & 1 \\
	0 & 1 
	\end{pmatrix} = 1 \neq 0.$$
	Similarly,
	$$\left\langle [\nu_{\alpha_1, \W}] \smallsmile [\nu_{\alpha_2, \W}] ,  \A(T_{\alpha_2} T_{\alpha'_2}^{-1}, T_\theta) \right\rangle = -1.$$
	Moreover, we have
	$$\left\langle [\nu_{\alpha_1, \W}] \smallsmile [\nu_{\alpha_2, \W}] ,  \A(T_{\alpha_3} T_{\alpha'_3}^{-1}, T_\theta) \right\rangle = -\det \begin{pmatrix}
	\nu_{\alpha_1, \W}(T_{\alpha_3} T_{\alpha'_3}^{-1}) & \nu_{\alpha_1, \W}(T_{\theta}) \\
	\nu_{\alpha_2, \W}(T_{\alpha_3} T_{\alpha'_3}^{-1}) & \nu_{\alpha_2, \W}(T_{\theta}) 
	\end{pmatrix} = 
	$$
	$$
	= -\det
	\begin{pmatrix}
	0 & 1 \\
	0 & 1 
	\end{pmatrix} = 0.$$
	
	Consider homology class 
	$$h = \sum_{j=1}^3 \lambda_j \A(T_{\alpha_j} T_{\alpha'_j}^{-1}, T_\theta) \in \O^i_U.$$
	If $h = 0$, then $$\left\langle [\nu_{\alpha_1, \W}] \smallsmile [\nu_{\alpha_2, \W}] ,  h \right\rangle = \left\langle [\nu_{\alpha_2, \W}] \smallsmile [\nu_{\alpha_3, \W}] ,  h \right\rangle = \left\langle [\nu_{\alpha_3, \W}] \smallsmile [\nu_{\alpha_1, \W}] ,  h \right\rangle = 0,$$
	that is
	$$\lambda_1 - \lambda_2 = \lambda_2 - \lambda_3 = \lambda_3 - \lambda_1 = 0.$$
	Hence $\lambda_1 = \lambda_2 = \lambda_3$.
	This concludes the proof. 
\end{proof}

\begin{lemma} \label{linearind3} 
	Let $i=3, 4$ and $N_1^i \in \N_1^i$. The inclusions $\O^i_U \hookrightarrow \H_2(\I_{N_1^i}, \Z)$ induce the injective homomorphism
	$$\bigoplus_U \O^i_U \hookrightarrow \H_2(\I_{N_1^i}, \Z),$$
	where the sum is over all admissible subgroup for $N_1^i$.
\end{lemma}
\begin{proof}
Assume the converse.
By Lemma \ref{pres} we can assume that there is a nontrivial linear dependence
\begin{equation} \label{rel}
\sum_{U \in S} \left( \lambda_{U, 2} \A_U^2 + \lambda_{U, 3} \A_U^3 \right) = 0,
\end{equation}
where $S$ is a finite set and $\lambda_{U, 2}, \lambda_{U, 3} \in \Z$. Let $U' \in S$. Let us compute the value of the cohomology class $[\nu_{\alpha_1, \W'}] \smallsmile [\nu_{\alpha_2, \W'}] \in \H^2(\I_{N_1^i}, \Z)$ on the both sides of (\ref{rel}), where $\W' = (U', U'^{\perp_{H_{\alpha_j}}})$.

Proposition \ref{hom} (c) implies that 
$$\nu_{\alpha_1, \W'}(T_{\alpha_3} T_{\alpha'_3}^{-1}) = 	\nu_{\alpha_2, \W'}(T_{\alpha_3} T_{\alpha'_3}^{-1}) = 0$$
for all possible choices of $\alpha'_3$, so we obtain
$$\left\langle [\nu_{\alpha_1, \W'}] \smallsmile [\nu_{\alpha_2, \W'}], \A_U^3 \right\rangle = 0$$
for all $U \in S$. Consequently, formula (\ref{rel}) implies
\begin{equation} \label{rel3}
\sum_{U \in S} \lambda_{U, 2} \left\langle [\nu_{\alpha_1, \W'}] \smallsmile [\nu_{\alpha_2, \W'}], \A_U^2 \right\rangle = 0.
\end{equation}
Let $U \in S$ is given by a separating curve $\theta$ disjoint from $N_1^i$ and $\A_{U}^2 = \A(T_{\alpha_2} T_{\alpha'_2}^{-1}, T_\theta)$. Then
\begin{equation} \label{eq1}
\left\langle [\nu_{\alpha_1, \W'}] \smallsmile [\nu_{\alpha_2, \W'}], \A_U^2 \right\rangle = -\det \begin{pmatrix}
\nu_{\alpha_1, \W'}(T_{\alpha_2} T_{\alpha'_2}^{-1}) & \nu_{\alpha_1, \W'}(T_{\theta}) \\
\nu_{\alpha_2, \W'}(T_{\alpha_2} T_{\alpha'_2}^{-1}) & \nu_{\alpha_2, \W'}(T_{\theta}) 
\end{pmatrix}.
\end{equation}
Since $\nu_{\alpha_1, \W'}(T_{\alpha_2} T_{\alpha'_2}^{-1}) = 0$, we have that (\ref{eq1}) is nonzero if and only if $\nu_{\alpha_1, \W'}(T_{\theta})$ and $\nu_{\alpha_2, \W'}(T_{\alpha_2} T_{\alpha'_2}^{-1})$ are nonzero. Therefore $U$ is admissible for $N_1^i \cup \alpha'_1$ and $N_1^i \cup \alpha'_2$. This condition defines $U$ uniquely, hence $U = U'$. Therefore (\ref{rel3}) implies $\lambda_{U', 2} = 0$. Similarly $\lambda_{U', 3} = 0$ for any $U' \in S$. This concludes the proof.	
\end{proof}

\subsection{The differential $d^1_{2, 2}$}

By Corollary \ref{freeE12} there is a basis of the free abelian group $E_{2, 2}^1$ consisting of the elements $P_{M^i_2} \otimes \A_{U^i}$, where $i = 4, 5$ and $U^i$ is admissible for $M^i_2 \in \M^i_2 / \I$. 
The differential $d_{2, 2}^{1}: E_{2, 2}^1 \to E_{1, 2}^1$ has the following form.
\begin{equation*}
	d_{2, 2}^1 \left(P_{M_2^i} \otimes \A_{U^i}\right) = \partial P_{M_2^i} \otimes \A_{U^i},
\end{equation*}
where $i = 4, 5$. 
For a fixed $i = 4, 5$ denote by $F^i \subset E_{2, 2}^1$ the subgroup generated by the elements $P_{M^i_2} \otimes \A_{U^i}$.

\begin{lemma} \label{lemind}
	The inclusions of the images $d_{2, 2}^1(F^4) \hookrightarrow E_{1, 2}^1$ and $d_{2, 2}^1(F^5) \hookrightarrow E_{1, 2}^1$ induce an injective map $d_{2, 2}^1(F^4) \oplus d_{2, 2}^1(F^5) \hookrightarrow E_{1, 2}^1$.
\end{lemma}
\begin{proof}
For $M_2^5 \in \M_2^5$ and for any $Q = P_M \subseteq \partial M_2^5$ we have that $\beta$ is a component of $M$. This follows from the fact that the homology class $[\beta]$ does not belong to the subgroup in $\H$ generated by the classes of the other components of $\M_2^5$. Therefore for any $U^5$ the image $d_{2, 2}^1(P_{M_2^5} \otimes \A_{U^5})$ belongs to the subgroup generated by the basic elements
\begin{equation} \label{imF5}
\{ P_{N_1^4} \otimes U, \;\;  P_{M_1^4} \otimes U, \;\; P_{M_1^3} \otimes U \},
\end{equation}
where in the last case we have $[\alpha] \in U$ (see Fig. \ref{M122}).

For $M_2^4 \in \M_2^4$ and for any $Q = P_M \subseteq \partial M_2^4$ we have that the image $d_{2, 2}^1(P_{M_2^4} \otimes \A_{U^4})$ belongs to the subgroup generated by the basic elements
\begin{equation} \label{imF4}
	\{ P_{N_1^3} \otimes U, \;\;  P_{M_1^2} \otimes U, \;\; P_{M_1^3} \otimes U \},
\end{equation}
where in the last case we have $[\alpha] \notin U$ (see Fig. \ref{M122}). The sets (\ref{imF5}) and (\ref{imF4}) do not intersect and these element are linearly independent by Corollary \ref{freeE12} and Lemma \ref{linearind3}. This concludes the proof.
\end{proof}

Moreover, Corollary \ref{freeE12} and Lemma \ref{linearind3} implies that for different $U$ the elements $P_M \otimes U$ (where $M \in \M_1(x)$) are also linearly independent in $E_{1, 2}^1$. Together with Lemma \ref{lemind}, this implies that in order to prove Proposition \ref{diff22} it suffices to prove the following simplified statement.

\begin{prop} \label{inj22}
Let $i = 4, 5$. Let us fix some symplectic subgroup $U \subset \H$ with $\rk U = 2$ and assume that we have 
\begin{equation} \label{dep}
	d_{2, 2}^{1} \left(\sum_{M \in S} \lambda_M P_{M} \otimes \A_{U}\right) = 0,
\end{equation}
where $M \in \M_2^i / \I$ are such that $U$ is admissible for $M$ and $S$ is a finite set. Then $\lambda_M = 0$ for all $M \in S$.
\end{prop}

\begin{proof}
First let us prove the case $i = 4$. Assume the converse.
For each $M \in S$ and for each edge $e \subset \partial M$ the restriction of the homology class $\A_U$ from $\I_M $ onto $\I_e$ is nonzero. Indeed, by Corollary \ref{H113} and Lemma \ref{pres} this image always can be chosen as a basic element of $\H_1(\I_e, \Z)$. Therefore (\ref{dep}) implies the following useful observation. 
\begin{prop} \label{edge}
	If $e$ is an edge such that $e \subset \partial P_M$ for some $M \in S$, then there exists $M \neq M' \in S$ such that $e \subset \partial  P_{M'}$.
\end{prop} 

Denote by $\Gamma \subset \B(x) / \I$ the subcomplex given by the closure of the union of all cells corresponding to the multicurves $M$ from the sets $\M_2^4$ such that $U$ is admissible for $M$. 

For a homology class $y \in U^{\perp}$ 
choose a bounding pair $\{\gamma^+, \gamma^-\}$ on $\S$ such that $[\gamma^+] = [\gamma^-] = y$ and $U$ is admissible for $\gamma^+ \cup \gamma^-$. Any multicurve $M$ from the set $\M_2^4 / \I$ containing $\gamma^+$ and $\gamma^-$ (see Fig. \ref{M22}) is uniquely determined (up to $\I$-equivalence) by the homology classes $[\alpha_1],[\alpha_2]$. Denote by $\Gamma_{y} \subset \Gamma$ the subcomplex spanned by these cells. 

Let $\alpha_1 = \alpha$ be a nonseparating curve such that $x = m [\alpha] + n y$ with $m, n > 0$. There are three possible situation: $[\alpha_2] = [\alpha] \pm y$ and $[\alpha_2] = y - [\alpha]$. Denote $l = \lfloor \frac{n}{m} \rfloor$. There are five possible types of edges in $\Gamma_y$. We denote them as follows.

$\bullet$ $d_k$ is the image in $\B(x) / \I$ of an edge in $\B(x)$ corresponding to a multicurve containing $\gamma^+$, $\gamma^-$ and a curve of the homology class $[\alpha]+ky$; $k \in \Z$, $k \leq l$. 

$\bullet$ $c^+_k$ is the image in $\B(x) / \I$ of an edge in $\B(x)$ corresponding to a multicurve containing $\gamma^+$ and curves of the homology classes $[\alpha]+ky, [\alpha]+(k+1)y$; $k \in \Z$, $k \leq l$.

$\bullet$ $c^-_k$ is the image in $\B(x) / \I$ of an edge in $\B(x)$ corresponding to a multicurve containing $\gamma^-$ and curves of the homology classes $[\alpha]+ky, [\alpha]+(k+1)y$; $k \in \Z$, $k \leq l$.

$\bullet$ $e^+_k$ is the image in $\B(x) / \I$ of an edge in $\B(x)$ corresponding to a multicurve containing $\gamma^+$ and curves of the homology classes $[\alpha]+ky, -[\alpha]-(l-1)y$; $k \in \Z$, $k \leq l$.

$\bullet$ $e^-_k$ is the image in $\B(x) / \I$ of an edge in $\B(x)$ corresponding to a multicurve containing $\gamma^-$ and curves of the homology classes $[\alpha]+ky, -[\alpha]-(l-1)y$; $k \in \Z$, $k \leq l$.

Note that the edges $c_k^+$ and $c_k^-$ ($e_k^+$ and $e_k^-$) are different. Indeed, in these cases the 3-punctured sphere in on opposite sides of the curve of the homology class $y$ (recall that we consider oriented multicurves).
However, the edges $c_k^+$ and $c_k^-$ ($e_k^+$ and $e_k^-$) have common endpoints (but these edges are not loops). Consequently, each edge $d_k$ is a loop in $\Gamma_y$.

Therefore $\Gamma_y$ is obtained from the complex $\Gamma'_y$ shown in Fig \ref{geom} by gluing together the pairs of endpoints of the edges $d_k$, where $k \in \Z$ and $k \leq l$. We are interested only on the structure of $1$- and $2$-cells of $\Gamma_y$, so it is convenient to work with $\Gamma'_y$.

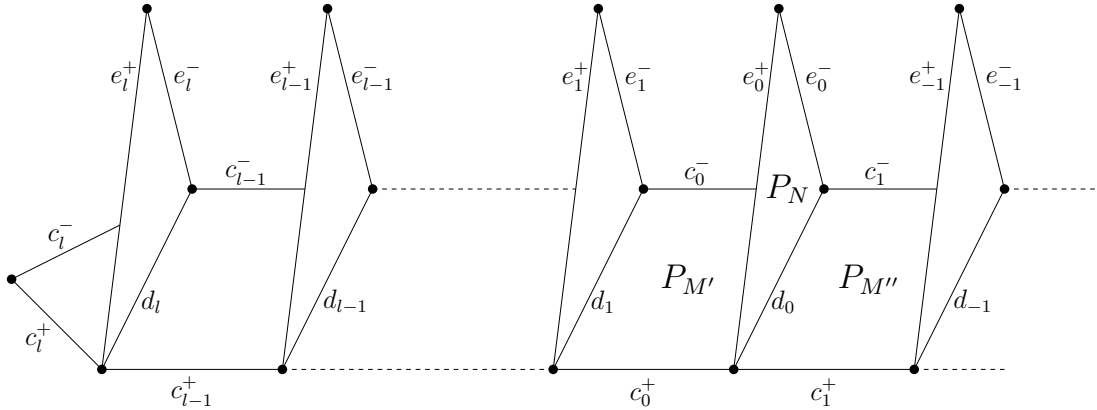
\begin{figure}[h]
	\begin{center} \scalebox{0.6}{
			\begin{tikzpicture}
			
			\draw[] (4, 0) to (8, 0);
			\draw[] (6, 4) to (8.5, 4);
			\draw[] (16, 4) to (18.5, 4);
			\draw[] (20, 4) to (22.5, 4);

			\draw[dashed] (8, 0) to (14, 0);
			\draw[dashed] (10, 4) to (14.5, 4);
			
			\draw[dashed] (22, 0) to (24, 0);
			\draw[dashed] (24, 4) to (26, 4);

			\draw[] (4, 0) to (6, 4);
			\draw[] (8, 0) to (10, 4);
			
			\draw[] (4, 0) to (2, 2);
			\draw[] (4.4, 3.2) to (2, 2);

			\draw[] (4, 0) to (5, 8);
			\draw[] (6, 4) to (5, 8);
			
			\draw[] (8, 0) to (9, 8);
			\draw[] (10, 4) to (9, 8);

			\draw[] (14, 0) to (22, 0);
			
			\draw[] (14, 0) to (16, 4);
			\draw[] (18, 0) to (20, 4);
			\draw[] (22, 0) to (24, 4);
			
			\draw[] (14, 0) to (15, 8);
			\draw[] (16, 4) to (15, 8);
			
			\draw[] (18, 0) to (19, 8);
			\draw[] (20, 4) to (19, 8);
			
			\draw[] (22, 0) to (23, 8);
			\draw[] (24, 4) to (23, 8);

            \node[scale = 1.4] at (2.6, 0.7) {$c^+_{l}$};
            \node[scale = 1.4] at (3.1, 3) {$c^-_{l}$};
			
		    \node[scale = 1.4] at (16, -0.5) {$c^+_0$};
		    \node[scale = 1.4] at (20, -0.5) {$c^+_1$};
			\node[scale = 1.4] at (6, -0.5) {$c^+_{l-1}$};
			
			\node[scale = 1.4] at (17.2, 4.4) {$c^-_0$};
			\node[scale = 1.4] at (21.2, 4.4) {$c^-_1$};
			\node[scale = 1.4] at (7.2, 4.4) {$c^-_{l-1}$};
			
			\node[scale = 1.4] at (5.1, 1.5) {$d_{l}$};
			\node[scale = 1.4] at (9.4, 1.5) {$d_{l-1}$};
			\node[scale = 1.4] at (15.1, 1.5) {$d_{1}$};
			\node[scale = 1.4] at (19.1, 1.5) {$d_{0}$};
			\node[scale = 1.4] at (23.3, 1.5) {$d_{-1}$};
			
			\node[scale = 1.4] at (4.5, 6.5) {$e^+_{l}$};
			\node[scale = 1.4] at (5.9, 6.5) {$e^-_{l}$};
			\node[scale = 1.4] at (8.2, 6.5) {$e^+_{l-1}$};
			\node[scale = 1.4] at (10, 6.5) {$e^-_{l-1}$};
			\node[scale = 1.4] at (14.5, 6.5) {$e^+_{1}$};
			\node[scale = 1.4] at (15.9, 6.5) {$e^-_{1}$};
			\node[scale = 1.4] at (18.5, 6.5) {$e^+_{0}$};
			\node[scale = 1.4] at (19.9, 6.5) {$e^-_{0}$};
			\node[scale = 1.4] at (22.3, 6.5) {$e^+_{-1}$};
			\node[scale = 1.4] at (24, 6.5) {$e^-_{-1}$};
			
			\node[scale = 1.8] at (17, 2) {$P_{M'}$};
			\node[scale = 1.8] at (21, 2) {$P_{M''}$};
			
			\node[scale = 1.8] at (19.2, 4) {$P_{N}$};

			\fill[black]  (2, 2) circle [radius=3pt];
			
			\fill[black]  (4, 0) circle [radius=3pt];
			\fill[black]  (8, 0) circle [radius=3pt];
			\fill[black]  (14, 0) circle [radius=3pt];
			\fill[black]  (18, 0) circle [radius=3pt];
			\fill[black]  (22, 0) circle [radius=3pt];
			
			\fill[black]  (6, 4) circle [radius=3pt];
			\fill[black]  (10, 4) circle [radius=3pt];
			\fill[black]  (16, 4) circle [radius=3pt];
			\fill[black]  (20, 4) circle [radius=3pt];
			\fill[black]  (24, 4) circle [radius=3pt];
			
			\fill[black]  (5, 8) circle [radius=3pt];
			\fill[black]  (9, 8) circle [radius=3pt];
			\fill[black]  (15, 8) circle [radius=3pt];
			\fill[black]  (19, 8) circle [radius=3pt];
			\fill[black]  (23, 8) circle [radius=3pt];

			\end{tikzpicture}} \end{center}
	\caption{The complex $\Gamma'_{y}$.}
	\label{geom}
\end{figure}

Obviously we have $$\Gamma = \bigcup_{y \in U^{\perp}} \Gamma_y,$$
and for $y \neq z$ the intersection $\Gamma_y \cap \Gamma_z$ consists of some boundary edges of $\Gamma_y$ and $\Gamma_z$. 

Let $\sigma$ be a cell of $\B$. Define the number $\Psi(\sigma) \in \mathbb{N}$ as follows. Each vertex $v$ of $\sigma$ corresponds to some multicurve $\gamma_1 \cup \dots \cup \gamma_t$ such that
the homology class $x$ can be uniquely represented as $x = \sum_{j=1}^t k_j [\gamma_j]$. Let us define $\psi(v) = \sum_{j=1}^t k_j$ and 
$$\Psi(\sigma) = \max_{v \in \sigma} \psi(v),$$
where maximum is taken over all vertices $v$ of $\sigma$. 

Each subcomplex $\Gamma_{y}$ has 'horizontal' 2-cells (all rectangles and the leftmost triangle) and 'vertical' 2-cells (all other triangles) as shown in Fig. \ref{geom}. The edges of the horizontal  2-cells are also said to be 'horizontal' (that is, $c_k^+, c_k^-$ and $d_k$ are horizontal and $e_k^+, e_k^-$ are vertical).
Let us remark that if (\ref{dep}) holds, then at least one of the cells $\{P_{M}, \; | \; M \in S\}$ is horizontal. This follows from the fact that the boundary of each vertical cell contains a horizontal edge that is not contained in the boundary of any other vertical 2-cells.

Let $M' \in S$ be a multicurve such that $P_{M'}$ is horizontal and $\Psi(P_{M'}) \geq \Psi(P_M)$ for any horizontal cell $P_M$, where $M \in S$. Without loss of generality we can assume that $M' = M_2^4$ shown in Fig \ref{M22}. Let $\alpha_1 = \alpha$ and $[\alpha_2] = [\alpha] + [\gamma^+]$, where $x = m[\alpha] + n[\gamma^+]$. The cell $P_{M'}$ is shown in Fig. \ref{geom} (the case when $P_{M'}$ triangular horizontal 2-cell is completely similar). We have $\Psi(P_{M'}) = m+n$.

Let $\beta$ be a curve disjoint from $\gamma^+$, $\gamma^-$ and $\alpha$ such that $[\beta] = [\alpha] - [\gamma^+]$.
Consider the multicurve $M'' = \gamma^+ \cup \gamma^- \cup \alpha \cup \beta$. The corresponding cell $P_{M''}$ is the rectangular 2-cell located on right hand side of $P_{M'}$ on Fig \ref{geom}. We have $$\Psi(P_{M''}) = 2m+n>m+n =\Psi(P_{M'}),$$ hence $M'' \notin S$. 
Consider the  multicurve $N = \gamma^+ \cup \gamma^- \cup \alpha \cup \beta$. The cell $P_N$ is the triangular 2-cell located on the top right of $P_{M'}$, see Fig \ref{geom}.
The edge $d_0 = P_{\gamma^+ \cup \gamma^- \cup \alpha}$ belongs to the boundary of exactly three 2-cells: $P_{M'}$, $P_{M''}$ and $P_{N}$. Since $M' \in S$ and $M'' \notin S$, Proposition \ref{edge} implies that $N \in S$.

Consider the edge $e^+_0 = P_{\gamma^+ \cup \alpha \cup \beta}$. Obviously $e^+_0 \subset \partial P_{N}$. By Proposition \ref{edge} there exist $L \in S$, such that $L \neq N$ and $e_0^+ \subset \partial P_L$. Therefore the homology classes of the components of $L$ are either $\{[\alpha], [\alpha], y, [\alpha]-y\}$ or $\{y - [\alpha], y-[\alpha], [\alpha], y\}$. Both of the corresponding cells are horizontal (the first cell is rectangular, the second one is triangular). In the first case we have 
$$\Psi(P_L) = 2m+n > m+n = \Psi(P_{M'}).$$
In the second case we have
$$\Psi(P_L) = m+2n > m+n = \Psi(P_{M'}).$$
Therefore we come to a contradiction in the case $i=4$.

Now consider the case $i=5$. The strategy is the same as for $i=4$. Let us remark that in (\ref{dep}) we can assume that all multicurves $M \in \M_2^5 / \I$ have the same component $\beta$ (see Fig. \ref{M22}) and the expansion of $x$ has the same coefficient $m$ at $[\beta]$. Then the previous argument works with replacing $x$ by $x - m[\beta]$.
\end{proof}

\begin{proof}[Proof of Proposition \ref{diff22}.]
	The result follows from Proposition \ref{inj22}.
\end{proof}

\section{Proof of Proposition \ref{surj1}} \label{S6}

\subsection{The term $E^1_{1, 3}$} 

In order to prove Proposition \ref{surj1} we need to compute the group $E^1_{1, 3}$ explicitly. There are six combinatorial types in $\M_1(x)$ shown in blue in Fig. \ref{M122}, \ref{M12} and $\ref{M'1}$. Namely, these are $\N_1^3$, $\N_1^4$, $\M'_1$, $\M_1^2$, $\M_1^3$ and $\M_1^4$. Indeed, any $M \in \M_1(x)$ divides $\S$ into two parts. If these parts have two common boundary components, the possible combinatorial types are $\M_1^2$, $\M_1^3$ and $\M_1^4$. For three common boundary components we have two options $\N_1^3$ and $\N_1^4$. Finally, if the there are four of them, then the only possible situation is $\M'_1$.
 
Let $N_1^3 \in \N_1^3$, $N_1^4 \in \N_1^4$ and $M'_1 \in \M'_1$.
Formula (\ref{cd}) applied to these multicurves implies  
$$\cd (\I_{N_1^3}) \leq 6 - 1 - 3 + 0 = 2,$$
$$\cd (\I_{N_1^4}) \leq 6 - 0 - 4 + 0 = 2,$$
$$\cd (\I_{M'_1}) \leq 6 - 0 - 4 + 0 = 2.$$
Hence we have $\H_3(\I_{N_1^3}, \Z) = \H_3(\I_{N_1^4}, \Z) = \H_3(\I_{M'_1}, \Z) = 0$.

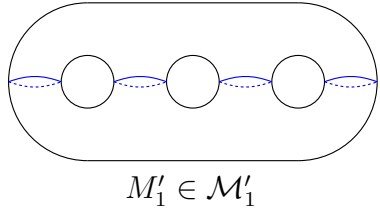
\begin{figure}[h]
	\begin{center}
	\scalebox{0.35}{
		\begin{tikzpicture}
		\draw[blue, very thick, dashed] (1,0) to[out = -20, in = 200] (3, 0);
		\draw[blue, very thick, dashed] (5,0) to[out = -20, in = 200] (7, 0);
		\draw[blue, very thick, dashed] (-3,0) to[out = -20, in = 200] (-1, 0);
		\draw[blue, very thick, dashed] (-7,0) to[out = -20, in = 200] (-5, 0);

		\draw[very thick] (-4, 3) to (4, 3);
		\draw[very thick] (-4, -3) to (4, -3);
		\draw[very thick] (-4, 3) arc (90:270:3);
		\draw[very thick] (4, 3) arc (90:-90:3);
		\draw[very thick] (0, 0) circle (1);
		\draw[very thick] (-4, 0) circle (1);
		\draw[very thick] (4, 0) circle (1);
		
		\draw[blue, very thick] (1,0) to[out = 20, in = 160] (3, 0);
		\draw[blue, very thick] (5,0) to[out = 20, in = 160] (7, 0);
		\draw[blue, very thick] (-3,0) to[out = 20, in = 160] (-1, 0);
		\draw[blue, very thick] (-7,0) to[out = 20, in = 160] (-5, 0);


		\end{tikzpicture}} \\  $M'_1 \in \M'_1$ \end{center}
	\caption{The combinatorial type $\M'_1$.}
	\label{M'1}
\end{figure}
Let $i = 1, 2, 3$ and $M_1^i \in \M_1^i$.
Let $U_1, U_2 \subset \H$ be admissible for symplectic subgroup for $M_1^i$ such that $U_1 \perp U_2$.
Define the homology class
$$\A_{U_1, U_2}  = \A(T_{\theta_1}, T_{\gamma^+} T_{\gamma^-}^{-1}, T_{\theta_2}) \in \H_3(\I_{\M_1^i}, \Z),$$
where $\theta_1$ and $\theta_2$ are separating curves disjoint from $\M_1^i$ such that $\H_{\theta_1} = U_1$ and $\H_{\theta_2} = U_2$. In order to distinguish the curves $\gamma^+$ and $\gamma^-$ let us agree that $\theta_1$ is on the right hand side with respect to $\gamma^+$ (recall cells of $\B$ correspond to oriented multicurves). Obviously we have $\A_{U_1, U_2} = -\A_{U_2, U_1}$. Proposition (\ref{H13}) immediately implies the following result.

\begin{prop}
Let $i = 2, 3, 4$ and $M_1^i \in \M_1^i$.
Then the homology classes $\A_{U_1, U_2}$ form a basis of the free abelian group $\H_3(\I_{M_1^i}, \Z)$. Here $U_1, U_2 \subset \H$ run over all unordered pairs orthogonal admissible symplectic subgroups for $M^i_1$.
\end{prop}

\begin{corollary} \label{E13gen}
	The elements 
	$$P_{M^i_1} \otimes \A_{U^i_1, U^i_2}, \; \; i =2, 3, 4$$
	form a basis of the free abelian group $E^1_{1, 3}$. Here $M^i_1 \in \M^i_1 / \I$, and $U^i_1, U^i_2 \subset \H$ run over all unordered pairs orthogonal admissible symplectic subgroups for $M^i_1$.
\end{corollary}

\subsection{The differential $d^1_{1, 3}$}

Our next goal is to compute the images under the differential $d^1_{1, 3}$ of the elements $P_{M^i_1} \otimes \A_{U^i_1, U^i_2} \in E^1_{1, 3}$ (see Corollary \ref{E13gen}). First we need to prove some general results, which will be used several times.

Let $\theta_1, \theta_2, \theta_3$ be pairwise disjoint separating curves on $\S$, see Fig. \ref{M12f}.
Consider the surface $\S_{2, 1} = \S \setminus \overline{X}_{\theta_1}$, where $\overline{X}_{\theta_1}$ is the closure of the once-punctured torus bounded by $\theta_1$. We have the exact sequences
\begin{equation*}\label{eq6}
\begin{CD}
1 @>>> \left\langle T_{\theta_1} \right\rangle @>>> \I_{\theta_1} @>{p}>> \I_{2, 1} @>>> 1.
\end{CD}
\end{equation*}
\begin{equation*}\label{BirmanT23}
\begin{CD}
1 @>>> \pi_1(\S_2, \pt) @>>> \I_{2, 1} @>{q}>> \I_2 @>>> 1.
\end{CD}
\end{equation*}
Consider the subgroups $Q = q^{-1}(\left\langle T_{\theta_3} \right\rangle ) \subset \I_{2, 1}$ and $G = p^{-1}(Q) \subset \I_{\theta_1}$. 
Let $\G \subseteq \I$ be a subgroup such that $G \subseteq \G$. Denote $Q' = p(\G)$. Suppose that $\BB(x) \subseteq \B(x)$ is a $\G$-invariant subcomplex. We will later prove Lemma \ref{dM2} concerning the action of $\G$ on $\BB(x)$. This result will be used several times for different $\G$ and $\BB(x)$. Namely, we will need the cases $\G = \I$, $\G = G$ and $\G = \I_\delta$, where $\delta$ is a separating or nonseparating curve on $\S$. The complex $\BB(x)$ will coincide either with the whole $\B(x)$, or with the \textit{auxiliary complex of cycles}, which will be defined in Section \ref{S7}.

\begin{figure}[h]
	\begin{minipage}[h]{0.49\linewidth}
		\begin{center}
			\scalebox{0.48}{
				\begin{tikzpicture}
					
					\draw[blue, very thick, dashed] (0,3) to[out = -110, in = 110] (0, 1);
					\draw[blue, very thick, dashed] (0,-1) to[out = -110, in = 110] (0, -3);
					\draw[blue, very thick, dashed] (5,0) to[out = -20, in = 200] (7, 0);
					\draw[blue, very thick, dashed] (-7,0) to[out = -20, in = 200] (-5, 0);
					\draw[red, very thick, dashed] (2,3) to[out = -100, in = 100] (2, -3);
					\draw[red, very thick, dashed] (-2,3) to[out = -100, in = 100] (-2, -3);

					\draw[very thick] (-4, 3) to (4, 3);
					\draw[very thick] (-4, -3) to (4, -3);
					\draw[very thick] (-4, 3) arc (90:270:3);
					\draw[very thick] (4, 3) arc (90:-90:3);
					\draw[very thick] (0, 0) circle (1);
					\draw[very thick] (-4, 0) circle (1);
					\draw[very thick] (4, 0) circle (1);
					
					\draw[blue, very thick] (0,3) to[out = -70, in = 70] (0, 1);
					\draw[blue, very thick] (0,-1) to[out = -70, in = 70] (0, -3);
					\draw[blue, very thick] (5,0) to[out = 20, in = 160] (7, 0);
					\draw[blue, very thick] (-7,0) to[out = 20, in = 160] (-5, 0);
					\draw[red, very thick] (2,3) to[out = -80, in = 80] (2, -3);
					\draw[red, very thick] (-2,3) to[out = -80, in = 80] (-2, -3);

					\node[blue, scale = 2] at (-0.8, 2) {$\alpha_2$};
					\node[blue, scale = 2] at (-0.8, -2) {$\alpha'_2$};
					
					\node[red, scale = 2] at (-2.7, 2) {$\theta_1$};
					\node[red, scale = 2] at (2.7, 2) {$\theta_3$};
					
					\node[blue, scale = 2] at (6, 0.5) {$\alpha_3$};
					\node[blue, scale = 2] at (-6, 0.5) {$\alpha_1$};
					
					\node[scale = 2] at (0.5, 0) {$\delta_2$};
					\node[scale = 2][green] at (1.4, -2.5) {$\delta_1$};
					
					\draw[green, very thick] (0, 1.2) arc (90:270:1.2);
					\draw[green, very thick] (0, 1.2) to[out = 0, in = -120] (0.8, 3);
					\draw[green, very thick] (0, -1.2) to[out = 0, in = 120] (0.8, -3);
					\draw[green, very thick, dashed] (0.8,3) to[out = -75, in = 75] (0.8, -3);

		\end{tikzpicture}} \end{center}
	\end{minipage}
	\hfill
	\begin{minipage}[h]{0.49\linewidth}
		\begin{center}
			\scalebox{0.72}{
				\begin{tikzpicture}
					\draw[blue, thick, dashed] (5.196 + 1 - 1.732 + 0.385, 2 + 1.732 + 0.58) to[out = 40, in = 260] (5.196 + 1 - 1.732 + 1, 2 + 1.732 + 1.732);
					\draw[blue, thick, dashed] (5.196 + 1 - 1.732 + 0.385, -2 - 1.732 - 0.58) to[out = -40, in = -260] (5.196 + 1 - 1.732 + 1, -2 - 1.732 - 1.732);
					
					\draw[blue, thick, dashed] (1, -2.26) to[out = 50, in = 250] (5.196 + 1 - 1.732 - 0.385, 2 + 1.732 - 0.58);
					
					\draw[blue, thick, dashed] (-4, 0) to[out = -20, in = 200] (-2-0.67, 0);
					\draw[red, thick, dashed] (0,2) to[out = -100, in = 100] (0, -2);
					\draw[red, thick, dashed] (1.732, 3) to[out = -40, in = 160] (5.196, 1);
					\draw[red, thick, dashed] (1.732, -3) to[out = 40, in = -160] (5.196, -1);
					
					\draw[orange, thick, dashed] (1, -2.26) to[out = 40, in = 200] (4.9, 0);
					
					\draw[orange, thick, dashed] (1, 2.26) to[out = 0, in = 220] (5.196 + 1 - 1.732 - 0.385, 2 + 1.732 - 0.58);

					\draw[thick] (0, 2) to (-2, 2);
					\draw[thick] (0, -2) to (-2, -2);
					\draw[thick] (-2, 2) arc (90:270:2);
					
					\draw[thick] (0, 2) to [out=0, in=240] (1.732, 3);
					\draw[thick] (0, -2) to [out=0, in=-240] (1.732, -3);
					
					\draw[thick] (5.196, 1) to [out=-120, in=120] (5.196, -1);
					
					\draw[thick] (5.196, 1) to (5.196 + 1, 1+1.732);
					\draw[thick] (1.732, 3) to (1.732 + 1, 3+1.732);
					
					\draw[thick] (5.196, -1) to (5.196 + 1, -1-1.732);
					\draw[thick] (1.732, -3) to (1.732 + 1, -3-1.732);
					
					\draw[thick] (5.196 + 1, 1+1.732) arc (-30:150:2);
					\draw[thick] (5.196 + 1, -1-1.732) arc (30:-150:2);
					
					\draw[thick] (-2, 0) circle (0.67);
					\draw[thick] (5.196 + 1 - 1.732, 2 + 1.732) circle (0.67);
					\draw[thick] (5.196 + 1 - 1.732, -2 - 1.732) circle (0.67);
					
					\draw[blue, thick] (5.196 + 1 - 1.732 + 0.385, 2 + 1.732 + 0.58) to[out = 80, in = 220] (5.196 + 1 - 1.732 + 1, 2 + 1.732 + 1.732);
					\draw[blue, thick] (5.196 + 1 - 1.732 + 0.385, -2 - 1.732 - 0.58) to[out = -80, in = -220] (5.196 + 1 - 1.732 + 1, -2 - 1.732 - 1.732);
					
					\draw[blue, thick] (1, -2.26) to[out = 70, in = 230] (5.196 + 1 - 1.732 - 0.385, 2 + 1.732 - 0.58);
					
					\draw[blue, thick] (-4, 0) to[out = 20, in = 160] (-2-0.67, 0);
					\draw[red][thick] (0,2) to[out = -80, in = 80] (0, -2);
					\draw[red, thick] (1.732, 3) to[out = -20, in = 140] (5.196, 1);
					\draw[red, thick] (1.732, -3) to[out = 20, in = -140] (5.196, -1);
					
					\draw[orange, thick] (1, -2.26) to[out = 80, in = -80] (1, 2.26);
					
					\draw[orange, thick] (4.9, 0) to[out = 120, in = 260] (5.196 + 1 - 1.732 - 0.385, 2 + 1.732 - 0.58);

					\node[red][scale = 1.5] at (-0.6, 0) {$\theta_1$};
					
					\node[orange][scale = 1.5] at (0.9, 0) {$\beta$};
					
					\node[blue][scale = 1.5] at (2, 1) {$\alpha'_2$};
					
					\node[][scale = 1.5] at (5.196 + 1 - 1.732+0.3, 2 + 1.732) {$\delta_2$};
					\node[red][scale = 1.5] at (2.5, -3.3) {$\theta_3$};
					\node[red][scale = 1.5] at (2.5, 3.3) {$\theta_2$};
					
					\node[blue][scale = 1.5] at (5.196 + 1 - 1.732+0.15, 2 + 1.732+1.2) {$\alpha_2$};
					
					\node[blue][scale = 1.5] at (5.196 + 1 - 1.732+0.15, -2 - 1.732-1.2) {$\alpha_3$};
					
					\node[blue][scale = 1.5] at (-3.2, 0.4) {$\alpha_1$};
					
		\end{tikzpicture}} \end{center}
	\end{minipage}
	\caption{The multicurves $M_1^2 = \alpha_2 \cup \alpha'_2$, $M_1^3 = \alpha_1 \cup \alpha_2 \cup \alpha'_2$ and $M_1^4 = \alpha_1 \cup \alpha_2 \cup \alpha'_2 \cup \alpha_3$.}
	\label{M12f}
\end{figure}
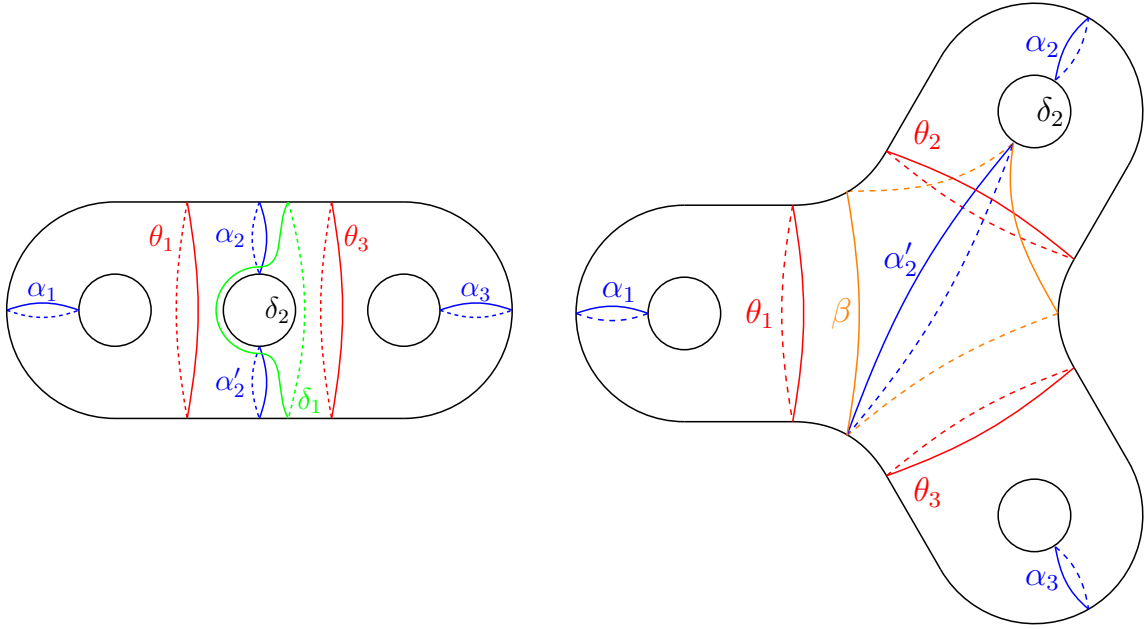

\begin{lemma}\label{dM2}
	Consider the multicurves $M_1^2 = \alpha_2 \cup \alpha'_2$, $M_1^3 = \alpha_1 \cup \alpha_2 \cup \alpha'_2$ and $M_1^4 = \alpha_1 \cup \alpha_2 \cup \alpha'_2 \cup \alpha_3$ shown in blue in Fig. \ref{M12f}. Consider the homology class $\A(T_{\theta_1}, T_{\theta_2}, T_{\theta_3}) \in \H_3(\G_{M^i_1 \setminus \alpha'_2}, \Z)$.
	
	Suppose that $i = 2, 3, 4$. Assume that the cells $P_{M^i_1} \subset \B(x)$ are well-defined and belong to $\BB(x)$. Denote by $(E'^{*}_{*, *}, d'^{*}_{*, *})$ the spectral sequence (\ref{spec_sec}) for the action of $\G$ on $\BB(x)$. 
	Then under the above assumptions we have $$d'^1_{1, 3} (P_{M^i_1} \otimes \A(T_{\theta_1}, T_{\alpha_2} T_{\alpha'_2}^{-1}, T_{\theta_3})) = P_{M^i_1 \setminus \alpha'_2} \otimes \A(T_{\theta_1}, T_{\theta_2}, T_{\theta_3}).$$
\end{lemma}

\begin{proof}
	Consider the curves $\delta_1$ and $\delta_2$ shown in Fig. \ref{M12f}. Note that $T_{\delta_2} T_{\delta_1}^{-1} \in G \subseteq \G$. A direct computation show that $T_{\delta_2} T_{\delta_1}^{-1} (\alpha'_2) = \alpha_2$ and $T_{\delta_2} T_{\delta_1}^{-1} (\alpha_2) = \beta$. Also by Lantern relation we have 
	$$T_{\theta_2} T_{\alpha'_2} T_{\beta} = T_{\theta_1} T_{\alpha_2}^2 T_{\theta_3},$$
	that is,
	\begin{equation} \label{lan}
	T_{\theta_1}^{-1} T_{\theta_2} T_{\theta_3}^{-1} = T_{\alpha'_2}^{-1} T_{\alpha_2}^2 T_{\beta}^{-1}.
	\end{equation}
Therefore we have $$d'^1_{1, 3} (P_{M^2_i} \otimes \A(T_{\theta_1}, T_{\alpha_2} T_{\alpha'_2}^{-1}, T_{\theta_3}) =$$
$$ =  P_{M^i_1 \setminus \alpha'_2} \otimes \A(T_{\theta_1}, T_{\alpha_2} T_{\alpha'_2}^{-1}, T_{\theta_3}) - P_{M^i_1 \setminus \alpha_2} \otimes \A(T_{\theta_1}, T_{\alpha_2} T_{\alpha'_2}^{-1}, T_{\theta_3})= $$
	$$= P_{M^i_1 \setminus \alpha'_2} \otimes \A(T_{\theta_1}, T_{\alpha_2} T_{\alpha'_2}^{-1}, T_{\theta_3}) - \left( (T_{\delta_2} T_{\delta_1}^{-1}) \cdot P_{M^i_1 \setminus \alpha_2} \right) \otimes \left( (T_{\delta_2} T_{\delta_1}^{-1}) \cdot \A(T_{\theta_1}, T_{\alpha_2} T_{\alpha'_2}^{-1}, T_{\theta_3}) \right) = $$
	$$= P_{M^i_1 \setminus \alpha'_2} \otimes \left(\A(T_{\theta_1}, T_{\alpha_2} T_{\alpha'_2}^{-1}, T_{\theta_3}) - \A(T_{\theta_1}, T_{\beta} T_{\alpha_2}^{-1}, T_{\theta_3}) \right) = $$
	$$= P_{M^i_1 \setminus \alpha'_2} \otimes \left(\A(T_{\theta_1}, T_{\alpha'_2}^{-1} T_{\alpha_2}, T_{\theta_3}) + \A(T_{\theta_1}, T_{\alpha_2} T_{\beta}^{-1}, T_{\theta_3}) \right) = $$
	$$= P_{M^i_1 \setminus \alpha'_2} \otimes \A(T_{\theta_1}, T_{\alpha'_2}^{-1} T_{\alpha_3}^2 T_{\beta}^{-1}, T_{\theta_3})  = P_{M^i_1 \setminus \alpha'_2} \otimes \A(T_{\theta_1}, T_{\theta_1}^{-1} T_{\theta_2} T_{\theta_3}^{-1}, T_{\theta_3}) =$$
	$$= P_{M^i_1 \setminus \alpha'_2} \otimes \A(T_{\theta_1},  T_{\theta_2}, T_{\theta_3}).$$
\end{proof}

\begin{lemma} \label{zero}
	Let $i = 2, 3$. Then the abelian cycle $\A(T_{\theta_1}, T_{\theta_2}, T_{\theta_3}) \in \H_3(\G_{M^i_1 \setminus \alpha'_2}, \Z)$ is zero.
\end{lemma}
\begin{proof}
	Since $G \subseteq G'$ and $\theta_1 \cup \alpha_1 \cup \alpha_2 \subseteq M^i_1 \setminus \alpha'_2$, we have $G_{\theta_1 \cup \alpha_1 \cup \alpha_2} \subseteq G'_{M^i_1 \setminus \alpha'_2}$. Hence it suffices to prove that the abelian cycle $\A(T_{\theta_1}, T_{\theta_2}, T_{\theta_3}) \in \H_3(G_{\theta_1 \cup \alpha_1 \cup \alpha_2}, \Z)$ is zero. We have $\A(T_{\theta_1}, T_{\theta_2}, T_{\theta_3}) = \A(T_{\theta_2} T_{\theta_3}^{-1}, T_{\theta_3}, T_{\theta_1})$. 
	
	Let $\S_{2, 1} \subset \S$ be the genus $2$ subsurface bounded by $\theta_1$ and let $\I_{2, 1}$ be its Torelli group. We have the exact sequence
	$$1 \to \left\langle T_{\theta_1} \right\rangle \to G_{\theta_1 \cup \alpha_1 \cup \alpha_2} \to Q_{\alpha_2} \to 1,$$
	where $Q_{\alpha_2} = \Stab_{Q} (\alpha_2)$. 
	
	Fact \ref{HS2} implies that there is an isomorphism
	$$\H_3(G_{\theta_1 \cup \alpha_1 \cup \alpha_2}, \Z) \cong \H_2(Q_{\alpha_2}, \Z),$$
	that maps $\A(T_{\theta_3} T_{\theta_2}^{-1}, T_{\theta_2}, T_{\theta_1})$ to $\A(T_{\theta_3} T_{\theta_2}^{-1}, T_{\theta_2})$. Hence it suffices to show that the homology class $\A( T_{\theta_2}, T_{\theta_3} T_{\theta_2}^{-1}) \in \H_2(Q_{\alpha_2}, \Z)$ is zero.

		\begin{figure}[h]
		\begin{center}
			\scalebox{0.6}{
				\begin{tikzpicture}
					\draw[blue, thick, dashed] (5.196 + 1 - 1.732 + 0.385, 2 + 1.732 + 0.58) to[out = 40, in = 260] (5.196 + 1 - 1.732 + 1, 2 + 1.732 + 1.732);
					
					
					
					

					\draw[thick] (0, 2) arc (90:270:2);
					
					\draw[thick] (0, 2) to [out=0, in=240] (1.732, 3);
					\draw[thick] (0, -2) to [out=0, in=-240] (1.732, -3);
					
					\draw[thick] (5.196, 1) to [out=-120, in=120] (5.196, -1);
					
					\draw[thick] (5.196, 1) to (5.196 + 1, 1+1.732);
					\draw[thick] (1.732, 3) to (1.732 + 1, 3+1.732);
					
					\draw[thick] (5.196, -1) to (5.196 + 1, -1-1.732);
					\draw[thick] (1.732, -3) to (1.732 + 1, -3-1.732);
					
					\draw[thick] (5.196 + 1, 1+1.732) arc (-30:150:2);
					\draw[thick] (5.196 + 1, -1-1.732) arc (30:-150:2);
					
					\draw[thick] (5.196 + 1 - 1.732, 2 + 1.732) circle (0.67);
					\draw[thick] (5.196 + 1 - 1.732, -2 - 1.732) circle (0.67);
					
					\draw[blue, thick] (5.196 + 1 - 1.732 + 0.385, 2 + 1.732 + 0.58) to[out = 80, in = 220] (5.196 + 1 - 1.732 + 1, 2 + 1.732 + 1.732);
					
					
					
					

					\node[red][scale = 1.8] at (-2.5, 0) {$\theta_1$};
					
					
					
					
					\node[blue][scale = 1.8] at (5.196 + 1 - 1.732+0.15, 2 + 1.732+1.4) {$\alpha_2$};
					
					

					\draw[red][thick] (-2,0) to[out = 10, in = 170] (4.9, 0);
					
					\draw[dashed][red][thick] (-2,0) to[out = -10, in = -170] (4.9, 0);
					\draw[red][thick] (1, 0.35) to (1+ 0.2, 0.35 + 0.2);
					\draw[red][thick] (1, 0.35) to (1+ 0.2, 0.35 - 0.2);
					
					\node[red][scale = 1.8] at (3, 0.6) {$\omega$};
					
					\draw[green][thick] (-2,0) to[out = -40, in = 100] (5.196 + 1 - 1.732, -2 - 1.732+0.67);
					\draw[green][thick][dashed] (-2,0) to[out = -60, in = 120] (5.196 + 1 - 1.732, -2 - 1.732+0.67);
					
					\draw[orange][thick] (-2,0) to[out = -25, in = 90] (5.196 + 1 - 1.732+1.5, -2 - 1.732);
					
					\draw[thick][orange] (5.196 + 1 - 1.732+1.5, -2 - 1.732) arc (0:-180:1.5);
					
					\draw[orange][thick] (5.196 + 1 - 1.732+1.5, -2 - 1.7325) to (5.196 + 1 - 1.732+1.5 + 0.2, -2 - 1.732 + 0.2);
					\draw[orange][thick] (5.196 + 1 - 1.732+1.5, -2 - 1.732) to (5.196 + 1 - 1.732+1.5 - 0.2, -2 - 1.732 + 0.2);

					\draw[orange][thick] (-2,0) to[out = -75, in = 90] (5.196 + 1 - 1.732-1.5, -2 - 1.732);

					\fill[red]  (-2, 0) circle [radius=5pt];
					
					\node[green][scale = 1.8] at (5.196 + 1 - 1.732 - 0.7, -2 - 1.732+0.67 + 0.3) {$a$};
					
					\node[orange][scale = 1.8] at (4.2, -1) {$b$};

					\draw[green][thick] (3.97, -2) to (3.97+0.28, -2);
					\draw[green][thick] (3.97, -2) to (3.97, -2-0.28);

		\end{tikzpicture}} \end{center}
		\caption{The surface $\S_2$ given by capping the puncture on $\S_{2, 1}$ corresponding to $\theta_1$.}
		\label{M12ff}
	\end{figure}
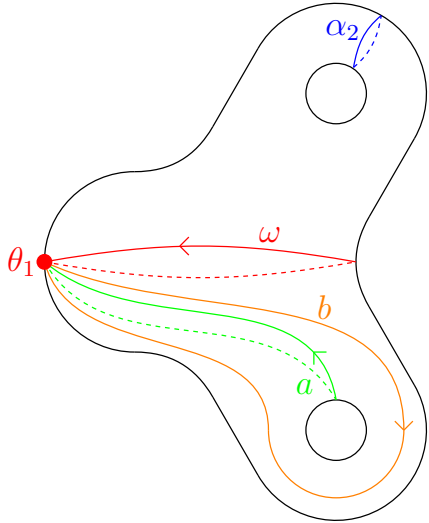

	Let $\S_2$ be the genus $2$ surface given by capping the puncture on $\S_{2, 1}$, see Fig. \ref{M12ff}.
	By our assumptions on $G$ we have the exact sequence
	$$1 \to \pi_1(\S_2 \setminus \alpha_2, \pt) \to Q_{\alpha_2} \to J \to 1,$$
	where $J \subseteq \Stab_{\I_2}(\alpha_2)$.
	The groups  $\pi_1(\S_2 \setminus \alpha_2, \pt)$ and $\Stab_{\I_2}(\alpha_2)$ are free, so $J$ is also free. The action of $\Stab_{\I_2}(\alpha_2)$ on $\H_*(\pi_1(\S_2 \setminus \alpha_2, \pt), \Z)$ is trivial. 
	Note that mapping class $T_{\theta_3} T_{\theta_2}^{-1}$ is ``pushing'' the puncture along the loop $\omega$ on $\S_2 \setminus \alpha_2$ shown in Fig. \ref{M12ff}. Hence we have $T_{\theta_3} T_{\theta_2}^{-1} \in \pi_1(\S_2 \setminus \alpha_2, \pt)$. 
	Therefore Fact \ref{HS4} implies that there is an isomorphism
	$$\H_2(Q_{\alpha_2}, \Z) \cong \H_1(J, \Z) \otimes \H_1(\pi_1(\S_2 \setminus \alpha_2, \pt), \Z),$$
	that maps the abelian cycle $\A(T_{\theta_2}, T_{\theta_3} T_{\theta_2}^{-1})$ to $[T_{\theta_2}] \otimes [T_{\theta_3} T_{\theta_2}^{-1}]$.

	The key point here is that $\omega = T_{\theta_3} T_{\theta_2}^{-1}$ belongs to the commutator subgroup of $\pi_1(\S_2 \setminus \alpha_2, \pt)$. Indeed, in the group $\pi_1(\S_2 \setminus \alpha_2, \pt)$ we have the relation $\omega = aba^{-1}b^{-1}$, where $a$ and $b$ are shown in Fig. \ref{M12ff}. 
	Consequently, the homology class $[T_{\theta_3} T_{\theta_2}^{-1}] \in \H_1(\pi_1(\S_2 \setminus \alpha_2, \pt), \Z)$ is trivial. Therefore $[T_{\theta_2}] \otimes [T_{\theta_3} T_{\theta_2}^{-1}] = 0$. This concludes the proof.
\end{proof}

It turns out that in the case $i = 4$ the abelian cycle $\A(T_{\theta_1}, T_{\theta_3}, T_{\theta_2}) \in \H_3(\I_{M^i_1}, \Z)$ is nonzero. Now let us consider this case in more detail.
Consider the multicurve $N = \alpha_1 \cup \alpha_2 \cup \alpha_3$ shown in Fig. \ref{M12f}. Fore each admissible symplectic splitting $\H = U_1 \oplus U_2 \oplus U_3$ we can take disjoint admissible separating curves $\theta_j$ with $\H_{\theta_j} = U_j$, $j = 1, 2, 3$. Let us denote
$$\A_{U_1, U_2, U_3} = \A(T_{\theta_1}, T_{\theta_2}, T_{\theta_3}) \in \H_3(\I_N, \Z).$$

The main goal of the next two subsections is to prove the following lemma.
\begin{lemma} \label{ind}
	The abelian cycles $\A_{U_1, U_2, U_3}$ is a basis of a free abelian subgroup in $\H_3(\I_N, \Z)$. Here $\{U_1, U_2, U_3\}$ runs over the set of all unordered admissible for $N$ symplectic splittings of $\H$.
\end{lemma}

In order to prove Lemma \ref{ind} we need to consider a new $CW$-complex called \textit{complex of relative cycles} introduced by the author in \cite{Spiridonov}. The idea is to introduce some analogue of $\B_g$ that makes sense for a sphere (i.e. $g=0$ case) with punctures.

\subsection{Complex of relative cycles}

Recall that by $\S_{0, 2g}$ we denote a sphere with $2g$ punctures. In order to construct the complex of relative cycles $\B_{0, 2g}$ we need some additional structure. Let us split the punctures into two disjoint sets: $P = \{p_1, \dots, p_g\}$ and $Q = \{q_1, \dots, q_g\}$.

By \textit{arc} on $\S_{0, 2g}$ we mean an embedded oriented curve with endpoints at punctures. By \textit{multiarc} we mean a disjoint union of arcs (common endpoints are allowed). We always consider arcs and multiarcs up to an isotopy.

Denote by $\D$ the set of isotopy classes of arcs starting at some point from $P$ and finishing at some point from $Q$. \textit{Relative basic 1-cycle} is a formal sum $\gamma = \gamma_{1}+\dots+\gamma_{g}$ where $\gamma_i \in \D$ such that

(1) in the group of singular 1-chains we have $\partial (\sum_{i=1}^{g}\gamma_{i}) = \sum_{i=1}^{g}(q_{i} - p_{i})$.

(2) we can choose pairwise disjoint representatives of the isotopy classes $\gamma_1, \dots, \gamma_g$,

The set $\gamma_{1} \cup \dots \cup \gamma_{g}$ is called the \textit{support} of $\gamma$. Denote by $\mathcal{L}$ the set multiarcs $L = \gamma_{1} \cup \dots \cup \gamma_s$ (for arbitrary $s$) satisfying the following property:

(i) fore each $1 \leq i \leq s$ there exists a relative basic 1-cycle supported in $L$ and contains $\gamma_i$.

For each $L \in \mathcal{L}$ let us denote by $P_L \subset \mathbb{R}_{+}^{\D}$ the convex hull of the relative basic 1-cycles supported in $L$. Obviously $P_L$ is a convex polytope. By definition complex of relative cycles is the regular $CW$-complex given by $\B_{0, 2g} = \cup_{L \in \mathcal{L}} P_L$. Denote by $\mathcal{L}_0 \subseteq \mathcal{L}$ the set of supports of all relative basic 1-cycles. Obviously $\{P_L \; | \; L \in \mathcal{L}_0\}$ is the set of 0-cells of $\B_{0, 2g}$.

\begin{remark} \label{rem}
	By construction $\B_{0, 2g}$ is the subset of $\mathbb{R}^\D$ consisting of the points (formal sums) $\sum_{i=1}^n k_i \gamma_i$ where $\gamma_i \in \D$ and $k_i \in \mathbb{R}_{\geq 0}$ satisfying the following conditions:
	
	(1) in the group of singular 1-chains we have $\partial (\sum_{i=1}^{n}k_i \gamma_{i}) = \sum_{i=1}^{g}(q_{i} - p_{i})$.
	
	(2) we can choose pairwise disjoint representatives of the isotopy classes $\gamma_1, \dots, \gamma_n$.
\end{remark}

\begin{theorem}\cite[Theorem 5.3]{Spiridonov} \label{contr}
	Let $g \geq 1$. Then $\B_{0, 2g}$ is contractible.
\end{theorem}

The group $\PMod(\S_{0, 2g})$ acts on $\B_{0, 2g}$ without rotations \cite[Theorem 5.5]{Spiridonov}

\begin{theorem}\cite[Theorem 5.5]{Spiridonov} \label{stabdim}
	Let $\sigma$ be a cell of $\B_{0, 2g}$. Then
	\begin{equation}\label{ineqdim}
	\dim(\sigma) + \cd(\Stab_{\PMod(\S_{0, 2g})}(\sigma)) \leq 2g-3.
	\end{equation}
\end{theorem}

Let $K \subseteq \PMod(\S_{0, 2g})$ be a subgroup. Denote by $(\widehat{E}_{*, *}^{*}, \widehat{d}^*_{*, *})$ the spectral sequence (\ref{spec_sec}) for the action of $K$ on $\B_{0, 2g}$. Theorem \ref{stabdim} implies that for any cell $\sigma$ of $\B_{0, 2g}$ we have 
\begin{equation} \label{cdineq}
\dim(\sigma) + \cd(\Stab_{K}(\sigma)) \leq 2g-3.
\end{equation}
This immediately implies $\widehat{E}^1_{p, q} = 0$ for $p+q > 2g-3$. Hence all differentials $\widehat{d}^1, \widehat{d}^2, \dots$ to the group $\widehat{E}^1_{0, 2g-3}$ are trivial, so $\widehat{E}^1_{0, 2g-3} = \widehat{E}^{\infty}_{0, 2g-3}$. Therefore we have the following result.
\begin{prop} \cite[Proposition 5.10]{Spiridonov} \label{prop2}
	Let $\mathfrak{L} \subseteq \mathcal{L}_{0}$ be a subset consisting of mutiarcs from pairwise disjoint $K$-orbits. Then the inclusions $\Stab_{K}(L) \subseteq K$, $L \in \mathfrak{L}$ induce the injective homomorhpism
	\begin{equation} \label{c-l3}
	\bigoplus_{L \in \mathfrak{L}} \H_{2g-3}(\Stab_{K}(L), \Z) \hookrightarrow \H_{2g-3}(K), \Z).
	\end{equation}
\end{prop} 

\begin{proof}[Proof of Lemma \ref{ind}.] Let $N = \alpha_1 \cup \alpha_2 \cup \alpha_3$ as in Fig. \ref{M12f}. Consider an admissible symplectic splitting $\U = \{U_1, U_2, U_3\}$ and let $\theta_j$ be separating curves such that $\H_{\theta_j} = U_j$, $j = 1, 2, 3$. Denote by $X_{\theta_j}$ the once-punctured torus bounded by $\theta_j$. Let $\beta'_j$ be a curve on $X_{\theta_j}$ such that the algebraic intersection number of $\alpha_j$ and $\beta'_j$ is one. Denote $B'_\U = \beta'_1 \cup \beta'_2 \cup \beta'_3$. The homology class $\A_{U_1, U_2, U_3} \in \H_3(\Stab_{\I_{N}}(B'_\U), \Z)$ is well defined. 

Consider the surface $\S_{0, 6} = \S \setminus N$; we have the inclusion $\I_{N} \hookrightarrow \PMod(\S_{0, 6})$. Denote $\beta_j = \beta'_j \cap \S_{0, 6}$ and consider the multiarc $B_\U = \beta_1 \cup \beta_2 \cup \beta_3$.
Obviously $\Stab_{I_{N}}(B'_\U) = \Stab_{I_N}(B_\U)$. 
One can easily check that $\Stab_{I_N}(B_\U) \cong \Z^3$ and the homology class $\A_{U_1, U_2, U_3} \in \H_3(\Stab_{I_{N}}(B_\U), \Z)$ is nonzero.

If $\U \neq \U'$ are two admissible symplectic splittings for $N$ it follows that $B_{\U}$ and $B_{\U'}$ belong to different $\I_N$-orbits. Let $\mathfrak{L} = \{\B_\U\}$, where $\U$ runs over the set of all admissible symplectic splittings for $N$, and $K = \I_N$. Proposition \ref{prop2} implies the result. \end{proof}

\subsection{The term $E^2_{1, 3}$} \label{6.5}

Consider an unordered symplectic splitting $\U = \{U_1, U_2, U_3\}$ of $\H$. There is the unique decomposition $x = x_1 + x_2 + x_3$, where $x_j \in U_j$. Let us introduce the following notation.
\begin{definition} \label{typex}
We say that the splitting $\U$

$\bullet$ is of type (a) w.r.t. $x$, if $x_1 \neq 0$ and $x_2 = x_3 = 0$;

$\bullet$ is of type (b) w.r.t. $x$, if $x_1, x_2 \neq 0$ and $x_3 = 0$;

$\bullet$ is of type (c) w.r.t. $x$, if $x_1, x_2, x_3 \neq 0$.
\end{definition}

Renumbering $U_1, U_2$ and $U_3$ we may achieve that one of the alternatives given in Definition \ref{typex} is true. Therefore, $U$ is of type (a), (b) or (c) w.r.t. $x$.

For $j=1, 2, 3$ let $x_j = k_j a_j$, where $a_j$ is a primitive homology class and $k_j \in \mathbb{N}$ (if $x_j = 0$ we put $a_j = 0$). Consider a bounding pair $\alpha_j, \alpha'_j$, with $[\alpha_j] = [\alpha'_j] = a_j$ and the splitting $\U$ is admissible for $\alpha_j \cup \alpha'_j$ (if $a_j = 0$ we define $\alpha_j = \alpha'_j = \varnothing$). All such bounding pairs are $\I$-equivalent.

If $\U$ is of type (a) w.r.t $x$, we consider the element 
\begin{equation} \label{type1}
P_{\alpha_1 \cup \alpha'_1} \otimes \A_{U_2, U_3} \in E_{1, 3}^1.
\end{equation}
If $\U$ is of type (b) w.r.t $x$, we consider the elements
\begin{equation} \label{type2}
P_{\alpha_1 \cup \alpha'_1 \cup \alpha_2} \otimes \A_{U_2, U_3}, P_{\alpha_1 \cup \alpha_2 \cup \alpha'_2} \otimes \A_{U_1, U_3}  \in E_{1, 3}^1.
\end{equation}
If $\U$ is of type (c) w.r.t $x$, we consider the elements
\begin{equation} \label{type3}
P_{\alpha_1 \cup \alpha'_1 \cup \alpha_2 \cup \alpha_3} \otimes \A_{U_2, U_3}, P_{\alpha_1 \cup \alpha_2 \cup \alpha'_2 \cup \alpha_3} \otimes \A_{U_3, U_1}, P_{\alpha_1 \cup \alpha_2 \cup \alpha_3 \cup \alpha'_3} \otimes \A_{U_1, U_2} \in E_{1, 3}^1.
\end{equation}
Note that the cells mentioned above are uniquely determined (up to $\I$-equivalence) by $\U$.
Corollary \ref{E13gen} implies that the elements (\ref{type1}), (\ref{type2}) and (\ref{type3}), where $\U$ runs over the set of all splittings of $\H$, form a basis of the free abelian group $E_{1, 3}^1$.

Lemmas \ref{dM2} and \ref{zero} (we take $\G = \I$ and $\BB(x) = \B(x)$ yield that the images of the elements (\ref{type1}) and (\ref{type2}) under the differential $d_{1, 3}^1$ are zero. By Lemma \ref{ind}, the image of each of three elements (\ref{type3}) is $$P_{\alpha_1 \cup \alpha_2 \cup \alpha_3} \otimes \A_{U_1, U_3, U_2} \in E_{0, 3}^1,$$
and these classes are linearly independent for different $\U$. Therefore we have the following result.

\begin{prop} \label{kernel}
	The free abelian group $E_{1, 3}^2 = \ker d_{1, 3}^1$ has a basis consisting of the following elements:
	$$P_{\alpha_1 \cup \alpha'_1} \otimes \A_{U_2, U_3},$$
	where $\U$ runs over the set of all splittings of $\H$ of type (a) w.r.t $x$,
	$$P_{\alpha_1 \cup \alpha'_1 \cup \alpha_2} \otimes \A_{U_2, U_3}, P_{\alpha_1 \cup \alpha_2 \cup \alpha'_2} \otimes \A_{U_1, U_3},$$
	where $\U$ runs over the set of all splittings of $\H$ of type (b) w.r.t $x$,
	$$P_{\alpha_1 \cup \alpha'_1 \cup \alpha_2 \cup \alpha_3} \otimes \A_{U_2, U_3} - P_{\alpha_1 \cup \alpha_2 \cup \alpha'_2 \cup \alpha_3} \otimes \A_{U_3, U_1}, P_{\alpha_1 \cup \alpha_2 \cup \alpha'_2 \cup \alpha_3} \otimes \A_{U_3, U_1} - P_{\alpha_1 \cup \alpha_2 \cup \alpha_3 \cup \alpha'_3} \otimes \A_{U_1, U_2},$$
	where $\U$ runs over the set of all splittings of $\H$ of type (c) w.r.t $x$. 
\end{prop}

\begin{proof}[Proof of Proposition \ref{surj1}.] 
Recall that for a separating curve $\gamma$ on $\S$ denote by $E^{(\gamma) *}_{*, *}$ the spectral sequence (\ref{spec_sec}) for the action of $\I_{\gamma}$ on $\B(x)$.
By $j^{(\gamma) *}_{*, *}: E^{(\gamma) *}_{*, *} \to E^*_{*, *}$ we denote the morphism of the spectral sequences induced by the inclusion $\iota_{\gamma}: \I_{\gamma} \hookrightarrow \I$.

We have the morphism of the spectral sequences
$$
J^*_{*, *}: \E^*_{*, *} \to E^*_{*, *}.
$$
Here $J_{*, *}^* = \bigoplus_\gamma j^{(\gamma) *}_{*, *}$ and $\E^*_{*, *} = \bigoplus_{\gamma} E^{(\gamma) *}_{*, *}$, where the sums are over all separating curves $\gamma$ on $\S$. Our goal is to prove the surjectivity of $J_{1, 3}^2$. In Proposition \ref{kernel} we constructed a basis of the free abelian group $E_{1, 3}^2$. Let us show that each of this elements belongs to the image of $J_{1, 3}^2$.
Let $\U = (U_1, U_2, U_3)$ be a symplectic splitting of $\H$. As before, let $\theta_1, \theta_2, \theta_3$ be separating curves disjoint from $\alpha_j$ and $\alpha'_j$ such that $\H_{\theta_j} = U_j$ ($j = 1, 2, 3$). There are three possible cases.

(a) \textit{$\U$ is of type (a) w.r.t $x$.} We need to check that $P_{\alpha_1 \cup \alpha'_1} \otimes \A_{U_2, U_3}$ belongs to the image of $J_{1, 3}^2$. Consider the group $\I_{\theta_2}$ and the homology class 
$$\A(T_{\theta_2}, T_{\alpha_1} T_{\alpha'_1}^{-1}, T_{\theta_3}) \in \H_3(\Stab_{\I_{\theta_2}}(\alpha_1 \cup \alpha'_1), \Z).$$
Let us consider the element 
$$P_{\alpha_1 \cup \alpha'_1} \otimes \A(T_{\theta_2}, T_{\alpha_1} T_{\alpha'_1}^{-1}, T_{\theta_3})\in E_{3, 1}^{(\theta_2) 1}.$$
Lemmas \ref{dM2} and \ref{zero} (we take $\G = \I_{\theta_2}$ and $\BB(x) = \B(x)$) yield that 
$$ d_{1, 3}^{(\theta_2) 1} \left( P_{\alpha_1 \cup \alpha'_1} \otimes \A(T_{\theta_2}, T_{\alpha_1} T_{\alpha'_1}^{-1}, T_{\theta_3}) \right) = 0.$$
Hence
$$P_{\alpha_1 \cup \alpha'_1} \otimes \A(T_{\theta_2}, T_{\alpha_1} T_{\alpha'_1}^{-1}, T_{\theta_3})\in E_{3, 1}^{(\theta_2) 2}.$$
Obviously we have
$$J_{1, 3}^2 \left( P_{\alpha_1 \cup \alpha'_1} \otimes \A(T_{\theta_2}, T_{\alpha_1} T_{\alpha'_1}^{-1}, T_{\theta_3}) \right) = P_{\alpha_1 \cup \alpha'_1} \otimes \A_{U_2, U_3}.$$

(b) \textit{$\U$ is of type (b) w.r.t $x$.} We need to check that $P_{\alpha_1 \cup \alpha'_1 \cup \alpha_2} \otimes \A_{U_2, U_3}$ belongs to the image of $J_{1, 3}^2$ (the case of $P_{\alpha_1 \cup \alpha_2 \cup \alpha'_2} \otimes \A_{U_1, U_3}$ is similar). Consider the group $\I_{\theta_2}$ and the homology class 
$$\A(T_{\theta_2}, T_{\alpha_1} T_{\alpha'_1}^{-1}, T_{\theta_3}) \in \H_3(\Stab_{\I_{\theta_2}}(\alpha_1 \cup \alpha'_1 \cup \alpha_2), \Z).$$
Let us consider the element 
$$P_{\alpha_1 \cup \alpha'_1 \cup \alpha_2} \otimes \A(T_{\theta_2}, T_{\alpha_1} T_{\alpha'_1}^{-1}, T_{\theta_3})\in E_{3, 1}^{(\theta_2) 1}.$$
Lemmas \ref{dM2} and \ref{zero} (we take $\G = \I_{\theta_2}$ and $\BB(x) = \B(x)$) yield that  
$$ d_{1, 3}^{(\theta_2) 1} \left( P_{\alpha_1 \cup \alpha'_1 \cup \alpha_2} \otimes \A(T_{\theta_2}, T_{\alpha_1} T_{\alpha'_1}^{-1}, T_{\theta_3}) \right) = 0.$$
Hence
$$P_{\alpha_1 \cup \alpha'_1 \cup \alpha_2} \otimes \A(T_{\theta_2}, T_{\alpha_1} T_{\alpha'_1}^{-1}, T_{\theta_3})\in E_{3, 1}^{(\theta_2) 2}.$$
Obviously we have
$$J_{1, 3}^2 \left( P_{\alpha_1 \cup \alpha'_1 \cup \alpha_2} \otimes \A(T_{\theta_2}, T_{\alpha_1} T_{\alpha'_1}^{-1}, T_{\theta_3}) \right) = P_{\alpha_1 \cup \alpha'_1 \cup \alpha_2} \otimes \A_{U_2, U_3}.$$

(c) \textit{$\U$ is of type (c) w.r.t $x$.} We need to check that 
$$P_{\alpha_1 \cup \alpha'_1 \cup \alpha_2 \cup \alpha_3} \otimes \A_{U_2, U_3} - P_{\alpha_1 \cup \alpha_2 \cup \alpha'_2 \cup \alpha_3} \otimes \A_{U_3, U_1}$$ 
belongs to the image of $J_{1, 3}^2$ (the case of 
$$P_{\alpha_1 \cup \alpha_2 \cup \alpha'_2 \cup \alpha_3} \otimes \A_{U_3, U_1} - P_{\alpha_1 \cup \alpha_2 \cup \alpha_3 \cup \alpha'_3} \otimes \A_{U_1, U_2}$$ 
is similar). Consider the group $\I_{\theta_3}$ 
and the homology classes 
$$\A(T_{\theta_2}, T_{\alpha_1} T_{\alpha'_1}^{-1}, T_{\theta_3}) \in \H_3(\Stab_{\I_{\theta_2}}(\alpha_1 \cup \alpha'_1 \cup \alpha_2 \cup \alpha_3), \Z)$$
and
$$\A(T_{\theta_3}, T_{\alpha_2} T_{\alpha'_2}^{-1}, T_{\theta_1}) \in \H_3(\Stab_{\I_{\theta_2}}(\alpha_1 \cup \alpha_2 \cup \alpha'_2 \cup \alpha_3), \Z).$$
Let us consider the elements
$$P_{\alpha_1 \cup \alpha'_1 \cup \alpha_2 \cup \alpha_3} \otimes \A(T_{\theta_2}, T_{\alpha_1} T_{\alpha'_1}^{-1}, T_{\theta_3}) \in E_{3, 1}^{(\theta_3) 1}$$
and
$$P_{\alpha_1 \cup \alpha_2 \cup \alpha'_2 \cup \alpha_3} \otimes \A(T_{\theta_3}, T_{\alpha_2} T_{\alpha'_2}^{-1}, T_{\theta_1}) \in E_{3, 1}^{(\theta_3) 1}.$$
Lemma \ref{dM2} (we take $\G = \I_{\theta_3}$ and $\BB(x) = \B(x)$) implies that 
$$ d_{1, 3}^{(\theta_3) 1} \left( P_{\alpha_1 \cup \alpha'_1 \cup \alpha_2 \cup \alpha_3} \otimes \A(T_{\theta_2}, T_{\alpha_1} T_{\alpha'_1}^{-1}, T_{\theta_3}) \right) = P_{\alpha_1 \cup \alpha_2 \cup \alpha_3} \otimes \A(T_{\theta_2},  T_{\theta_3}, T_{\theta_1})$$
and
$$ d_{1, 3}^{(\theta_3) 1} \left( P_{\alpha_1 \cup \alpha_2 \cup \alpha'_2 \cup \alpha_3} \otimes \A(T_{\theta_3}, T_{\alpha_2} T_{\alpha'_2}^{-1}, T_{\theta_1}) \right) = P_{\alpha_1 \cup \alpha_2 \cup \alpha_3} \otimes \A(T_{\theta_2},  T_{\theta_3}, T_{\theta_1}).$$
Hence
$$d_{1, 3}^{(\theta_3) 1} \left( P_{\alpha_1 \cup \alpha'_1 \cup \alpha_2 \cup \alpha_3} \otimes \A(T_{\theta_2}, T_{\alpha_1} T_{\alpha'_1}^{-1}, T_{\theta_3}) - P_{\alpha_1 \cup \alpha_2 \cup \alpha'_2 \cup \alpha_3} \otimes \A(T_{\theta_3}, T_{\alpha_2} T_{\alpha'_2}^{-1}, T_{\theta_1}) \right)  = 0.$$
Therefore
$$ \left( P_{\alpha_1 \cup \alpha'_1 \cup \alpha_2 \cup \alpha_3} \otimes \A(T_{\theta_2}, T_{\alpha_1} T_{\alpha'_1}^{-1}, T_{\theta_3}) - P_{\alpha_1 \cup \alpha_2 \cup \alpha'_2 \cup \alpha_3} \otimes \A(T_{\theta_3}, T_{\alpha_2} T_{\alpha'_2}^{-1}, T_{\theta_1}) \right) \in E_{3, 1}^{(\theta_3) 2}.$$
Obviously we have
$$J_{1, 3}^2 \left( P_{\alpha_1 \cup \alpha'_1 \cup \alpha_2 \cup \alpha_3} \otimes \A(T_{\theta_2}, T_{\alpha_1} T_{\alpha'_1}^{-1}, T_{\theta_3}) - P_{\alpha_1 \cup \alpha_2 \cup \alpha'_2 \cup \alpha_3} \otimes \A(T_{\theta_3}, T_{\alpha_2} T_{\alpha'_2}^{-1}, T_{\theta_1}) \right) = $$
$$= P_{\alpha_1 \cup \alpha'_1 \cup \alpha_2 \cup \alpha_3} \otimes \A_{U_2, U_3} - P_{\alpha_1 \cup \alpha_2 \cup \alpha'_2 \cup \alpha_3} \otimes \A_{U_3, U_1}.$$
This concludes the proof.
\end{proof}

\section{Proof of Proposition \ref{surj0}} \label{S7}

\subsection{Auxiliary complex of cycles}

Let us describe the structure of the group $E_{0, 4}^1$. There are three combinatorial types in $\M_0(x)$. Namely, we denote by $\M_0^i \subset \M_0(x)$ the subset consisting of all multicurves with precisely $i$ components. Let $M_0^1 = \alpha_1 \in \M_0^1$, $M_0^2 = \alpha_1 \cup \alpha_2 \in \M_0^1$ and $M_0^3 = \alpha_1 \cup \alpha_2 \cup \alpha_3 \in \M_0^1$ be any representatives, see Fig. \ref{M12f}. Formula (\ref{cd}) applied the multicurves $M_0^2$ and $M_0^3$ implies
$$\cd (\I_{M_0^2}) \leq 6 - 1 - 2 + 0 = 3,$$
$$\cd (\I_{M_0^3}) \leq 6 - 0 - 3 + 0 = 3,$$
Hence we have $\H_4(\I_{M_0^2}, \Z) = \H_4(\I_{M_0^3}, \Z) = 0$.

Let $\alpha = \alpha_1$ be a curve with $[\alpha] = x$.
Since any two homological curves on $\S$ are $\I$-equivalent, we have in isomorphism
$$\H_{4}(\I_{\alpha}, \Z) \cong E_{0, 4}^1,$$
given by $h \mapsto P_{\alpha} \otimes h$, where $h \in \H_{4}(\I_{\alpha}, \Z)$.

Consider any nonzero homology class $y \in \H_1(\S, \Z)$ with $y \neq x$ and $x \cdot y = 0$. Now we need to introduce auxiliary complex $\B_\alpha(y) \subset \B(y)$. By definition $\B_\alpha(y)$ consists of those cells $P_M \subset \B(y)$, for which we can choose a representative of $M$ disjoint from $\alpha$. Obviously $\B_\alpha(y) \subset \B(y)$ is a subcomplex and the group $\I_{\alpha}$ acts on $\B_\alpha(y)$ cellularly and without rotations. Moreover, for each cell $\sigma \in \B_\alpha(y)$ we have 
\begin{equation} \label{ineq2}
\dim(\sigma) + \cd(\Stab_{\I_{\alpha}}(\sigma)) \leq 4.
\end{equation}
\begin{prop} 
	$\B_\alpha(y)$ is contractible.
\end{prop}
\begin{proof}
	The surgical proof of contractibility of $\B(y)$ \cite[Section 5]{Bestvina} works with no modification.
\end{proof}

Denote by $\M_{\alpha, p}(y)$ the set of multicurves corresponding to the cells of $\B_\alpha(y)$ of dimension $p$.

Now let $(\widetilde{E}_{*,*}^*, \widetilde{d}_{*, *}^*)$ be the spectral sequence (\ref{spec_sec}) for the action on $\I_{\alpha}$ on $\B_\alpha(y)$.
This spectral sequence has the form 
\begin{equation} \label{spec_sec2}
\widetilde{E}_{p, q}^1 \cong \bigoplus_{M \in \M_{\alpha, p}(y) / \I_{\alpha}}\H_q(\Stab_{\I_{\alpha}} (M)) \Rightarrow \H_{p+q}(\I_{\alpha}, \Z),
\end{equation}
where by $\M_{\alpha, p}(y) / \I_{\alpha}$ we denote the set containing one representative from each $\I_{\alpha}$-orbit in the set $\M_{\alpha, p}(y)$.

\begin{corollary} \label{cor2}
	Let $\widetilde{E}_{*,*}^*$ be the spectral sequence (\ref{spec_sec2}).
	Then $\widetilde{E}^1_{p, q} = 0$ for $p+q > 4$.
\end{corollary}

Our next goal it to show that $\widetilde{E}_{0,4}^1 = 0$ (Lemma \ref{E04}) and to prove that the differentials $\widetilde{d}_{3, 1}^1$ and $\widetilde{d}_{2, 2}^1$ are injective (Lemmas \ref{di31} and \ref{di22}), see Fig. \ref{tildeE}. 

\begin{figure}[h]
	\begin{minipage}[h]{0.49\linewidth}
		\begin{center}
			\scalebox{1}{
				\begin{tikzpicture}
					\draw[->] (0, 0) to (0, 6);
					\draw[] (1, 0) to (1, 5);
					\draw[] (2, 0) to (2, 4);
					\draw[] (3, 0) to (3, 3);
					\draw[] (4, 0) to (4, 2);
					
					\draw[->] (0, 0) to (5, 0);
					\draw[] (0, 1) to (4, 1);
					\draw[] (0, 2) to (4, 2);
					\draw[] (0, 3) to (3, 3);
					\draw[] (0, 4) to (2, 4);
					\draw[] (0, 5) to (1, 5);
					
					\draw[->] (3.5, 1.5) to (2.5, 1.5);
					\draw[] (3.5, 1.5) arc (-90:90:0.1);
					
					\draw[->] (2.5, 2.5) to (1.5, 2.5);
					\draw[] (2.5, 2.5) arc (-90:90:0.1);
					
					\node[scale = 1][left] at (0, 0.5) {$0$};
					\node[scale = 1][left] at (0, 1.5) {$1$};
					\node[scale = 1][left] at (0, 2.5) {$2$};
					\node[scale = 1][left] at (0, 3.5) {$3$};
					\node[scale = 1][left] at (0, 4.5) {$4$};
					
					\node[scale = 1][below] at (0.5, 0) {$0$};
					\node[scale = 1][below] at (1.5, 0) {$1$};
					\node[scale = 1][below] at (2.5, 0) {$2$};
					\node[scale = 1][below] at (3.5, 0) {$3$};
					
					\node[scale = 1][left] at (0, 5.5) {$q$};
					\node[scale = 1][below] at (4.5, 0) {$p$};
					
					\node[scale = 2] at (0.5, 4.5) {$0$};
					
			\end{tikzpicture}} \\  $\widetilde{E}^1$ \end{center}
	\end{minipage}
	\hfill
	\begin{minipage}[h]{0.49\linewidth}
		\begin{center}
			\scalebox{1}{
				\begin{tikzpicture}
					\draw[->] (0, 0) to (0, 6);
					\draw[] (1, 0) to (1, 4);
					\draw[] (2, 0) to (2, 4);
					\draw[] (3, 0) to (3, 2);
					\draw[] (4, 0) to (4, 1);
					
					\draw[->] (0, 0) to (5, 0);
					\draw[] (0, 1) to (4, 1);
					\draw[] (0, 2) to (3, 2);
					\draw[] (0, 3) to (2, 3);
					\draw[] (0, 4) to (2, 4);
					
					\node[scale = 1][left] at (0, 0.5) {$0$};
					\node[scale = 1][left] at (0, 1.5) {$1$};
					\node[scale = 1][left] at (0, 2.5) {$2$};
					\node[scale = 1][left] at (0, 3.5) {$3$};
					\node[scale = 1][left] at (0, 4.5) {$4$};
					
					\node[scale = 1][below] at (0.5, 0) {$0$};
					\node[scale = 1][below] at (1.5, 0) {$1$};
					\node[scale = 1][below] at (2.5, 0) {$2$};
					\node[scale = 1][below] at (3.5, 0) {$3$};

					\node[scale = 1][left] at (0, 5.5) {$q$};
					\node[scale = 1][below] at (4.5, 0) {$p$};
					
			\end{tikzpicture}} \\  $\widetilde{E}^2$ \end{center}
	\end{minipage}
	\caption{The pages $\widetilde{E}^1$ and $\widetilde{E}^2$ and the differentials $\widetilde{d}^1_{3, 1}$ and $\widetilde{d}^1_{2, 2}$.}
	\label{tildeE}
\end{figure}
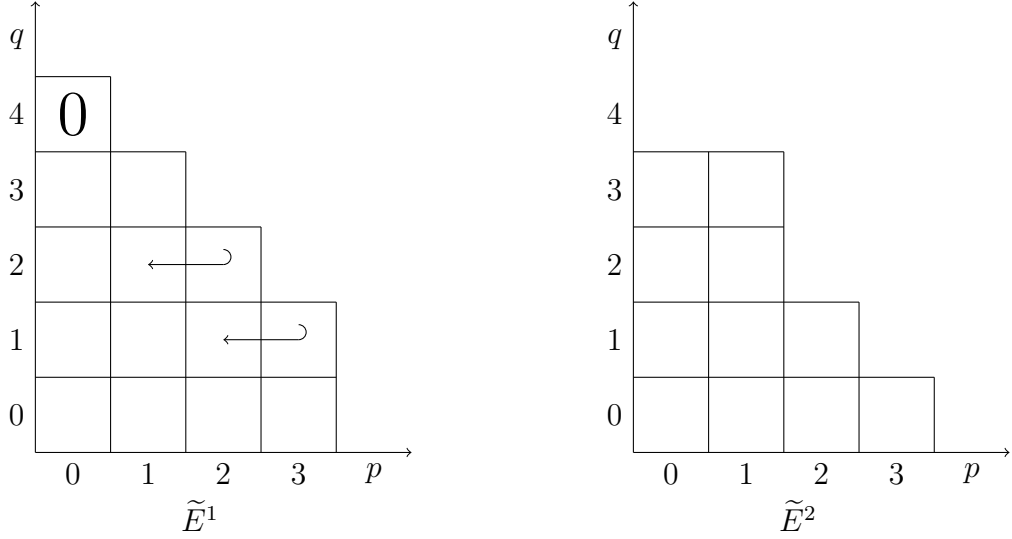

\begin{lemma} \label{E04}
	We have $\widetilde{E}_{0,4}^1 = 0$.
\end{lemma}
\begin{proof}
	Using formula (\ref{cd}) one can easily check that for any multicurve $M$ from the set $\M_{\alpha, 0}(y) / \Stab_{\Mod(\S)}(\alpha)$, we have $\cd(\Stab_{\I_{\alpha}}(M)) < 4$. This implies the result.
\end{proof}

\begin{lemma} \label{di31}
	The differential $\widetilde{d}_{3, 1}^1: \widetilde{E}_{3,1}^1 \to \widetilde{E}_{2,1}^1$ is injective, hence $\widetilde{E}_{3,1}^2 = \widetilde{E}_{3,1}^\infty = 0$.
\end{lemma}
\begin{proof}
	Let us consider the multicurves $M$ from the set $\M_{\alpha, 3}(y) / \Stab_{\Mod(\S)}(\alpha)$, such that $\cd(\Stab_{\I_{\alpha}}(M)) = 1$.
	Using formula (\ref{cd}) one can easily check that all the representatives are shown in Fig. \ref{A3}. The multicurve $M$ is shown in green, the curve $\alpha$ is shown in blue.
	
	Let a multicurve $M$ has one of the types shown in Fig. \ref{A3}. We have the isomorphism $\Stab_{\I_{\alpha}}(M) = \I_M$, so be obtain
	$$\H_{1}(\Stab_{\I_{\alpha}}(M), \Z) \cong \H_{1}(\I_M, \Z).$$ Moreover, if $M, M' \in \M_{\alpha, 3}(y)$ are $\I$-equivalent, it follows that $M$ and $M'$ are $\I_{\alpha}$-equivalent.
	This follows from the general fact that $\I$ acts on $\B(x)$ without rotations. Therefore, the injectivity of the differential $\widetilde{d}_{3, 1}^1$ immediately follows from the injectivity of the differential $d_{3, 1}^1$ (see Proposition \ref{diff31}).
	\begin{figure}[h]
		\begin{minipage}[h]{0.49\linewidth}
			\begin{center}
				\scalebox{0.35}{
					\begin{tikzpicture}
					\draw[green, very thick, dashed] (0,3) to[out = -110, in = 110] (0, 1);
					\draw[green, very thick, dashed] (0,-1) to[out = -110, in = 110] (0, -3);
					\draw[green, very thick, dashed] (1,0) to[out = -20, in = 200] (3, 0);
					\draw[green, very thick, dashed] (5,0) to[out = -20, in = 200] (7, 0);
					\draw[green, very thick, dashed] (-3,0) to[out = -20, in = 200] (-1, 0);
					\draw[green, very thick, dashed] (-7,0) to[out = -20, in = 200] (-5, 0);
					
					\draw[blue, very thick, dashed] (-7 + 0.1,0.5) to[out = -20, in = 200] (-5+0.1, 0.5);


					\draw[very thick] (-4, 3) to (4, 3);
					\draw[very thick] (-4, -3) to (4, -3);
					\draw[very thick] (-4, 3) arc (90:270:3);
					\draw[very thick] (4, 3) arc (90:-90:3);
					\draw[very thick] (0, 0) circle (1);
					\draw[very thick] (-4, 0) circle (1);
					\draw[very thick] (4, 0) circle (1);
					
					\draw[green, very thick] (0,3) to[out = -70, in = 70] (0, 1);
					\draw[green, very thick] (0,-1) to[out = -70, in = 70] (0, -3);
					\draw[green, very thick] (1,0) to[out = 20, in = 160] (3, 0);
					\draw[green, very thick] (5,0) to[out = 20, in = 160] (7, 0);
					\draw[green, very thick] (-3,0) to[out = 20, in = 160] (-1, 0);
					\draw[green, very thick] (-7,0) to[out = 20, in = 160] (-5, 0);
					
					\draw[blue, very thick] (-7+0.1, 0.5) to[out = 20, in = 160] (-5+0.1, 0.5);

					\node[blue, scale = 2] at (-5.9, 1.1) {$\alpha$};

					\end{tikzpicture}}  \end{center}
		\end{minipage}
		\hfill
		\begin{minipage}[h]{0.49\linewidth}
			\begin{center}
				\scalebox{0.35}{
					\begin{tikzpicture}
					\draw[green, very thick, dashed] (0,3) to[out = -110, in = 110] (0, 1);
					\draw[green, very thick, dashed] (0,-1) to[out = -110, in = 110] (0, -3);
					\draw[green, very thick, dashed] (1,0) to[out = -20, in = 200] (3, 0);
					\draw[green, very thick, dashed] (5,0) to[out = -20, in = 200] (7, 0);
					\draw[green, very thick, dashed] (-3,0) to[out = -20, in = 200] (-1, 0);
					\draw[green, very thick, dashed] (-7,0) to[out = -20, in = 200] (-5, 0);
					
					\draw[blue, very thick, dashed] (0.5,-3) to[out = 110, in = -110] (0.5, -1+0.1);


					\draw[very thick] (-4, 3) to (4, 3);
					\draw[very thick] (-4, -3) to (4, -3);
					\draw[very thick] (-4, 3) arc (90:270:3);
					\draw[very thick] (4, 3) arc (90:-90:3);
					\draw[very thick] (0, 0) circle (1);
					\draw[very thick] (-4, 0) circle (1);
					\draw[very thick] (4, 0) circle (1);
					
					\draw[green, very thick] (0,3) to[out = -70, in = 70] (0, 1);
					\draw[green, very thick] (0,-1) to[out = -70, in = 70] (0, -3);
					\draw[green, very thick] (1,0) to[out = 20, in = 160] (3, 0);
					\draw[green, very thick] (5,0) to[out = 20, in = 160] (7, 0);
					\draw[green, very thick] (-3,0) to[out = 20, in = 160] (-1, 0);
					\draw[green, very thick] (-7,0) to[out = 20, in = 160] (-5, 0);
					
					\draw[blue, very thick] (0.5,-3) to[out = 70, in = -70] (0.5, -1+0.1);

					
					\node[blue, scale = 2] at (1.2, -1.9) {$\alpha$};

					\end{tikzpicture}} \end{center}
		\end{minipage}
		\caption{The multicurves $M$ from the set $\M_{\alpha, 3}(y) / \Stab_{\Mod(\S)}(\alpha)$ with $\cd(\I^{(\alpha)}_M) = 1$.}
		\label{A3}
	\end{figure}
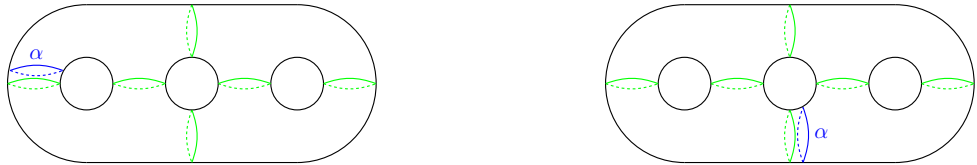
\end{proof}

\begin{lemma} \label{di22}
	The differential $\widetilde{d}_{2, 2}^1: \widetilde{E}_{2,2}^1 \to \widetilde{E}_{1,2}^1$ is injective, hence $\widetilde{E}_{2, 2}^2 = \widetilde{E}_{2,2}^\infty = 0$.
\end{lemma}
\begin{proof}
	Let us consider the multicurves $M$ from the set $\M_{\alpha, 2}(y) / \Stab_{\Mod(\S)}(\alpha)$, such that $\cd(\Stab_{\I_{\alpha}}(M)) = 2$.
	Using formula (\ref{cd}) one can easily check that all the representatives are shown in Fig. \ref{A2}. The multicurve $M$ is shown in green, the curve $\alpha$ is shown in blue.

	\begin{figure}[h]
		\begin{minipage}[h]{0.31\linewidth}
			\begin{center}
				\scalebox{0.3}{
					\begin{tikzpicture}
					\draw[green, very thick, dashed] (0,3) to[out = -110, in = 110] (0, 1);
					\draw[green, very thick, dashed] (0,-1) to[out = -110, in = 110] (0, -3);
					\draw[green, very thick, dashed] (-3,0) to[out = -20, in = 200] (-1, 0);
					\draw[green, very thick, dashed] (-7,0) to[out = -20, in = 200] (-5, 0);
					
					\draw[blue, very thick, dashed] (0.5,3) to[out = -110, in = 110] (0.5, 1-0.1);


					\draw[very thick] (-4, 3) to (4, 3);
					\draw[very thick] (-4, -3) to (4, -3);
					\draw[very thick] (-4, 3) arc (90:270:3);
					\draw[very thick] (4, 3) arc (90:-90:3);
					\draw[very thick] (0, 0) circle (1);
					\draw[very thick] (-4, 0) circle (1);
					\draw[very thick] (4, 0) circle (1);
					
					\draw[green, very thick] (0,3) to[out = -70, in = 70] (0, 1);
					\draw[green, very thick] (0,-1) to[out = -70, in = 70] (0, -3);
					\draw[green, very thick] (-3,0) to[out = 20, in = 160] (-1, 0);
					\draw[green, very thick] (-7,0) to[out = 20, in = 160] (-5, 0);
					
					\draw[blue, very thick] (0.5,3) to[out = -70, in = 70] (0.5, 1-0.1);


					\end{tikzpicture}} \\ $\;$ \end{center}
		\end{minipage}
		\hfill
		\begin{minipage}[h]{0.31\linewidth}
			\begin{center}
				\scalebox{0.3}{
					\begin{tikzpicture}
					\draw[green, very thick, dashed] (0,3) to[out = -110, in = 110] (0, 1);
					\draw[green, very thick, dashed] (0,-1) to[out = -110, in = 110] (0, -3);
					\draw[green, very thick, dashed] (-3,0) to[out = -20, in = 200] (-1, 0);
					\draw[green, very thick, dashed] (-7,0) to[out = -20, in = 200] (-5, 0);
					
					\draw[blue, very thick, dashed] (-7 + 0.1,0.5) to[out = -20, in = 200] (-5+0.1, 0.5);


					\draw[very thick] (-4, 3) to (4, 3);
					\draw[very thick] (-4, -3) to (4, -3);
					\draw[very thick] (-4, 3) arc (90:270:3);
					\draw[very thick] (4, 3) arc (90:-90:3);
					\draw[very thick] (0, 0) circle (1);
					\draw[very thick] (-4, 0) circle (1);
					\draw[very thick] (4, 0) circle (1);
					
					\draw[green, very thick] (0,3) to[out = -70, in = 70] (0, 1);
					\draw[green, very thick] (0,-1) to[out = -70, in = 70] (0, -3);
					\draw[green, very thick] (-3,0) to[out = 20, in = 160] (-1, 0);
					\draw[green, very thick] (-7,0) to[out = 20, in = 160] (-5, 0);
					
					\draw[blue, very thick] (-7+0.1, 0.5) to[out = 20, in = 160] (-5+0.1, 0.5);

					
					\end{tikzpicture}} \\ $\;$ \end{center}
		\end{minipage}
		\hfill
		\begin{minipage}[h]{0.31\linewidth}
			\begin{center}
				\scalebox{0.3}{
					\begin{tikzpicture}
					\draw[green, very thick, dashed] (0,3) to[out = -110, in = 110] (0, 1);
					\draw[green, very thick, dashed] (0,-1) to[out = -110, in = 110] (0, -3);
					\draw[blue, very thick, dashed] (5,0) to[out = -20, in = 200] (7, 0);
					\draw[green, very thick, dashed] (-3,0) to[out = -20, in = 200] (-1, 0);
					\draw[green, very thick, dashed] (-7,0) to[out = -20, in = 200] (-5, 0);
					


					\draw[very thick] (-4, 3) to (4, 3);
					\draw[very thick] (-4, -3) to (4, -3);
					\draw[very thick] (-4, 3) arc (90:270:3);
					\draw[very thick] (4, 3) arc (90:-90:3);
					\draw[very thick] (0, 0) circle (1);
					\draw[very thick] (-4, 0) circle (1);
					\draw[very thick] (4, 0) circle (1);
					
					\draw[green, very thick] (0,3) to[out = -70, in = 70] (0, 1);
					\draw[green, very thick] (0,-1) to[out = -70, in = 70] (0, -3);
					\draw[blue, very thick] (5,0) to[out = 20, in = 160] (7, 0);
					\draw[green, very thick] (-3,0) to[out = 20, in = 160] (-1, 0);
					\draw[green, very thick] (-7,0) to[out = 20, in = 160] (-5, 0);
					

					
					\end{tikzpicture}} \\ $\;$ \end{center}
		\end{minipage}
		\hfill
		\begin{minipage}[h]{0.31\linewidth}
			\begin{center}
				\scalebox{0.3}{
					\begin{tikzpicture}
					\draw[green, very thick, dashed] (0,3) to[out = -110, in = 110] (0, 1);
					\draw[green, very thick, dashed] (0,-1) to[out = -110, in = 110] (0, -3);
					\draw[green, very thick, dashed] (5,0) to[out = -20, in = 200] (7, 0);
					\draw[green, very thick, dashed] (-3,0) to[out = -20, in = 200] (-1, 0);
					\draw[green, very thick, dashed] (-7,0) to[out = -20, in = 200] (-5, 0);
					
					\draw[blue, very thick, dashed] (0.5,3) to[out = -110, in = 110] (0.5, 1-0.1);


					\draw[very thick] (-4, 3) to (4, 3);
					\draw[very thick] (-4, -3) to (4, -3);
					\draw[very thick] (-4, 3) arc (90:270:3);
					\draw[very thick] (4, 3) arc (90:-90:3);
					\draw[very thick] (0, 0) circle (1);
					\draw[very thick] (-4, 0) circle (1);
					\draw[very thick] (4, 0) circle (1);
					
					\draw[green, very thick] (0,3) to[out = -70, in = 70] (0, 1);
					\draw[green, very thick] (0,-1) to[out = -70, in = 70] (0, -3);
					\draw[green, very thick] (5,0) to[out = 20, in = 160] (7, 0);
					\draw[green, very thick] (-3,0) to[out = 20, in = 160] (-1, 0);
					\draw[green, very thick] (-7,0) to[out = 20, in = 160] (-5, 0);
					
					\draw[blue, very thick] (0.5,3) to[out = -70, in = 70] (0.5, 1-0.1);


					\end{tikzpicture}} \\ $\;$ \end{center}
		\end{minipage}
		\hfill
		\begin{minipage}[h]{0.31\linewidth}
			\begin{center}
				\scalebox{0.3}{
					\begin{tikzpicture}
					\draw[green, very thick, dashed] (0,3) to[out = -110, in = 110] (0, 1);
					\draw[green, very thick, dashed] (0,-1) to[out = -110, in = 110] (0, -3);
					\draw[green, very thick, dashed] (5,0) to[out = -20, in = 200] (7, 0);
					\draw[green, very thick, dashed] (-3,0) to[out = -20, in = 200] (-1, 0);
					\draw[green, very thick, dashed] (-7,0) to[out = -20, in = 200] (-5, 0);
					
					\draw[blue, very thick, dashed] (-7 + 0.1,0.5) to[out = -20, in = 200] (-5+0.1, 0.5);


					\draw[very thick] (-4, 3) to (4, 3);
					\draw[very thick] (-4, -3) to (4, -3);
					\draw[very thick] (-4, 3) arc (90:270:3);
					\draw[very thick] (4, 3) arc (90:-90:3);
					\draw[very thick] (0, 0) circle (1);
					\draw[very thick] (-4, 0) circle (1);
					\draw[very thick] (4, 0) circle (1);
					
					\draw[green, very thick] (0,3) to[out = -70, in = 70] (0, 1);
					\draw[green, very thick] (0,-1) to[out = -70, in = 70] (0, -3);
					\draw[green, very thick] (5,0) to[out = 20, in = 160] (7, 0);
					\draw[green, very thick] (-3,0) to[out = 20, in = 160] (-1, 0);
					\draw[green, very thick] (-7,0) to[out = 20, in = 160] (-5, 0);
					
					\draw[blue, very thick] (-7+0.1, 0.5) to[out = 20, in = 160] (-5+0.1, 0.5);

					
					\end{tikzpicture}} \\ $\;$ \end{center}
		\end{minipage}
		\hfill
		\begin{minipage}[h]{0.31\linewidth}
	\begin{center}
		\scalebox{0.3}{
			\begin{tikzpicture}
			\draw[green, very thick, dashed] (0,3) to[out = -110, in = 110] (0, 1);
			\draw[green, very thick, dashed] (0,-1) to[out = -110, in = 110] (0, -3);
			\draw[green, very thick, dashed] (5,0) to[out = -20, in = 200] (7, 0);
			\draw[green, very thick, dashed] (-3,0) to[out = -20, in = 200] (-1, 0);
			\draw[green, very thick, dashed] (-7,0) to[out = -20, in = 200] (-5, 0);
			
			\draw[blue, very thick, dashed] (7 - 0.1,0.5) to[out = 200, in = -20] (5-0.1, 0.5);


			\draw[very thick] (-4, 3) to (4, 3);
			\draw[very thick] (-4, -3) to (4, -3);
			\draw[very thick] (-4, 3) arc (90:270:3);
			\draw[very thick] (4, 3) arc (90:-90:3);
			\draw[very thick] (0, 0) circle (1);
			\draw[very thick] (-4, 0) circle (1);
			\draw[very thick] (4, 0) circle (1);
			
			\draw[green, very thick] (0,3) to[out = -70, in = 70] (0, 1);
			\draw[green, very thick] (0,-1) to[out = -70, in = 70] (0, -3);
			\draw[green, very thick] (5,0) to[out = 20, in = 160] (7, 0);
			\draw[green, very thick] (-3,0) to[out = 20, in = 160] (-1, 0);
			\draw[green, very thick] (-7,0) to[out = 20, in = 160] (-5, 0);
			
			\draw[blue, very thick] (7-0.1, 0.5) to[out = 160, in = 20] (5-0.1, 0.5);

			
			\end{tikzpicture}} \\ $\;$ \end{center}
\end{minipage}
		
		\caption{The multicurves $M$ from the set $\M_{\alpha, 2}(y) / \Stab_{\Mod(\S)}(\alpha)$ with $\cd(\Stab_{\I_{\alpha}}(M)) = 2$.}
		\label{A2}
	\end{figure}
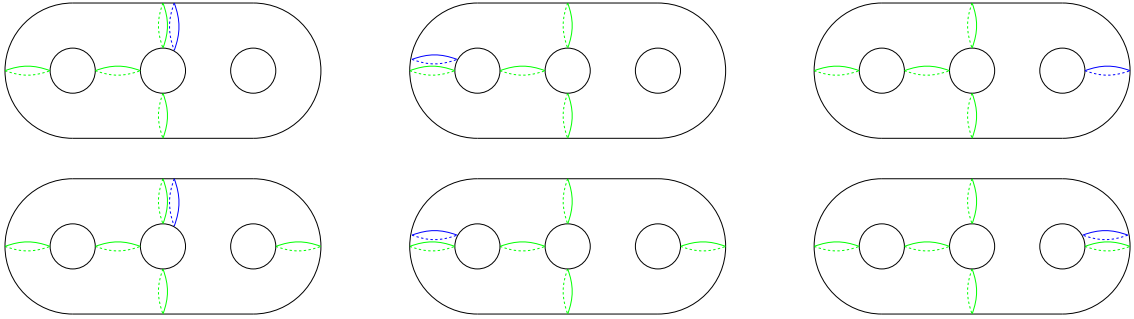

Let a multicurve $M$ has one of the types shown in Fig. \ref{A2}. We have the inclusion 
$$\H_{2}(\Stab_{\I_{\alpha}}(M), \Z) \hookrightarrow \H_{2}(\I_M, \Z)$$
(these homology groups are already computed in Section \ref{S5}). Moreover, if $M, M' \in \M_{\alpha, 2}(y)$ are $\I$-equivalent, it follows that $M$ and $M'$ are $\I_{\alpha}$-equivalent. Indeed, for the third type in Fig. \ref{A2} it is straightforward.
For the other types this follows from the general fact that $\I$ acts on $\B(x)$ without rotations. Therefore, the injectivity of the differential $\widetilde{d}_{2, 2}^1$ immediately follows from the injectivity of the differential $d_{2, 2}^1$ (see Proposition \ref{diff22}).
\end{proof}

Lemmas \ref{E04}, \ref{di31} and \ref{di22} imply that 
$$\widetilde{E}_{0, 4}^\infty = \widetilde{E}_{2, 2}^\infty = \widetilde{E}_{3, 1}^\infty = 0.$$
Therefore, all nonzero terms of the pages $\widetilde{E}^1$ and $\widetilde{E}^2$ are shown in Fig. \ref{tildeE}.

\begin{corollary} \label{corEinf2}
	We have an isomorphism
	$$\H_4(\I_{\alpha}, \Z) \cong \widetilde{E}^\infty_{1, 3} \cong \widetilde{E}^2_{1, 3}.$$
\end{corollary}

\subsection{The term $\widetilde{E}^2_{1, 3}$}

\begin{figure}[h]
	\begin{minipage}[h]{0.31\linewidth}
		\begin{center}
			\scalebox{0.3}{
				\begin{tikzpicture}
				\draw[green, very thick, dashed] (0,3) to[out = -110, in = 110] (0, 1);
				\draw[green, very thick, dashed] (0,-1) to[out = -110, in = 110] (0, -3);
				\draw[blue, very thick, dashed] (-7,0) to[out = -20, in = 200] (-5, 0);
				


				\draw[very thick] (-4, 3) to (4, 3);
				\draw[very thick] (-4, -3) to (4, -3);
				\draw[very thick] (-4, 3) arc (90:270:3);
				\draw[very thick] (4, 3) arc (90:-90:3);
				\draw[very thick] (0, 0) circle (1);
				\draw[very thick] (-4, 0) circle (1);
				\draw[very thick] (4, 0) circle (1);
				
				\draw[green, very thick] (0,3) to[out = -70, in = 70] (0, 1);
				\draw[green, very thick] (0,-1) to[out = -70, in = 70] (0, -3);
				\draw[blue, very thick] (-7,0) to[out = 20, in = 160] (-5, 0);
				


				\end{tikzpicture}} \\ type $(1)$ w.r.t. $y$ \\ $\;$ \end{center}
	\end{minipage}
	\hfill
	\begin{minipage}[h]{0.31\linewidth}
		\begin{center}
			\scalebox{0.3}{
				\begin{tikzpicture}
				\draw[green, very thick, dashed] (0,3) to[out = -110, in = 110] (0, 1);
				\draw[green, very thick, dashed] (0,-1) to[out = -110, in = 110] (0, -3);
				\draw[green, very thick, dashed] (-7,0) to[out = -20, in = 200] (-5, 0);
				
				\draw[blue, very thick, dashed] (-7 + 0.1,0.5) to[out = -20, in = 200] (-5+0.1, 0.5);


				\draw[very thick] (-4, 3) to (4, 3);
				\draw[very thick] (-4, -3) to (4, -3);
				\draw[very thick] (-4, 3) arc (90:270:3);
				\draw[very thick] (4, 3) arc (90:-90:3);
				\draw[very thick] (0, 0) circle (1);
				\draw[very thick] (-4, 0) circle (1);
				\draw[very thick] (4, 0) circle (1);
				
				\draw[green, very thick] (0,3) to[out = -70, in = 70] (0, 1);
				\draw[green, very thick] (0,-1) to[out = -70, in = 70] (0, -3);
				\draw[green, very thick] (-7,0) to[out = 20, in = 160] (-5, 0);
				
				\draw[blue, very thick] (-7+0.1, 0.5) to[out = 20, in = 160] (-5+0.1, 0.5);

				
				\end{tikzpicture}} \\ type $(2)$ w.r.t. $y$ \\ $\;$ \end{center}
	\end{minipage}
	\hfill
	\begin{minipage}[h]{0.31\linewidth}
		\begin{center}
			\scalebox{0.3}{
				\begin{tikzpicture}
				\draw[green, very thick, dashed] (0,3) to[out = -110, in = 110] (0, 1);
				\draw[green, very thick, dashed] (0,-1) to[out = -110, in = 110] (0, -3);
				\draw[blue, very thick, dashed] (5,0) to[out = -20, in = 200] (7, 0);
				\draw[green, very thick, dashed] (-7,0) to[out = -20, in = 200] (-5, 0);
				


				\draw[very thick] (-4, 3) to (4, 3);
				\draw[very thick] (-4, -3) to (4, -3);
				\draw[very thick] (-4, 3) arc (90:270:3);
				\draw[very thick] (4, 3) arc (90:-90:3);
				\draw[very thick] (0, 0) circle (1);
				\draw[very thick] (-4, 0) circle (1);
				\draw[very thick] (4, 0) circle (1);
				
				\draw[green, very thick] (0,3) to[out = -70, in = 70] (0, 1);
				\draw[green, very thick] (0,-1) to[out = -70, in = 70] (0, -3);
				\draw[blue, very thick] (5,0) to[out = 20, in = 160] (7, 0);
				\draw[green, very thick] (-7,0) to[out = 20, in = 160] (-5, 0);
				

				
				\end{tikzpicture}} \\ type $(3)$ w.r.t. $y$ \\ $\;$ \end{center}
	\end{minipage}
	\hfill
	\begin{minipage}[h]{0.31\linewidth}
		\begin{center}
			\scalebox{0.3}{
				\begin{tikzpicture}
				\draw[green, very thick, dashed] (0,3) to[out = -110, in = 110] (0, 1);
				\draw[green, very thick, dashed] (0,-1) to[out = -110, in = 110] (0, -3);
				\draw[green, very thick, dashed] (-7,0) to[out = -20, in = 200] (-5, 0);
				
				\draw[blue, very thick, dashed] (0.5,3) to[out = -110, in = 110] (0.5, 1-0.1);


				\draw[very thick] (-4, 3) to (4, 3);
				\draw[very thick] (-4, -3) to (4, -3);
				\draw[very thick] (-4, 3) arc (90:270:3);
				\draw[very thick] (4, 3) arc (90:-90:3);
				\draw[very thick] (0, 0) circle (1);
				\draw[very thick] (-4, 0) circle (1);
				\draw[very thick] (4, 0) circle (1);
				
				\draw[green, very thick] (0,3) to[out = -70, in = 70] (0, 1);
				\draw[green, very thick] (0,-1) to[out = -70, in = 70] (0, -3);
				\draw[green, very thick] (-7,0) to[out = 20, in = 160] (-5, 0);
				
				\draw[blue, very thick] (0.5,3) to[out = -70, in = 70] (0.5, 1-0.1);


				\end{tikzpicture}} \\ type $(2)$ w.r.t. $y$ \\ $\;$ \end{center}
	\end{minipage}
	\hfill
	\begin{minipage}[h]{0.31\linewidth}
	\begin{center}
		\scalebox{0.3}{
			\begin{tikzpicture}
			\draw[green, very thick, dashed] (0,3) to[out = -110, in = 110] (0, 1);
			\draw[green, very thick, dashed] (0,-1) to[out = -110, in = 110] (0, -3);
			\draw[green, very thick, dashed] (5,0) to[out = -20, in = 200] (7, 0);
			\draw[green, very thick, dashed] (-7,0) to[out = -20, in = 200] (-5, 0);
			
			\draw[blue, very thick, dashed] (0.5,3) to[out = -110, in = 110] (0.5, 1-0.1);


			\draw[very thick] (-4, 3) to (4, 3);
			\draw[very thick] (-4, -3) to (4, -3);
			\draw[very thick] (-4, 3) arc (90:270:3);
			\draw[very thick] (4, 3) arc (90:-90:3);
			\draw[very thick] (0, 0) circle (1);
			\draw[very thick] (-4, 0) circle (1);
			\draw[very thick] (4, 0) circle (1);
			
			\draw[green, very thick] (0,3) to[out = -70, in = 70] (0, 1);
			\draw[green, very thick] (0,-1) to[out = -70, in = 70] (0, -3);
			\draw[green, very thick] (5,0) to[out = 20, in = 160] (7, 0);
			\draw[green, very thick] (-7,0) to[out = 20, in = 160] (-5, 0);
			
			\draw[blue, very thick] (0.5,3) to[out = -70, in = 70] (0.5, 1-0.1);


			\end{tikzpicture}} \\ type $(4)$ w.r.t. $y$ \\ $\;$ \end{center}
    \end{minipage}
	\hfill
	\begin{minipage}[h]{0.31\linewidth}
		\begin{center}
			\scalebox{0.3}{
				\begin{tikzpicture}
				\draw[green, very thick, dashed] (0,3) to[out = -110, in = 110] (0, 1);
				\draw[green, very thick, dashed] (0,-1) to[out = -110, in = 110] (0, -3);
				\draw[green, very thick, dashed] (5,0) to[out = -20, in = 200] (7, 0);
				\draw[green, very thick, dashed] (-7,0) to[out = -20, in = 200] (-5, 0);
				
				\draw[blue, very thick, dashed] (-7 + 0.1,0.5) to[out = -20, in = 200] (-5+0.1, 0.5);


				\draw[very thick] (-4, 3) to (4, 3);
				\draw[very thick] (-4, -3) to (4, -3);
				\draw[very thick] (-4, 3) arc (90:270:3);
				\draw[very thick] (4, 3) arc (90:-90:3);
				\draw[very thick] (0, 0) circle (1);
				\draw[very thick] (-4, 0) circle (1);
				\draw[very thick] (4, 0) circle (1);
				
				\draw[green, very thick] (0,3) to[out = -70, in = 70] (0, 1);
				\draw[green, very thick] (0,-1) to[out = -70, in = 70] (0, -3);
				\draw[green, very thick] (5,0) to[out = 20, in = 160] (7, 0);
				\draw[green, very thick] (-7,0) to[out = 20, in = 160] (-5, 0);
				
				\draw[blue, very thick] (-7+0.1, 0.5) to[out = 20, in = 160] (-5+0.1, 0.5);

				
				\end{tikzpicture}} \\ type $(4)$ w.r.t. $y$ \\ $\;$ \end{center}
	\end{minipage}
	\caption{The multicurves $M$ from the set $\M_{\alpha, 1}(y) / \Stab_{\Mod(\S)}(\alpha)$ with $\cd(\Stab_{\I_{\alpha}}(M)) = 3$.}
	\label{A1}
\end{figure}
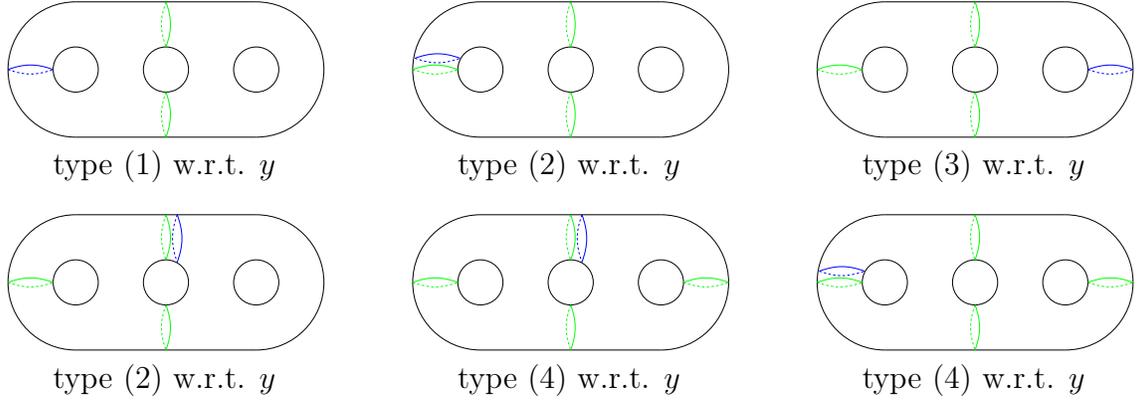

Let us compute the group $\widetilde{E}^2_{1, 3}$ explicitly. We forget about the orientation and consider the multicurves $M$ from the set $\M_{\alpha, 3}(y) / \Stab_{\Mod(\S)}(\alpha)$, such that $\cd(\Stab_{\I_{\alpha}}(M)) = 3$.
Using formula (\ref{cd}) one can easily check that all the representatives are shown in Fig. \ref{A1}. The multicurve $M$ is shown in green, the curve $\alpha$ is shown in blue.

Now let us describe a basis of the free abelian group $\widetilde{E}_{1, 3}^1$.
Consider an unordered symplectic splitting $\U = \{U_1, U_2, U_3\}$ of $\H$ such that $x \in U_1$ (that is, $\U$ is of type (a) w.r.t $x$). There is the unique decomposition $y = y_1 + y_2 + y_3$, where $y_j \in U_j$. Let us introduce the following notation.

\begin{definition} \label{typey}
	We say that the splitting $\U$
	
	$\bullet$ is of type (1) w.r.t. $y$, if $y_2 \neq 0$ and $y_1 = y_3 = 0$;
	
	$\bullet$ is of type (2) w.r.t. $y$, if $y_1, y_2 \neq 0$ and $y_3 = 0$;
	
	$\bullet$ is of type (3) w.r.t. $y$, if $y_2, y_3 \neq 0$ and $y_1 = 0$;
	
	$\bullet$ is of type (4) w.r.t. $y$, if $y_1, y_2, y_3 \neq 0$.
	
\end{definition}

Renumbering $U_2$ and $U_3$ we may achieve that one of the alternatives given in Definition \ref{typey} is true. Therefore, $U$ is of type (1), (2), (3) or (4) w.r.t. $y$.

Suppose that $j=1, 2, 3$. If $y_i \neq 0$, then put $y_j = k_j b_j$, where $b_j$ is a primitive homology class and $k_j \in \mathbb{N}$. Consider a bounding pair $\beta_j, \beta'_j$ disjoint from $\alpha$, with $[\beta_j] = [\beta'_j] = b_j$ and the splitting $\U$ is admissible for $\beta_j \cup \beta'_j$ (if $b_j = 0$ we define $\beta_j = \beta'_j = \varnothing$). 
We also always take $\beta_1 = \alpha$ (if $\beta_1$ exists).
All such bounding pairs are $\I_{\alpha}$-equivalent.

If $\U$ is of type (1) w.r.t $y$, we consider the element 
\begin{equation} \label{type12}
P_{\beta_2 \cup \beta'_2} \otimes \A_{U_1, U_3} \in \widetilde{E}_{1, 3}^1.
\end{equation}
If $\U$ is of type (2) w.r.t $y$, we consider the elements
\begin{equation} \label{type22}
P_{\beta_1 \cup \beta'_1 \cup \beta_2} \otimes \A_{U_2, U_3} \in \widetilde{E}_{1, 3}^1.
\end{equation}
\begin{equation} \label{type22'}
P_{\beta_1 \cup \beta_2 \cup \beta'_2} \otimes \A_{U_1, U_3} \in \widetilde{E}_{1, 3}^1,
\end{equation}
If $\U$ is of type (3) w.r.t $y$, we consider the elements
\begin{equation} \label{type32}
P_{\beta_2 \cup \beta'_2 \cup \beta_3} \otimes \A_{U_3, U_1}, \qquad P_{\beta_2 \cup \beta_3 \cup \beta'_3} \otimes \A_{U_1, U_2} \in \widetilde{E}_{1, 3}^1.
\end{equation}
If $\U$ is of type (4) w.r.t $y$, we consider the elements
\begin{equation} \label{type42}
P_{\beta_1 \cup \beta'_1 \cup \beta_2 \cup \beta_3} \otimes \A_{U_2, U_3} \in \widetilde{E}_{1, 3}^1,
\end{equation}
\begin{equation} \label{type52}
P_{\beta_1 \cup \beta_2 \cup \beta'_2 \cup \beta_3} \otimes \A_{U_3, U_1}, \qquad P_{\beta_1 \cup \beta_2 \cup \beta_3 \cup \beta'_3} \otimes \A_{U_1, U_2} \in \widetilde{E}_{1, 3}^1.
\end{equation}
Note that the cells mentioned above are uniquely determined (up to $\I_{\alpha}$-equivalence) by $\U$.

Corollary \ref{E13gen} implies that the elements (\ref{type12}), (\ref{type22}), (\ref{type22}), (\ref{type32}) and (\ref{type42}), where $\U$ runs over the set of all splittings of $\H$, form a basis of the free abelian group $\widetilde{E}_{1, 3}^1$.

Lemmas \ref{dM2} and \ref{zero} (we take $\G = \I_{\alpha}$ and $\BB(y) = \B_\alpha(y)$) imply that the images of the elements (\ref{type12}) and (\ref{type22'}) differential $\widetilde{d}_{1, 3}^1$ are zero.
By Lemma (\ref{ind}), the image of (\ref{type22}) is 
$$P_{\beta_1 \cup \beta_2} \otimes \A_{U_2, U_3} - P_{\beta'_1 \cup \beta_2} \otimes \A_{U_2, U_3} \in \widetilde{E}_{0, 3}^1.$$
The image of the elements (\ref{type32}) is
$$P_{\beta_2 \cup \beta_3} \otimes \A_{U_1, U_3, U_2} \in \widetilde{E}_{0, 3}^1;$$
note that by Lemma \ref{ind} this element is nonzero because $\Stab_{\I_{\alpha}}(\beta_2 \cup \beta_3) = \I_{\beta_1\cup \beta_2 \cup \beta_3}$.
The image of (\ref{type42}) is
$$P_{\beta_1 \cup \beta_2 \cup \beta_3} \otimes \A_{U_2, U_3} - P_{\beta'_1 \cup \beta_2 \cup \beta_3} \otimes \A_{U_2, U_3} \in \widetilde{E}_{0, 3}^1.$$
The image of the elements (\ref{type52}) is
$$P_{\beta_1 \cup \beta_2 \cup \beta_3} \otimes \A_{U_1, U_3, U_2} \in \widetilde{E}_{0, 3}^1.$$

By Lemma \ref{ind} these elements are linearly independent for different $\U$.
Therefore we have the following result.

\begin{prop} \label{kernel2}
	The free abelian group $\widetilde{E}_{1, 3}^2 = \ker \widetilde{d}_{1, 3}^1$ has a basis consisting of the following elements:
	$$P_{\beta_2 \cup \beta'_2} \otimes \A_{U_1, U_3},$$
	where $\U$ runs over the set of all splittings of $\H$ of type (1) w.r.t $y$,
	$$P_{\beta_1 \cup \beta_2 \cup \beta'_2} \otimes \A_{U_1, U_3},$$
	where $\U$ runs over the set of all splittings of $\H$ of type (2) w.r.t $y$,
	$$P_{\beta_2 \cup \beta'_2 \cup \beta_3} \otimes \A_{U_3, U_1} - P_{\beta_2 \cup \beta_3 \cup \beta'_3} \otimes \A_{U_1, U_2},$$
	where $\U$ runs over the set of all splittings of $\H$ of type (3) w.r.t $y$,
	$$P_{\beta_1 \cup \beta_2 \cup \beta'_2 \cup \beta_3} \otimes \A_{U_3, U_1} - P_{\beta_1 \cup \beta_2 \cup \beta_3 \cup \beta'_3} \otimes \A_{U_1, U_2},$$
	where $\U$ runs over the set of all splittings of $\H$ of type (4) w.r.t $y$.
\end{prop}

\subsection{The term $E^1_{0, 4}$}

\begin{proof}[Proof of Proposition \ref{surj0}.] We need to prove that the map $J_{0, 4}^1: \E_{0, 4}^1 \to E_{0, 4}^1$ is surjective. We have an isomorphism
$$E_{0, 4}^1 \cong \H_4(\I_{\alpha}, \Z).$$
We need to prove that the map
$$\bigoplus_{\gamma} \H_{4}(\Stab_{\I_{\gamma}}(\alpha), \Z) \to \H_4(\I_{\alpha}, \Z),$$
where the sum is over all separating curves $\gamma$ disjoint from $\alpha$, is surjective. For each such $\gamma$ let us denote by $\widetilde{E}^{(\gamma) *}_{*, *}$ the spectral sequence (\ref{spec_sec}) for the action of $\I_{\gamma}$ on $\B_{\alpha}(y)$. 

Denote by $\widetilde{j}^{(\gamma) *}_{*, *}: \widetilde{E}^{(\gamma) *}_{*, *} \to \widetilde{E}^*_{*, *}$ the morphism of the spectral sequences induced by the inclusion $\widetilde{\iota}_{\gamma}: \Stab_{\I_{\gamma}}(\alpha) \hookrightarrow \I_{\alpha}$. 
Consider the morphism
\begin{equation} \label{morspec2}
\bigoplus_\gamma \widetilde{j}^{(\gamma) *}_{*, *}: \bigoplus_{\gamma} \widetilde{E}^{(\gamma) *}_{*, *} \to \widetilde{E}^*_{*, *}
\end{equation}
and denote $\widetilde{J}_{*, *}^* = \bigoplus_\gamma \widetilde{j}^{(\gamma) *}_{*, *}$, where the sums are over all separating curves $\gamma$ on $\S$ disjoint from $\alpha$.
Corollary \ref{corEinf2} implies that it suffices to prove that the map $\widetilde{J}_{1, 3}^2$ is surjective. 

Proposition \ref{kernel2} we constructed a basis of the free abelian group $\widetilde{E}_{1, 3}^2$. Let us show that each of this elements belongs to the image of $\widetilde{J}_{1, 3}^2$.
Let $\U = (U_1, U_2, U_3)$ be a symplectic splitting of $\H$. Let $\theta_1, \theta_2, \theta_3$ be separating curves disjoint from $\alpha$ such that $\H_{\theta_j} = U_j$ ($j = 1, 2, 3$). Let $\beta_j, \beta'_j$ be as before such that all these curves are pairwise disjoint except the three pairs ($\theta_j$, $\beta'_j$).
There are four possible cases.

(1) \textit{$\U$ is of type (1) w.r.t $y$.} We need to check that $P_{\beta_2 \cup \beta'_2} \otimes \A_{U_1, U_3}$ belongs to the image of $\widetilde{J}_{1, 3}^2$. Consider the group $\Stab_{\I_{\theta_1}}(\alpha)$ and the homology class 
$$\A(T_{\theta_1}, T_{\beta_2} T_{\beta'_2}^{-1}, T_{\theta_3}) \in \H_3(\Stab_{\I_{\theta_1}}(\alpha \cup \beta_2 \cup \beta'_2), \Z).$$
Let us consider the element 
$$P_{\beta_2 \cup \beta'_2} \otimes \A(T_{\theta_1}, T_{\beta_2} T_{\beta'_2}^{-1}, T_{\theta_3}) \in \widetilde{E}_{3, 1}^{(\theta_2) 1}.$$
Lemmas \ref{dM2} and \ref{zero} (we take $\G = \I_{\theta_2}$ and $\BB(y) = \B_\alpha(y)$) yield that 
$$ \widetilde{d}_{1, 3}^{(\theta_2) 1} \left( P_{\beta_2 \cup \beta'_2} \otimes \A(T_{\theta_1}, T_{\beta_2} T_{\beta'_2}^{-1}, T_{\theta_3}) \right) = 0.$$
Hence
$$P_{\beta_2 \cup \beta'_2} \otimes \A(T_{\theta_1}, T_{\beta_2} T_{\beta'_2}^{-1}, T_{\theta_3}) \in \widetilde{E}_{3, 1}^{(\theta_2) 2}.$$
Obviously we have
$$\widetilde{J}_{1, 3}^2 \left( P_{\beta_2 \cup \beta'_2} \otimes \A(T_{\theta_1}, T_{\beta_2} T_{\beta'_2}^{-1}, T_{\theta_3}) \right) = P_{\beta_2 \cup \beta'_2} \otimes \A_{U_1, U_3}.$$

(2) \textit{$\U$ is of type (2) w.r.t $y$.} We need to check that $P_{\beta_1 \cup \beta_2 \cup \beta'_2} \otimes \A_{U_1, U_3}$ belongs to the image of $\widetilde{J}_{1, 3}^2$. Consider the group $\Stab_{\I_{\theta_1}}(\alpha)$ and the homology class 
$$\A(T_{\theta_1}, T_{\beta_2} T_{\beta'_2}^{-1}, T_{\theta_3}) \in \H_3(\Stab_{\I_{\theta_1}}(\alpha \cup \beta_1 \cup \beta_2 \cup \beta'_2), \Z).$$
Let us consider the element 
$$P_{\beta_1 \cup \beta_2 \cup \beta'_2} \otimes \A(T_{\theta_1}, T_{\beta_2} T_{\beta'_2}^{-1}, T_{\theta_3}) \in \widetilde{E}_{3, 1}^{(\theta_1) 1}.$$
Lemmas \ref{dM2} and \ref{zero} (we take $\G = \I_{\theta_2}$ and $\BB(y) = \B_\alpha(y)$) yield that 
$$ \widetilde{d}_{1, 3}^{(\theta_2) 1} \left( P_{\beta_1 \cup \beta_2 \cup \beta'_2} \otimes \A(T_{\theta_1}, T_{\beta_2} T_{\beta'_2}^{-1}, T_{\theta_3}) \right) = 0.$$
Hence
$$P_{\beta_1 \cup \beta_2 \cup \beta'_2} \otimes \A(T_{\theta_1}, T_{\beta_2} T_{\beta'_2}^{-1}, T_{\theta_3}) \in \widetilde{E}_{3, 1}^{(\theta_1) 2}.$$
Obviously we have
$$\widetilde{J}_{1, 3}^2 \left( P_{\beta_1 \cup \beta_2 \cup \beta'_2} \otimes \A(T_{\theta_1}, T_{\beta_2} T_{\beta'_2}^{-1}, T_{\theta_3}) \right) = P_{\beta_1 \cup \beta_2 \cup \beta'_2} \otimes \A_{U_1, U_3}.$$

(3) \textit{$\U$ is of type (3) w.r.t $y$.} This case is similar to the case (4) with removed $\beta_1$ everywhere.

(4) \textit{$\U$ is of type (4) w.r.t $y$.} We need to check that 
$$P_{\beta_1 \cup \beta_2 \cup \beta'_2 \cup \beta_3} \otimes \A_{U_3, U_1} - P_{\beta_1 \cup \beta_2 \cup \beta_3 \cup \beta'_3} \otimes \A_{U_1, U_2}$$ 
belongs to the image of $\widetilde{J}_{1, 3}^2$. Consider the group $\Stab_{\I_{\theta_1}}(\alpha)$ 
and the homology classes 
$$\A(T_{\theta_3}, T_{\beta_2} T_{\beta'_2}^{-1}, T_{\theta_1}), \in \H_3(\Stab_{\I_{\theta_1}}(\alpha \cup \beta_1 \cup \beta_2 \cup \beta'_2 \cup \beta'_3), \Z)$$
and
$$\A(T_{\theta_1}, T_{\beta_3} T_{\beta'_3}^{-1}, T_{\theta_2}) \in \H_3(\Stab_{\I_{\theta_1}}(\alpha \cup \beta_1 \cup \beta_2 \cup \beta_3 \cup \beta'_3), \Z).$$
Let us consider the elements
$$P_{\beta_1 \cup \beta_2 \cup \beta'_2 \cup \beta_3} \otimes \A(T_{\theta_3}, T_{\beta_2} T_{\beta'_2}^{-1}, T_{\theta_1}) \in \widetilde{E}_{3, 1}^{(\theta_1) 1}$$
and
$$P_{\beta_1 \cup \beta_2 \cup \beta_3 \cup \beta'_3} \otimes \A(T_{\theta_1}, T_{\beta_3} T_{\beta'_3}^{-1}, T_{\theta_2}) \in \widetilde{E}_{3, 1}^{(\theta_1) 1}.$$
The arguments similar to Lemmas \ref{dM2} and \ref{ind} (we take $\G = \I_{\theta_1}$ and $\BB(y) = \B_\alpha(y)$) yield that 
$$ \widetilde{d}_{1, 3}^{(\theta_1) 1} \left( P_{\beta_1 \cup \beta_2 \cup \beta'_2 \cup \beta_3} \otimes \A(T_{\theta_3}, T_{\beta_2} T_{\beta'_2}^{-1}, T_{\theta_1}) \right) = P_{\beta_1 \cup \beta_2 \cup \beta_3} \otimes \A(T_{\theta_3},  T_{\theta_2}, T_{\theta_1})$$
and
$$\widetilde{d}_{1, 3}^{(\theta_1) 1} \left( P_{\beta_1 \cup \beta_2 \cup \beta_3 \cup \beta'_3} \otimes \A(T_{\theta_1}, T_{\beta_3} T_{\beta'_3}^{-1}, T_{\theta_2}) \right) = P_{\beta_1 \cup \beta_2 \cup \beta_3} \otimes \A(T_{\theta_3},  T_{\theta_2}, T_{\theta_1}).$$
Hence
$$\widetilde{d}_{1, 3}^{(\theta_1) 1} \left( P_{\beta_1 \cup \beta_2 \cup \beta'_2 \cup \beta_3} \otimes \A(T_{\theta_3}, T_{\beta_2} T_{\beta'_2}^{-1}, T_{\theta_1}) - P_{\beta_1 \cup \beta_2 \cup \beta_3 \cup \beta'_3} \otimes \A(T_{\theta_1}, T_{\beta_3} T_{\beta'_3}^{-1}, T_{\theta_2}) \right)  = 0.$$
Therefore
$$ \left( P_{\beta_1 \cup \beta_2 \cup \beta'_2 \cup \beta_3} \otimes \A(T_{\theta_3}, T_{\beta_2} T_{\beta'_2}^{-1}, T_{\theta_1}) - P_{\beta_1 \cup \beta_2 \cup \beta_3 \cup \beta'_3} \otimes \A(T_{\theta_1}, T_{\beta_3} T_{\beta'_3}^{-1}, T_{\theta_2}) \right) \in \widetilde{E}_{3, 1}^{(\theta_1) 2}.$$
Obviously we have
$$\widetilde{J}_{1, 3}^2 \left( P_{\beta_1 \cup \beta_2 \cup \beta'_2 \cup \beta_3} \otimes \A(T_{\theta_3}, T_{\beta_2} T_{\beta'_2}^{-1}, T_{\theta_1}) - P_{\beta_1 \cup \beta_2 \cup \beta_3 \cup \beta'_3} \otimes \A(T_{\theta_1}, T_{\beta_3} T_{\beta'_3}^{-1}, T_{\theta_2}) \right) = $$
$$= P_{\beta_1 \cup \beta_2 \cup \beta'_2 \cup \beta_3} \otimes \A_{U_3, U_1} - P_{\beta_1 \cup \beta_2 \cup \beta_3 \cup \beta'_3} \otimes \A_{U_1, U_2}.$$
This concludes the proof.
\end{proof}

\section{Proof of Theorem \ref{th2}} \label{S8}

\subsection{Alternative construction of s-classes}

Let $\U = \{U_1, U_2, U_3\}$ be a symplectic splitting of $\H$ and let $\theta_j$ be a separating curve on $\S$ with $\H_{\theta_j} = U_j$, where $j = 1, 2, 3$.
In Section \ref{S3} we have defined the homology class $s(U_2, U_3) \in \H_4(\I_{\theta_1}, \Z)$. 
Let $0 \neq x \in \H$ be a homology class such that $\U$ is of type (c) w.r.t $x$ and let the curves $\alpha_j, \alpha'_j$ be as in Subsection \ref{6.5}.
Recall that we have the natural inclusions
\begin{equation} \label{inc0}
E^{(\theta_1) 1}_{0, 4} = E^{(\theta_1) \infty}_{0, 4} \hookrightarrow \H_4(\I_{\theta_1}, \Z).
\end{equation}
and
\begin{equation} \label{inc1}
E^{(\theta_1) 2}_{1, 3} = E^{(\theta_1) \infty}_{1, 3} \hookrightarrow \H_4(\I_{\theta_1}, \Z) / E^{(\theta_1) 1}_{0, 4}.
\end{equation}

First we need to prove the following simple result.

\begin{lemma} \label{homology}
	Let $F$ be a free group. Suppose that $K \subseteq \Z \times F$ is a subgroup. Then $\H_2(K, \Z)$ is a free abelian group.
\end{lemma}
\begin{proof}
	Consider the projection $p: \Z \times F \twoheadrightarrow F$. Then $\Im p|_K \subseteq F$ is a free group.
	We have $\ker p|_K = K \cap \Z$ that is either trivial or isomorphic to $\Z$. In the first case $K$ is a free group. In the second case we have that $K \cong (K \cap \Z) \times \Im p|_K \cong \Z \times \Im p|_K$. This immediately implies the result.
\end{proof}

\begin{lemma} \label{alt}

The element
$$\left( P_{\alpha_1 \cup \alpha_2 \cup \alpha'_2 \cup \alpha_3} \otimes \A(T_{\theta_1}, T_{\alpha_2} T_{\alpha'_2}^{-1}, T_{\theta_3}) - P_{\alpha_1 \cup \alpha_2 \cup \alpha_3 \cup \alpha'_3} \otimes \A(T_{\theta_2}, T_{\alpha_3} T_{\alpha'_3}^{-1}, T_{\theta_1}) \right) \in E^{(\theta_1) 2}_{1, 3}$$
maps to the coset containing $\pm s(U_2, U_3)$ under the map (\ref{inc1}).
\end{lemma}

\begin{proof}
	Consider the surface $\S_{2, 1} = \S \setminus \overline{X}_{\theta_1}$. We have the exact sequences
	\begin{equation*}\label{eq5}
	\begin{CD}
	1 @>>> \left\langle T_{\theta_1} \right\rangle @>>> \I_{\theta_1} @>{p}>> \I_{2, 1} @>>> 1.
	\end{CD}
	\end{equation*}
	\begin{equation*}\label{BirmanT22}
	\begin{CD}
	1 @>>> \pi_1(\S_2, \pt) @>>> \I_{2, 1} @>{q}>> \I_2 @>>> 1.
	\end{CD}
	\end{equation*}
	Consider the subgroups $Q = q^{-1}(\left\langle T_{\theta_3} \right\rangle ) \subset \I_{2, 1}$ and $G = p^{-1}(Q) \subset \I_{\theta_1}$. Hochschild-Serre spectral sequence implies that we have 
	$$\H_4(G, \Z) = \left\langle s(U_2, U_3) \right\rangle \cong \Z.$$
	
	Let $(\mathcal{E}_{*, *}^{(\theta_1) *}, \partial_{*, *}^{(\theta_1) *})$ be the spectral sequence (\ref{spec_sec}) for the action of $G$ on $\B(x)$. By Lemma \ref{dM2} (we take $\G = G$ and $\BB(x) = \B(x)$) the element 
	\begin{equation}\label{el}
	\left( P_{\alpha_1 \cup \alpha_2 \cup \alpha'_2 \cup \alpha_3} \otimes \A(T_{\theta_1}, T_{\alpha_2} T_{\alpha'_2}^{-1}, T_{\theta_3}) - P_{\alpha_1 \cup \alpha_2 \cup \alpha_3 \cup \alpha'_3} \otimes \A(T_{\theta_2}, T_{\alpha_3} T_{\alpha'_3}^{-1}, T_{\theta_1}) \right)
	\end{equation}
	has an infinite order and lies in the kernel of $\partial_{1, 3}^{(\theta_1) 1}: \mathcal{E}_{1, 3}^{(\theta_1) 1} \to \mathcal{E}_{0, 3}^{(\theta_1) 1}$. So (\ref{el}) belongs to the group $\mathcal{E}_{1, 3}^{(\theta_1) 2}$.
	Hence $\mathcal{E}^{(\theta_1) 2}_{1, 3} \subseteq \H_4(G, \Z) / \mathcal{E}^{(\theta_1) 1}_{0, 4}$ contains a subgroup $\Z$ generated by the element (\ref{el}). Since $\H_4(G, \Z) \cong \Z$ it follows that the group $\mathcal{E}^{(\theta_1) 1}_{0, 4}$ is zero. 
	 
	Let us show that the groups $\mathcal{E}^{(\theta_1) \infty}_{2, 2}$ and $\mathcal{E}^{(\theta_1) \infty}_{3, 1}$ are free abelian. For a 3-cell $\sigma \in \B(x)$ we have $\Stab_{G}(\sigma) \subseteq \Stab_{\I}(\sigma)$. The group $\Stab_{\I}(\sigma)$ is either isomorphic to $\Z$, or trivial (see Section \ref{S4}). Therefore $\H_1(\Stab_{G}(\sigma), \Z)$ is a free abelian group. For a 2-cell $\sigma \in \B(x)$ we also have $\Stab_{G}(\sigma) \subseteq \Stab_{\I}(\sigma)$. If these groups are not trivial, Propositions \ref{H124} and \ref{H125} imply that $\Stab_{\I}(\sigma) \cong \Z \times F_\infty$.
	So we can apply Lemma \ref{homology}.

	Hence the groups $\mathcal{E}^{(\theta_1) \infty}_{2, 2}$ and $\mathcal{E}^{(\theta_1) \infty}_{3, 1}$ are free abelian. 
	Since $\H_4(G, \Z) \cong \Z$ it follows that they are zero. Consequently, we have an isomorphism
	$$
	\left\langle P_{\alpha_1 \cup \alpha_2 \cup \alpha'_2 \cup \alpha_3} \otimes \A(T_{\theta_1}, T_{\alpha_2} T_{\alpha'_2}^{-1}, T_{\theta_3}) - P_{\alpha_1 \cup \alpha_2 \cup \alpha_3 \cup \alpha'_3} \otimes \A(T_{\theta_2}, T_{\alpha_3} T_{\alpha'_3}^{-1}, T_{\theta_1}) \right\rangle \cong  
	$$
	$$
	\cong \left\langle s(U_2, U_3) \right\rangle = \H_4(G, \Z)
	$$
	Therefore this isomorphism maps 
	$$
	\left( P_{\alpha_1 \cup \alpha_2 \cup \alpha'_2 \cup \alpha_3} \otimes \A(T_{\theta_1}, T_{\alpha_2} T_{\alpha'_2}^{-1}, T_{\theta_3}) - P_{\alpha_1 \cup \alpha_2 \cup \alpha_3 \cup \alpha'_3} \otimes \A(T_{\theta_2}, T_{\alpha_3} T_{\alpha'_3}^{-1}, T_{\theta_1}) \right)
	$$
	to $\pm s(U_2, U_3)$.
	
	Consider the morphism $\mathcal{E}^{(\theta_1) *}_{*, *} \to E^{(\theta_1) *}_{*, *}$ induced by the inclusion $G \hookrightarrow \I^{\theta_1}$. The lemma follows by functoriality. 
\end{proof}

Without loss of generality can assume that in Lemma \ref{alt} we have the sign $'+'$.

\begin{lemma} \label{gen}
	The group $E_{0, 4}^1$ is generated by elements of the form $s(U_1, U_2, U_3)$, such that $\U = (U_1, U_2, U_3)$ is of type (a) w.r.t $x$.
\end{lemma}
\begin{proof}
	Recall that 
	$$E_{0, 4}^1 \cong \H_4(\I_{\alpha}, \Z)$$
	and 
	$$\Phi: \H_4(\I_{\alpha}, \Z) \cong \widetilde{E}_{1, 3}^2.$$
	
	Consider the preimages under the mapping $\widetilde{J}_{1, 3}^2$ of the basis elements of $\widetilde{E}^2_{1, 3}$ (see Proposition \ref{kernel2}) constructed in the proof of Proposition \ref{surj1}. 
	The arguments similar to the proof of Lemma \ref{alt} show that $\pm\Phi(s(U_2, U_3))$ equals
	$$P_{\beta_2 \cup \beta'_2} \otimes \A(T_{\theta_1}, T_{\beta_2} T_{\beta'_2}^{-1}, T_{\theta_3}),$$
	if $\U$ is of type (1) w.r.t $y$;
	$$P_{\beta_1 \cup \beta_2 \cup \beta'_2} \otimes \A(T_{\theta_1}, T_{\beta_2} T_{\beta'_2}^{-1}, T_{\theta_3}),$$
	if $\U$ is of type (2) w.r.t $y$;
	$$ \left( P_{\beta_2 \cup \beta'_2 \cup \beta_3} \otimes \A(T_{\theta_3}, T_{\beta_2} T_{\beta'_2}^{-1}, T_{\theta_1}) - P_{\beta_2 \cup \beta_3 \cup \beta'_3} \otimes \A(T_{\theta_1}, T_{\beta_3} T_{\beta'_3}^{-1}, T_{\theta_2}) \right) ,$$
	if $\U$ is of type (3) w.r.t $y$;
	$$ \left( P_{\beta_1 \cup \beta_2 \cup \beta'_2 \cup \beta_3} \otimes \A(T_{\theta_3}, T_{\beta_2} T_{\beta'_2}^{-1}, T_{\theta_1}) - P_{\beta_1 \cup \beta_2 \cup \beta_3 \cup \beta'_3} \otimes \A(T_{\theta_1}, T_{\beta_3} T_{\beta'_3}^{-1}, T_{\theta_2}) \right)  ,$$
	if $\U$ is of type (4) w.r.t $y$. Since $x = [\alpha] \in U_1$, in each of these cases $\U$ is of type (a) w.r.t $x$.
\end{proof}

\subsection{Proof of linear independence}

Denote by $S_{\U} \subseteq \H_4(\I, \Z)$ the subgroup by three homology classes
$$s(U_1, U_2, U_3), s(U_2, U_3, U_1), s(U_3, U_1, U_2) \in \H_4(\I, \Z).$$

\begin{lemma} \label{inds}
	(a)The inclusions $S_\U \hookrightarrow \H_4(\I, \Z)$ induce an injective homomorphism
	$$\bigoplus_{\U} S_{\U} \hookrightarrow \H_4(\I, \Z),$$
	where the sum is over all unordered symplectic splittings $\U$ of $\H$.
	
	(b)The homology classes 
	$$s(U_1, U_2, U_3), s(U_2, U_3, U_1) \in \H_4(\I, \Z),$$
	where $\U$ rans over the set of all unordered symplectic splittings of $\H$, are linearly independent.
\end{lemma}

\begin{proof}
	It suffices to prove that for a finite set of splittings $\{\U^1, \dots, \U^k \}$ the map
	\begin{equation*}
	\bigoplus_{1}^r S_{\U^k} \to \H_4(\I, \Z)
	\end{equation*}
	is injective.
	
	The following straightforward result is proved in \cite{Gaifullin}.
	\begin{prop}\cite[Lemma 4.5]{Gaifullin} \label{prop_dec}
		There is a homology class $x \in \H$ such that the splittings $\U^1, \dots, \U^k$ are of type (c) w.r.t $x$.
	\end{prop}

Take any $x \in \H$ satisfying the conditions of Proposition \ref{prop_dec}.
It suffices to prove that the map
\begin{equation} \label{fi}
\bigoplus_{1}^r S_{\U^k} \to \H_4(\I, \Z) / E_{0, 4}^1 \cong E_{1, 3}^2
\end{equation}
is injective.

By Lemma \ref{alt} the image of $S_{\U}$ ($\U \in \{\U^1, \dots, \U^k\}$) in $E_{1, 3}^2$ under the mapping (\ref{fi}) is contained in the linear span of the elements
$$P_{\alpha_1 \cup \alpha'_1 \cup \alpha_2 \cup \alpha_3} \otimes \A_{U_2, U_3} - P_{\alpha_1 \cup \alpha_2 \cup \alpha'_2 \cup \alpha_3} \otimes \A_{U_3, U_1}, P_{\alpha_1 \cup \alpha_2 \cup \alpha'_2 \cup \alpha_3} \otimes \A_{U_3, U_1} - P_{\alpha_1 \cup \alpha_2 \cup \alpha_3 \cup \alpha'_3} \otimes \A_{U_1, U_2}.$$
By Proposition \ref{kernel} these elements are linearly independent for different $\U$. This implies the result.
\end{proof}

\subsection{Relations between s-classes}

\begin{proof}[Proof of Theorem \ref{th2}.] Lemma \ref{inds} and Proposition \ref{relation1} imply that it suffices to prove that for any splitting $(U_1, U_2, U_3)$ we have
\begin{equation}\label{rel22}
s(U_1, U_2, U_3) + s(U_2, U_3, U_1) + s(U_3, U_1, U_2) = 0.
\end{equation}
By Lemma \ref{alt} the images of $s(U_1, U_2, U_3), s(U_2, U_3, U_1), s(U_3, U_1, U_2)$ in the group \\ $\H_4(\I, \Z) / E^1_{0, 4}$ are
$$\left( P_{\alpha_1 \cup \alpha_2 \cup \alpha'_2 \cup \alpha_3} \otimes \A(T_{\theta_1}, T_{\alpha_2} T_{\alpha'_2}^{-1}, T_{\theta_3}) - P_{\alpha_1 \cup \alpha_2 \cup \alpha_3 \cup \alpha'_3} \otimes \A(T_{\theta_2}, T_{\alpha_3} T_{\alpha'_3}^{-1}, T_{\theta_1}) \right),$$

$$\left( P_{\alpha_1 \cup \alpha_2 \cup \alpha_3 \cup \alpha'_3} \otimes \A(T_{\theta_2}, T_{\alpha_3} T_{\alpha'_3}^{-1}, T_{\theta_1}) - P_{\alpha_1 \cup \alpha'_1 \cup \alpha_2 \cup \alpha_3} \otimes \A(T_{\theta_3}, T_{\alpha_1} T_{\alpha'_1}^{-1}, T_{\theta_2}) \right),$$

$$\left( P_{\alpha_1 \cup \alpha'_1 \cup \alpha_2 \cup \alpha_3} \otimes \A(T_{\theta_3}, T_{\alpha_1} T_{\alpha'_1}^{-1}, T_{\theta_2}) - P_{\alpha_1 \cup \alpha_2 \cup \alpha'_2 \cup \alpha_3} \otimes \A(T_{\theta_1}, T_{\alpha_2} T_{\alpha'_2}^{-1}, T_{\theta_3}) \right),$$
respectively.

Therefore the image of 
$$s(U_1, U_2, U_3) + s(U_2, U_3, U_1) + s(U_3, U_1, U_2)$$
in $\H_4(\I, \Z) / E^1_{0, 4}$ is zero. Hence we have
$$s(U_1, U_2, U_3) + s(U_2, U_3, U_1) + s(U_3, U_1, U_2) \in E^1_{0, 4}.$$

By Lemma \ref{gen} the group $E^1_{0, 4}$ is generated by elements of subgroups $S_{\U'}$, where $\U'$ is of type (a) w.r.t $x$. Since $\U$ is of type (c) w.r.t $x$, Lemma \ref{inds} (a) implies (\ref{rel22}).
\end{proof}

\end{document}